%% file: HHS_Artin_arxiv_v1.tex
\documentclass[11pt, oneside]{amsart}

\usepackage{tabularx, hyperref,ifthen}
\usepackage{amssymb} \usepackage{amsfonts} \usepackage{amsmath}
\usepackage{amsthm} \usepackage{epsfig, subfig}
\usepackage{ amscd, amsxtra, latexsym,pb-diagram,pb-xy, mathtools}
\usepackage[margin=1.2in]{geometry}
\usepackage{epsfig, graphicx, psfrag}
\usepackage{graphicx}
\usepackage[all]{xy}
\usepackage{caption}
\usepackage{enumerate}
\usepackage[dvipsnames]{xcolor}
\usepackage{epigraph}
\usepackage{overpic}

\newtheorem{lemma}{Lemma}[section]
\newtheorem{thm}[lemma]{Theorem}
\newtheorem{prop}[lemma]{Proposition}
\newtheorem{cor}[lemma]{Corollary}

\newtheorem*{questintro}{Question}
\newtheorem{claim}[lemma]{Claim}
\newtheorem{theoremalph}{Theorem}
\newtheorem{coralph}[theoremalph]{Corollary}

\theoremstyle{definition}
\newtheorem{defn}[lemma]{Definition}

\newtheorem{rem}[lemma]{Remark}
\newtheorem{conv}[lemma]{Convention}
\theoremstyle{definition}
\newtheorem{notation}[lemma]{Notation}

\definecolor{darkgreen}{cmyk}{1,0,1,.2}

\newcommand{\showcomments}{yes}

\newsavebox{\commentbox}

\newcommand{\Z} {\ensuremath {\mathbb{Z}}}

\newcommand{\calG} {\ensuremath {\mathcal{G}}}

\newcommand{\calB} {\ensuremath {\mathcal{B}}}
\newcommand{\calH} {\ensuremath {\mathcal{H}}}
\newcommand{\calX} {\ensuremath {\mathcal{X}}}

\newcommand{\stabilizer}{\mathrm{Stab}}

\newcommand{\nest}{\sqsubseteq}
\usepackage{mathabx}

\newcommand{\orth}{\bot}

\newcommand{\cayley}{\mathrm{Cay}}

\newcommand{\link}{\mathrm{Lk}}
\newcommand{\Star}{\mathrm{St}}
\newcommand{\sat}{\mathrm{Sat}}
\newcommand{\naturals}{\mathbb N}
\newcommand{\dist}{\mathrm{d}}
\newcommand{\diam}{\mathrm{diam}}

\newcommand{\cuco}[1]{{\mathcal #1}}

\newcommand{\tsh}[1]{\left\{\kern-.7ex\left\{#1\right\}\kern-.7ex\right\}}

\newcommand{\integers}{\mathbb Z}
\newcommand{\reals}{\mathbb R}

\makeatletter
\@tfor\next:=abcdefghijklmnopqrstuvwxyzABCDEFGHIJKLMNOPQRSTUVWXYZ\do{%
	\def\command@factory#1{%
		\expandafter\def\csname cal#1\endcsname{\mathcal{#1}}
		\expandafter\def\csname frak#1\endcsname{\mathfrak{#1}}
		\expandafter\def\csname scr#1\endcsname{\mathscr{#1}}
		\expandafter\def\csname bb#1\endcsname{\mathbb{#1}}
		\expandafter\def\csname rm#1\endcsname{\mathrm{#1}}
		\expandafter\def\csname bf#1\endcsname{\mathbf{#1}}
	}
	\expandafter\command@factory\next
}
\makeatother

\begin{document}
	
	\title[Extra-large type Artin groups are hierarchically hyperbolic]{Extra-large type Artin groups are hierarchically hyperbolic}
	\author[M Hagen]{Mark Hagen}
	\address{School of Mathematics, University of Bristol, Bristol, UK}
	\email{markfhagen@posteo.net}
	\author[A Martin]{Alexandre Martin}
	\address{Department of Mathematics, Heriot-Watt University, Edinburgh, UK}
	\email{alexandre.martin@hw.ac.uk}
	\author[A Sisto]{Alessandro Sisto}
	\address{Department of Mathematics, Heriot-Watt University, Edinburgh, UK}
	\email{a.sisto@hw.ac.uk}
	\maketitle

	\begin{abstract}
	We show that   Artin groups of extra-large type, and more generally Artin groups of large and hyperbolic type, are hierarchically hyperbolic. This implies in particular that these groups have finite 
asymptotic dimension and uniform exponential growth.
		
	We prove these results by using a combinatorial approach to hierarchical hyperbolicity, via the action of these groups on a new complex that 
is quasi-isometric both to the coned-off Deligne complex introduced by Martin--Przytycki and to a generalisation due to Morris-Wright of the 
graph of irreducible parabolic subgroups of finite type introduced by Cumplido--Gebhardt--Gonz\'alez-Meneses--Wiest.
	\end{abstract}

	\tableofcontents
	
\section*{Introduction}

\subsection*{Hyperbolic features of Artin groups} The geometry of Artin groups has seen an explosion of results in the last decade. While Artin groups remain in general much less understood than their Coxeter relatives, a driving theme behind current research has been to show that these groups are as well-behaved as Coxeter groups, and this has indeed been verified for several classes of Artin groups. On the geometric side,  
a  popular theme has been to understand the ``hyperbolic features'' of groups, and Artin groups are conjectured to display such hyperbolic features. This rather vague notion comes in many flavours, a 
first one being the notion of \textit{acylindrically hyperbolic} group. Loosely speaking, such groups can be described as having ``hyperbolic directions'' (see \cite{Osin} for the precise definition 
and its many consequences). It is conjectured that for an irreducible Artin group $A_\Gamma$,  its central quotient $A_\Gamma/Z(A_\Gamma)$  is acylindrically hyperbolic. 

This question has been answered positively for most standard classes of Artin groups, such as right-angled Artin groups \cite{CS}, Artin groups of finite type \cite{CW:acyl}, Artin groups of Euclidean type \cite{Calvez}, Artin groups whose underlying presentation graph is not a join \cite{CharneyMW}, and two-dimensional Artin groups \cite{Vaskou}. 

While acylindrical hyperbolicity guarantees the existence of some hyperbolic directions, it does not provide much control on the overall geometry of the group. For instance, the free product $A*B$, where $A, B$ are the worst infinite groups you can think of, is acylindrically hyperbolic.

A notion of non-positive curvature that provides a much stronger control over the coarse geometry of the group is the notion of \textit{hierarchically hyperbolic group} (or HHG). This notion was
introduced by Behrstock--Hagen--Sisto \cite{HHS_I,HHS_II} and inspired by work of Masur-Minsky \cite{MMI,MMII} as a framework that unifies and generalises the geometry of mapping class groups of 
hyperbolic surfaces and that of (all known) cocompactly cubulated groups~\cite{HagenSusse}. The idea of hierarchical hyperbolicity is to describe the coarse geometry of a group/space 
using a ``coordinate system'' where the coordinates take values in various hyperbolic spaces.

Since this notion gives much better control than acylindrical hyperbolicity, it implies many results expected of non-positively curved groups that are not true for more general notions of 
non-positive curvature; for example, it gives quadratic isoperimetric inequality~\cite{Bowditch, HHS_II}, solubility of the word and conjugacy problem~\cite{HHS_II,HHP}, the Tits 
alternative~\cite{HHS_IV,HHS:corr}, finite asymptotic dimension~\cite{HHS_III}, semi-hyperbolicity~\cite{HHP, DMS}, etc. In particular, since both braid groups and right-angled Artin groups 
belong to this family \cite{HHS_I,HHS_II}, the following is a natural question that was for instance raised by Calvez--Wiest \cite{CW}: 

\begin{questintro}
	Which Artin groups are hierarchically hyperbolic?
\end{questintro}

Building new examples of hierarchically hyperbolic groups is nontrivial; there are combination theorems~\cite{HHS_II,BerlaiRobbio}, and theorems about persistence of hierarchical hyperbolicity under 
various quotients~\cite{HHS_III,CHHS}, but the examples coming from ``nature'' are dominated by compact special groups, mapping class groups, and fundamental groups of certain $3$--manifolds.  The 
preceding question has only been previously answered positively for right-angled Artin groups and braid groups. 

In this article, we add a large class of Artin groups to the ``naturally occurring'' hierarchically hyperbolic groups:

\begin{theoremalph}\label{thmintro:main}
	Artin groups of large and hyperbolic type are hierarchically hyperbolic. 
\end{theoremalph}

Note that Artin groups of large and hyperbolic type contain in particular all Artin groups of extra-large type (see Section~\ref{subsec:artin_background} for precise definitions of these 
classes). We refer to Theorem~\ref{thm:artin_HHG} for a description of the HH structure.

The following results are new for Artin groups of large and hyperbolic type, and follow from the cited results about hierarchically hyperbolic groups:

\begin{coralph}\label{corintro}
Let $A_\Gamma$ be an Artin group of large and hyperbolic type, with $\Gamma$ connected and not a single vertex. Then:
\begin{enumerate}
 \item \label{item:asdim} $A_\Gamma$ has finite asymptotic dimension~\cite[Thm. A]{HHS_III},
 \item \label{item:unif_exp} $A_\Gamma$ has uniform exponential growth~\cite[Cor. 1.3]{ANS},
 \item \label{item:stable} A finitely generated subgroup of $A_\Gamma$ is stable if and only if the orbit maps to the coned-off Deligne complex are quasi-isometric embeddings.~\cite[Thm. B]{ABD}
\end{enumerate}
\end{coralph}

Items \ref{item:unif_exp} and \ref{item:stable} require a short argument to link the cited results to Theorem~\ref{thmintro:main}; this appears at the end of Section~\ref{subsec:assembling}. The coned-off Deligne complex - which is one of our most important tools - was defined and extensively studied in \cite{MP}. As a consequence of our construction, this complex is (equivariantly quasi-isometric to) the maximal hyperbolic space in our HHG structure, see Theorem~\ref{thm:artin_HHG}. 

Regarding asymptotic dimension, it follows from Theorem \ref{thm:artin_HHG} and \cite[Theorem 5.2, Corollary 3.3]{HHS_III} that the asymptotic dimension of the coned-off Deligne complex $\hat D_\Gamma$ is finite (with no explicit bound) and that the asymptotic dimension of $A_\Gamma$ is at most $asdim(\hat D_\Gamma)+3$. In view of this, it is natural to ask:

\begin{questintro}
 Given $\Gamma$, what is (a good bound on) the asymptotic dimension of the coned-off Deligne complex?
\end{questintro}

Note that our main Theorem  also allows us to recover from a unified perspective several other known results for these groups, such as the above-mentioned Tits alternative (see~\cite[Corollary 
B]{MP}; this was proved in the extra-large case in~\cite{OP}), the solubility of the conjugacy problem (also proved in \cite[Corollary 1.3]{OsajdaHuang} for any two-dimensional Artin group), and 
control over their quasi-flats (see \cite[Theorem 1.1]{OsajdaHuang2} for the statement about two-dimensional Artin groups, and the closely-related~\cite[Theorem A]{HHS_V} about HHGs). Moreover, we recover the fact that these groups are semi-hyperbolic, as they are coarsely Helly by~\cite[Thm. A, Cor. F]{HHP}, see also \cite{DMS}. (The semi-hyperbolicity also follows from the fact that these groups are  systolic \cite{OsajdaHuang3}, hence biautomatic \cite{januszkiewicz2006simplicial, swikatkowski2006regular}, hence semi-hyperbolic \cite{alonso1995semihyperbolic}.)

\subsection*{A curve complex for Artin groups} There are earlier results in the direction of hierarchical hyperbolicity for Artin groups, see e.g.~\cite{CW} for a detailed discussion.  The first step in such a result is to introduce an analogue 
of the curve graph, to play the role of the ``maximal'' hyperbolic space in the HHS structure, i.e. a space quasi-isometric to the space obtained from the group by coning off the standard product 
regions.  There are various significant results about actions of Artin groups on hyperbolic complexes, notably those of Calvez-Wiest for spherical type Artin groups~\cite{CW:acyl}, Calvez in the 
Euclidean-type case~\cite{Calvez}, and Martin--Przytycki in the two-dimensional case \cite{MP}.  

A candidate for such a curve complex for Artin groups of finite type has been proposed recently by Cumplido-Gebhardt--Gonzales-Meneses--Wiest~\cite{CumplidoEtAl}, and generalised to all Artin groups by 
Morris-Wright~\cite{MorrisWright}  who studied this 
complex extensively for Artin groups of type FC. This is the notion of \textit{graph of irreducible proper parabolic subgroups of finite type} of an Artin group. 

 It is conjectured that this complex is an infinite-diameter hyperbolic space (except in degenerate cases) for  Artin 
groups of finite type~\cite[Conjecture 2.4]{CumplidoEtAl} and type FC~\cite[Conjecture 5.6]{MorrisWright}.  We prove the analogous result for Artin groups of large and hyperbolic type as a corollary of Proposition~\ref{prop:irred_parabolic} and Theorem~\ref{thmintro:main}:

\begin{coralph}
	Let $A_\Gamma$ an Artin group of large and hyperbolic type on at least three generators. Then $A_\Gamma$ admits an HHG structure whose maximal hyperbolic space is the graph of irreducible proper parabolic 
	subgroups of finite type. 
\end{coralph}

Note that the action of $A_\Gamma$ on this 
graph is acylindrical. This can be seen by combining the above corollary with \cite[Theorem~K]{HHS_I}, or by combining Proposition~\ref{prop:irred_parabolic} with \cite[Theorem A]{MP}. This corollary gives hope that this complex not only is hyperbolic for larger classes of Artin groups but also is the maximal hyperbolic space in an HHG structure.

\subsection*{Combinatorial approach to HHS, and strategy of proof}
We will not strictly need the full definition of hierarchical hyperbolicity but, roughly, a hierarchically hyperbolic space (HHS) is a space $\cuco X$ that comes with a family of hyperbolic spaces and 
projections from $\cuco X$ to the various hyperbolic spaces satisfying various conditions \cite[Definition 1.1]{HHS_II}. Additionally, a group is hierarchically hyperbolic if it acts geometrically on 
an HHS preserving the HHS structure, in a way expressed most simply in \cite[p.4]{PetytSpriano}.

We will not need the full definition because we will verify a more combinatorial criterion from \cite{CHHS} for a group $G$ to be hierarchically hyperbolic. We state this as Theorem \ref{thm:main 
criterion}, but for this discussion it suffices to know that criterion involves a hyperbolic simplicial complex $X$ on which $G$ acts, and maximal simplices of $X$ are in bijection with the vertices 
of a graph quasi-isometric to $G$. The links of simplices in $X$ give the hyperbolic spaces in the HHS structure, so the fine geometry of $X$ is relevant for this purpose.

We now explain how to come up with a candidate $X$ for a large-type  Artin group $A_\Gamma$ of hyperbolic type. The strategy makes sense more generally and potentially could be applied to other 
classes of groups as well.
First, we need a little more discussion on hierarchical hyperbolicity. An HHS contains a family of so-called \emph{standard product regions}, which are HHS themselves and (coarsely) split as a 
product of HHS. Moreover, the standard product regions are ``arranged hyperbolically'' meaning that coning them off yields a hyperbolic space. (This is the ``hierarchical structure'' that gives the 
name.) In the context of the combinatorial criterion, said coned-off space is quasi-isometric to the hyperbolic simplicial complex $X$.

In view of this, a natural candidate for an action of a given group on a simplicial complex to which Theorem 
\ref{thm:main criterion} can be applied is constructed as follows. 

\medskip

\noindent \textit{Step 1:} Consider the family of subgroups that are maximal virtual 
products (these should give the standard product regions if the group is hierarchically hyperbolic). For our Artin groups, the tools to identify these subgroups come from the acylindricity of the action 
of $A_\Gamma$ on its coned-off Deligne complex and  CAT($-1$) geometry \cite{MP}, see Section \ref{subsec:std_tree}.

\par\medskip

\noindent \textit{Step 2:} Isolate the subgroups arising as minimal infinite intersections between such virtual product subgroups. Coarse intersections of standard product regions in an HHS are again coarse products of 
HHS, which are ``smaller'' than the original product subgroups. This should identify the simplest sub-HHS, if the group is to be hierarchically hyperbolic. For our Artin groups, here we have the 
cyclic subgroups conjugate to either the subgroup generated by a standard generators or centres of \emph{dihedral parabolic} subgroups. See Section~\ref{subsec:commutation_graph} for more discussion.

\par\medskip

\noindent \textit{Step 3:}  To construct a hyperbolic space that the group acts on, it is natural to consider the graph encoding the intersections of candidates for the standard product regions. This graph is quasi-isometric to the \emph{commutation graph} of the minimal subgroups above, see Section \ref{subsec:commutation_graph}. This graph has a vertex for each such minimal subgroup and we put 
an edge when two such subgroups commute. In the case of right-angled Artin groups, we recover the extension graph, which is indeed hyperbolic \cite{KimKoberda}.  In our case, there is a natural map between the commutation graph and the coned-off Deligne complex. This map is a quasi-isometry with very nice local properties, and this is crucial for our arguments in Section~\ref{sec:augment}.

\par\medskip

\noindent \textit{Step 4:}  The commutation graph as above is expected to have the right coarse geometry, but not the right local geometry, because  stabilisers of maximal simplices are 
usually infinite. We remedy this by a 
\emph{blow-up} construction, where we replace each vertex of the commutation graph by a quasi-isometric copy of the corresponding subgroup, preserving the $A_\Gamma$--action. 
In our situation, the vertices of the commutation graph need to be blown up to quasi-lines.  This creates a somewhat delicate situation where, for a certain cyclic subgroup $H$, we need to construct 
an action of $N(H)$ on a quasiline with certain properties.  
Remarkably, this difficulty is circumvented using quasimorphisms (which are related to actions on quasilines by~\cite{ABO}), in a similar way as in \cite{GraphManifolds}.

\medskip

Once we have our hyperbolic complex $X$, we can build our quasi-isometry model $W$ of $A_\Gamma$.  As mentioned above, the vertex set of $W$ is the set of maximal simplices of $X$.  The adjacency relation is defined in 
such a way as to guarantee that orbit maps $A_\Gamma\to W$ are quasi-isometries (Section~\ref{subsec:augmented}).  The remaining work is to verify the technical conditions of Theorem~\ref{thm:main 
criterion}, and here we once again rely on the CAT($-1$) geometry of the coned-off Deligne complex.

In the end, we get an HHS structure on $A_\Gamma$ where the maximal hyperbolic space --- the HHS analogue of the curve graph in the mapping class group setting --- is quasi-isometric to $X$ and hence 
to the commutation graph.

As a final remark, our results do not cover the more general case of 2-dimensional Artin groups of hyperbolic type (that is, the case where edges with label 2 are allowed). In this case, the combinatorics of the commutation graph and associated spaces are more complex, with several statements in Section \ref{sec:commute} and beyond needing more nuance since more cases have to be considered. However, we still believe that with more sophisticated arguments one can deal with the more general case using the same combinatorial HHS approach.

\subsection*{Outline of the paper}  Section~\ref{sec:CHHS} contains the definitions and results from~\cite{CHHS} that we will need. Section~\ref{sec:artin_background} contains background on Artin 
groups, and in particular the subclass of Artin groups considered in this paper, along with the coned-off Deligne complex.  Section~\ref{sec:commute} is about the commutation graph, and also contains 
the discussion relating the commutation graph and the graph of irreducible parabolics.  Section~\ref{sec:blowup} contains the construction of the simplicial complex $X$ --- the blow-up of the 
commutation graph --- along with some purely combinatorial facts about $X$, and the relationship between maximal simplices in $X$ and coarse points in $A_\Gamma$.  In Section~\ref{sec:augment}, we 
define the graph of maximal simplices of $X$, prove that it is quasi-isometric to $A_\Gamma$, and study its combinatorial structure. Finally, in Section~\ref{sec:augment_geom}, we verify the remaining hypotheses of Theorem~\ref{thm:main criterion}.  The final subsection of 
Section~\ref{sec:augment_geom} assembles the pieces into a proof of Theorem~\ref{thmintro:main} and Corollary~\ref{corintro}. 

\subsection*{Acknowledgments}
We are grateful for the support of the International Centre for Mathematical Sciences, Edinburgh, which hosted the authors for a week in February 2020 during a Research in Groups where much of the work on this project 
was done.  Hagen was partially supported by EPSRC New Investigator Award EP/R042187/1. Martin was partially supported by the EPSRC New Investigator Award EP/S010963/1.

\section{A combinatorial criterion for hierarchical hyperbolicity} \label{sec:CHHS}

In this section, we recall the main combinatorial criterion introduced in \cite{CHHS} to show that a group is hierarchically 
hyperbolic by means of a map to a hyperbolic simplicial complex.  We will not require these notions until Section~\ref{sec:blowup}, but we introduce them now to motivate the constructions in earlier 
sections.  We refer the reader to~\cite[Section 1.2,1.5]{CHHS} for a more informal discussion of the various hypotheses in the definition of a combinatorial HHS.

\begin{defn}[$X$--graph, augmented complex]\label{def:aug_X}
	Let $X$ be a flag simplicial complex. An \textbf{$X$-graph} is a graph $W$ whose vertex set is the set 
	of maximal simplices of $X$. We say that two maximal simplices $\Delta, \Delta'$ of $X$ are $W$\textbf{-adjacent} if the 
	corresponding vertices of $W$ are adjacent in $W$.
	
	We denote by $X^{+W}$ the complex  obtained from the $1$-skeleton of $X$ by adding edges between vertices that belong 
	to $W$-adjacent maximal simplices. For a subcomplex $X_0$ of $X$, we denote by $X_0^{+W} $  the full subcomplex of $X^{+W}$ 
	induced by $X_0$. We refer to these various objects $X^{+W}, X_0^{+W}$, etc. as being \textbf{augmented complexes}.
\end{defn}

\begin{defn}[Link, star, saturation of a simplex]
	Let $X$ be a flag simplicial complex, and let $\Delta$ be a simplex of $X$. 
	
	The \textbf{star} of $\Delta$, denoted $\Star_X(\Delta)$, is the union of all the simplices of $X$  containing 
	$\Delta$. 
	
	The \textbf{link} of $\Delta$, denoted $\link_X(\Delta)$, is the full subcomplex of $\Star_X(\Delta)$  induced by 
	$\Star_X(\Delta) - \Delta$.
	
	The \textbf{saturation} of $\Delta$ is $$\sat_X(\Delta)=\bigcup_{\{\Sigma:\link_X(\Delta)=\link_X(\Sigma)\}}\Sigma^{(0)}.$$
\end{defn}

The following is Definition~1.8 in~\cite{CHHS}:

\begin{defn}[Combinatorial HHS]\label{defn:combinatorial_HHS}
	A \emph{combinatorial HHS} is a pair $(X,W)$, where $X$ is a simplicial complex and $W$ is an $X$--graph, such that all of 
	the following hold for some $\delta<\infty,n\in\naturals$:
	\begin{enumerate}[(I)]
		\item If $\Delta_0,\ldots,\Delta_m$ are simplices of $X$ and 
		$\link_X(\Delta_i)\subsetneq\link_X(\Delta_{i+1})$ for $0\leq i\leq m-1$, then $m\leq n$.  This condition will be called 
		\emph{finite complexity}.\label{item:finite_comp}
		\item Let $\Delta$ be a non-maximal simplex of $X$.  Let 
		$\mathcal C(\Delta)=\link_X(\Delta)^{+W}$ and let $Y_\Delta=(X^{(0)}-\sat(\Delta))^{+W}$.  Then $\mathcal C(\Delta)$ is 
		$\delta$--hyperbolic and the inclusion $\mathcal C(\Delta)\hookrightarrow Y_\Delta$ is a $(\delta,\delta)$--quasi-isometric 
		embedding.  We will call this condition \emph{hyperbolic links}.\label{item:hyperbolic_links}
		
		\item Let $\Delta$ be a non-maximal simplex of $X$ and let $v,w\in\link_X(\Delta)$ be distinct non-adjacent vertices.  
		Suppose that $v,w$ are contained in $W$--adjacent maximal simplices of $X$.  Then there exist maximal simplices 
		$\Sigma_v,\Sigma_w$ of $\link_X(\Delta)$, respectively containing $v,w$, such that $\Sigma_v\star\Delta$ and 
		$\Sigma_w\star\Delta$ are $W$--adjacent.  We will call this condition \emph{fullness of links}.\label{item:fullness}
		
		\item Let $\Sigma,\Delta$ be non-maximal simplices of $X$ such that there exists a non-maximal simplex $\Gamma$ with 
		$\link_X(\Gamma)\subseteq \link_X(\Sigma)\cap\link_X(\Delta)$ and $\diam(\mathcal C(\Gamma))>\delta$.  Then there exists a 
		non-maximal simplex $\Pi\subset\link_X(\Sigma)$ such that $\link_X(\Sigma\star\Pi)\subseteq\link_X(\Delta)$ and all $\Gamma$ 
		as above satisfy $\link_X(\Gamma)\subset\link_X(\Sigma\star\Pi)$.  We call this condition the \emph{intersection condition}.\label{item:intersection_condition}

	\end{enumerate}
	
\end{defn}

The following criterion is immediate from \cite[Theorem 1.18, Remark 1.19]{CHHS} and the fact that hierarchical 
hyperbolicity is a 
quasi-isometry invariant property~\cite[Proposition 1.10]{HHS_II}:

\begin{thm}\label{thm:main criterion}
	Let $(X,W)$ be a combinatorial HHS.  Then any quasigeodesic space quasi-isometric to $W$ is a hierarchically 
	hyperbolic space.  Moreover, suppose that the group $G$ acts by simplicial automorphisms on $X$, and that the resulting 
	$G$--action on the set of maximal simplices of $X$ extends to a 
	proper cobounded action of $G$ on $W$.  Suppose moreover that $X$ contains finitely many $G$--orbits of subcomplexes of the 
	form $\link_X(\Delta)$, for $\Delta$ a simplex.  Then $G$ is a hierarchically hyperbolic group.
\end{thm}

(In the statement, the notion of properness used is sometimes called metric properness, and what we mean is that given a ball in $X$ there are only finitely many elements of $G$ that do not map the ball to a disjoint ball.)

Note that in addition to the properties of combinatorial HHS, Theorem~\ref{thm:main criterion} requires another condition, namely that there are finitely many orbits of links of simplices. For readers familiar with HHS terminology, we mention that this is so that the action of $A_\Gamma$ on the index set of the eventual HHG structure 
is cofinite, as the elements of the index set correspond to the links of the non-maximal simplices.

We will apply the above theorem to an Artin group of large and hyperbolic type $A_\Gamma$ by explicitly constructing $(X,W)$ with $W$ 
quasi-isometric to $A_\Gamma$, and then verifying each of the properties from Definition~\ref{defn:combinatorial_HHS}.

\section{Background on Artin groups and Deligne complexes}\label{sec:artin_background}

\subsection{Artin groups}\label{subsec:artin_background}

A \textbf{presentation graph} is a finite simplicial graph $\Gamma$ such that every edge between vertices $a, b \in V(\Gamma)$ is 
labelled by an integer $m_{ab} \geq 2$. The \textbf{Artin group} associated to $\Gamma$ is the group $A_{\Gamma}$ given by 
the following presentation:
$$A_\Gamma \coloneqq \langle a \in V(\Gamma ) ~|~ \underbrace{aba\cdots}_{m_{ab}} = \underbrace{bab\cdots}_{m_{ab}}  ~~ 
\mbox{ whenever } a,b \mbox{ are connected by an edge of }\Gamma \rangle.$$

An Artin group is of \textbf{large type} if all coefficients $m_{ab}$ are at least $3$, and of \textbf{extra-large type} if they all are at least $4$. An Artin group on two generators $a,b$ with $m_{ab}<\infty$ is a \textbf{dihedral} Artin group.

Given an Artin group $A_\Gamma$, the \textbf{associated Coxeter group}  $W_\Gamma$ is obtained by further requiring that each generator $a \in V(\Gamma)$ 
satisfies the relation $a^2 =1$. An Artin group is said to be of \textbf{hyperbolic type} if the associated Coxeter group is hyperbolic, and of \textbf{finite type} if the associated 
Coxeter group is finite.

For a (possibly empty) full subgraph $\Gamma'\subset \Gamma$, the subgroup of $A_{\Gamma}$ generated by the vertices of $\Gamma'$ 
is called a \textbf{standard parabolic subgroup}. Such a subgroup is isomorphic to the Artin group $A_{\Gamma'}$ by a result 
of Van der Lek \cite{vdL}, and moreover we have $A_{\Gamma_1} \cap A_{\Gamma_2} = A_{\Gamma_1\cap\Gamma_2}$ for full subgraphs $\Gamma_1, \Gamma_2$ of $\Gamma$. Conjugates of standard parabolic subgroups are called \textbf{parabolic subgroups}.

\subsection{Structure of dihedral Artin groups}\label{subsec: toolbox}

Since dihedral Artin groups appear as stabilisers of vertices of dihedral type in the modified Deligne complex and its cone-off, we 
mention some structural results that will be needed in this article.

A dihedral Artin group on two standard generators $a, b$ will be denoted $A_{ab}$ for simplicity, even though the group depends on the coefficient $m_{ab}$. Dihedral Artin groups come into two types: If $m_{ab}=2$, the group is a copy of $\Z^2$. The rest of this subsection focuses on the structure of dihedral Artin groups with $m_{ab}\geq 3$. We start by recalling the following definition:

\begin{defn} For a dihedral Artin group with $m_{ab}\geq 3$,  the \textbf{Garside element} $\Delta_{ab}\in A_{ab}$ is defined as follows:
$$\Delta_{ab} \coloneqq \underbrace{aba\cdots}_{m_{ab}}= \underbrace{bab\cdots}_{m_{ab}}.$$
\end{defn}

\begin{lemma}[{\cite{BS}}]\label{lem:centre dihedral}
	 The centre of a dihedral Artin group $A_{ab}$ with $m_{ab}\geq 3$ is infinite cyclic and generated by the element $$z_{ab} 
	\coloneqq\Delta_{ab}~\mbox{ if } ~ m_{ab} \mbox{ is even, and } ~ z_{ab} \coloneqq \Delta_{ab}^2~~\mbox{ if } ~ m_{ab} \mbox{ 
		is odd}.$$
\end{lemma}

\begin{lemma}\label{lem:virtually splits}
	Let $A_{ab}$ be a dihedral Artin group with $m_{ab}\geq 3$. The central quotient $A_{ab}/\langle z_{ab}\rangle $ is virtually a finitely generated non-trivial free group. In particular,   $A_{ab}$ contains a finite-index subgroup that splits as a direct product of the form $\langle 
	z_{ab} \rangle \times K$, where $K$ is a finitely generated free subgroup of $A_{ab}$. 
	\label{lem:dihedral_split}
\end{lemma}

This virtual splitting is well-known to experts, see for instance \cite{Crisp}. We give here a 
geometric proof of this result that uses objects that will be needed in Section~\ref{sec:augment}.

\begin{defn}[Atoms, left-weighted form]
	An \textbf{atom} of $A_{ab}$ is a strict subword of $\underbrace{aba\cdots}_{m_{ab}}$ or 
	$\underbrace{bab\cdots}_{m_{ab}}$, that is, an alternating product of $a$ and $b$ with strictly fewer than $m_{ab}$ letters. 
	We denote by $M$ the set of all atoms of $A_{ab}$. A product of the form $m_1\cdots m_k$, with each $m_i \in M$, is said to 
	be \textbf{left-weighted} if for each $i$ the last letter of $m_i$ coincides with the first letter of $m_{i+1}$.  
\end{defn}

It follows from the existence and uniqueness of Garside normal forms (see for instance\cite{MairesseMatheus}) that elements of  the quotient 
$A_{ab}/\langle \Delta_{ab}\rangle $ are in bijection with left-weighted elements of the free monoid $M^\bullet$ on $M$, 
where  $\langle \Delta_{ab}\rangle $ acts on $A_{ab}$ by right multiplication. 

\begin{defn}\label{def:quasi_tree_simplices}
	We denote by $\calT_{ab}$ the full subgraph of the Cayley graph $ \mathrm{Cayley}(M^\bullet, M)$ spanned by 
	left-weighted elements. 
	The action of $A_{ab}$ on $A_{ab}/\langle \Delta_{ab}\rangle  $ by left multiplication induces an action of $A_{ab}$  
	on $\calT_{ab}$ (seen as an unlabelled graph).
\end{defn}

Since $M^\bullet$ is a free monoid, the graph $\calT_{ab}$ is a quasi-tree, as already explained in \cite{BestvinaArtin}. More precisely, the flag completion of $\calT_{ab}$ has a structure of tree of simplices of dimension $m_{ab}-1$ 
glued along vertices (see Figure \ref{quasitree}), where the simplices are either in the $A_{ab}$-orbit of the simplex 
spanned by $$e, a, ab, \ldots, \underbrace{aba\cdots}_{m_{ab}-1},$$ or in the $A_{ab}$-orbit of the simplex spanned by $$e, 
b, ba, \ldots, \underbrace{bab\cdots}_{m_{ab}-1}.$$

\begin{figure}[h]
	\begin{center}
		\scalebox{0.75}{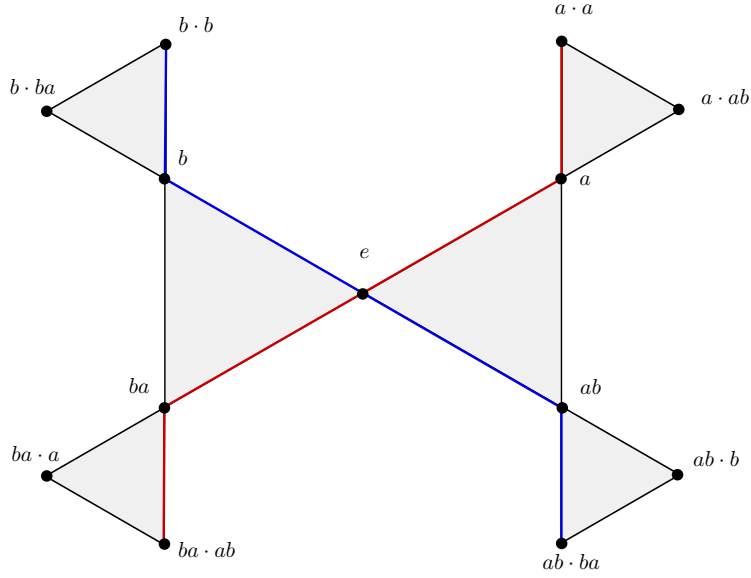}
		\caption{A portion of the quasi-tree $\calT_{ab}$ for $m_{ab}=3$. In red, the full subgraph spanned by the 
			$\langle a \rangle$-orbit of $e$. In blue, the full subgraph spanned by the $\langle b \rangle$-orbit of $e$.}
		\label{quasitree}
	\end{center}
\end{figure}

\begin{proof}[Proof of Lemma \ref{lem:virtually splits}] The group $A_{ab}$ acts by left multiplication on $\calT_{ab}$. 
	Since the element $z_{ab}$ is central and is a power of $\Delta_{ab}$, it follows that $\langle z_{ab}\rangle$ acts trivially 
	on that graph, hence the quotient $A_{ab}/\langle z_{ab} \rangle$ acts by left multiplication on $\calT_{ab}$. The action is 
	cocompact and proper, thus
	$A_{ab}/\langle z_{ab} \rangle$ is virtually free, and hence $A_{ab}$ is virtually a direct product of the form $\langle z_{ab} 
	\rangle \times K$, where $K$ is a finitely-generated free group. Note that $a$ and $b$ define elements of infinite order in 
	$A_{ab}/\langle z_{ab}\rangle$ since their orbits in $\calT_{ab}$ span embedded lines (see Figure \ref{quasitree}).
\end{proof}

 We recall that the \textbf{syllabic length} of an element $g \in A_{ab}$ is the smallest non-negative integer $n$ such that $g$ can be written as a product of the form $g=x_1^{k_1}\cdots x_n^{k_n}$ with $k_i \in \Z$ and $x_i \in \{a,b\}$ for all $1\leq i \leq n$.

\begin{lemma}[{\cite[Proposition 4.6]{Vaskou}}]  \label{lem:Vaskou_syll_length}
	Let $g$ be an element  of $A_{ab}$ that can be written only with positive letters and that has syllabic length 
	greater than $1$. If $m_{ab}\geq 3$, then the syllabic length of $g^n$ goes to infinity as $n$ goes to infinity. 
\end{lemma}

\begin{cor}\label{cor:power_centre}
	In a dihedral Artin group $A_{ab}$ with $m_{ab}\ge 3$, no non-trivial power of $z_{ab}$ is equal to a power of a standard generator.  
\end{cor}

\begin{proof}
	By Lemma~\ref{lem:Vaskou_syll_length}, the syllabic length of $z_{ab}^n$ explodes as $n$ grows, while the syllabic length of $a^n$ and $b^n$ is always equal to one.
\end{proof}

\begin{lemma}
	In a dihedral Artin group $A_{ab}$ with $m_{ab}\ge 3$, two distinct conjugates of standard generators never generate a 
	subgroup 
	isomorphic to $\mathbb{Z}^2$.  
	\label{lem:dihedral Z2}
\end{lemma}

\begin{proof} 	Let us consider a dihedral Artin group $A_{ab}$   with $m_{ab}\ge 3$. Up to conjugation, we can assume that the element $a$ and a 
	conjugate $x\coloneqq gcg^{-1}$, with $c \in \{a,b\}$, commute. Since the centraliser of $a$ in $A_{ab}$ is $\langle a, z_{ab}\rangle$ by \cite[Lemma 7]{Crisp}, it follows that there exist integers 
	$\ell, k$ such that $x = a^kz_{ab}^\ell$. We claim that necessarily $\ell = 0$. Indeed, if that were not the case, then the syllabic length of the powers of $z_{ab}^\ell$ would go to infinity by 
	Lemma \ref{lem:Vaskou_syll_length}, and since $a^k$ and $z_{ab}^\ell$ commute, so would the syllabic length of the powers of $x$ (since the powers of $a$ all have syllabic length $1$). But since $x$ is 
	conjugate to a power of a generator, its powers have a uniformly bounded syllabic length, a contradiction. We thus have $x = a^k$. By using the homomorphism $A_{ab}\rightarrow \mathbb{Z}$ sending 
	both generators to $1$, we get that $k=1$, hence $x = a$.
	
	By taking the contrapositive,  distinct conjugates of standard generators of $A_{ab}$ do not commute.
\end{proof}

\subsection{The modified Deligne complex}

Parabolic subgroups of finite type of an Artin group are used to define a simplicial complex as follows: 

\begin{defn}[Modified Deligne complex \cite{CharneyDavis}]
	The cosets $gA_{\Gamma'}$ of  standard parabolic subgroups of finite type of $A_{\Gamma}$ form a partially 
ordered set, for the partial order given by 
	$$gA_{\Gamma'} < gA_{\Gamma'' }~~~~~~~~ \mathrm{ if } ~g \in A_{\Gamma}~ \mathrm{ and } ~\Gamma' \subsetneq 
\Gamma'' \mbox{ are full subgraphs of } \Gamma.$$
	The \textbf{modified Deligne complex} (or Charney--Davis complex) $D_\Gamma$ of an Artin group $A_\Gamma$ is the geometric realisation of this 
poset. That is, vertices of $D_\Gamma$ correspond to cosets $gA_{\Gamma'}$ of  standard parabolic subgroups of finite 
type, and for every chain of the form 
	$$gA_{\Gamma_0} < gA_{\Gamma_1}  < \cdots < gA_{\Gamma_n},$$
 we add an $n$-simplex spanned by the vertices $gA_{\Gamma_0}, gA_{\Gamma_1}, \ldots, 
gA_{\Gamma_n}$. 
 The group $A_\Gamma$ acts on its modified Deligne 
complex by left multiplication on left cosets.
	\end{defn}

\begin{conv}
	From now on, we fix a large-type  Artin group $A_\Gamma$ of hyperbolic type. 
\end{conv}

Since $A_\Gamma$ is assumed to be of large-type and of hyperbolic type here, its only parabolic subgroups of finite type are its parabolic subgroups on at most two generators. In particular, the Deligne complex of $A_\Gamma$ is a $2$-dimensional simplicial complex.

	It was shown in \cite{CharneyDavis} that  for an Artin group of large and hyperbolic type, there 
exists an $A_\Gamma$--invariant piecewise hyperbolic metric that turns $D_\Gamma$ into a CAT($-1$) space. 
From now on, we will assume that 
$D_\Gamma$ is endowed with such a metric.

	\begin{notation} 
		For simplicity, we will often omit the `modified' from the name and call $D_\Gamma$ the Deligne complex. 

The vertex of $D_\Gamma$ corresponding to the standard dihedral parabolic 
subgroup $A_{ab}$ will be denoted $v_{ab}$.  Vertices of $D_\Gamma$ corresponding to cosets of  dihedral parabolic subgroups are  said to be of
\textbf{dihedral type}.
	\end{notation}

\begin{rem}
	We describe the stabilisers of vertices of the Deligne complex. Since a vertex  of $D_\Gamma$ is a left coset of the 
form $gA_{\Gamma'}$ (with $\Gamma' \subset \Gamma$), its stabiliser is the conjugate $gA_{\Gamma'}g^{-1}$. In particular, we 
get the following description, for each type of vertices of $D_\Gamma$:
	\begin{itemize}
		\item A vertex that corresponds to a left coset of the trivial subgroup has  trivial stabiliser.
		\item A vertex  that corresponds to a left coset of the form $g\langle a \rangle$, with $a\in V(\Gamma)$, 
has a stabiliser that is infinite cyclic.
		\item A vertex of dihedral type has a stabiliser that is isomorphic to a dihedral Artin group.
	\end{itemize}
\end{rem}

\subsection{Standard trees and the coned-off Deligne complex}\label{subsec:std_tree}

The structure of fixed-point sets of parabolic subgroups of $A_\Gamma$ play a crucial role. We start by a useful result: 

\begin{lemma}\label{lem:self-normalising}
	The fixed-point set in $D_\Gamma$ of a parabolic subgroup on two generators of $A_\Gamma$ is a single vertex. In particular, such parabolic subgroups are self-normalising.
\end{lemma}

\begin{proof}
	If a parabolic subgroup $gA_{ab}g^{-1}$ were to fix two distinct points of $D_\Gamma$, then it would fix the unique CAT(0) geodesic of $D_\Gamma$ between them.  Since stabilisers of edges and triangles 
of $D$ are either trivial or $\Z$, $gA_{ab}g^{-1}$ would embed in an infinite cyclic group, which is impossible as $A_{ab}$ is either $\Z^2$ or contains a copy of $\Z^2$ by Lemma~\ref{lem:virtually splits}.
	
	Since an element of the normaliser of $gA_{ab}g^{-1}$ stabilises $\mbox{Fix}(gA_{ab}g^{-1})$, hence fixes the vertex $gv_{ab}$, it follows that $gA_{ab}g^{-1}$ is self-normalising.
\end{proof}

The fixed-point sets of infinite cyclic parabolic subgroup are much more interesting:
	
	\begin{defn}[Standard trees {\cite[Definition 4.1]{MP}}]\label{defn:standard_tree}
		For an element $g \in A_\Gamma$ that is a conjugate of a standard generator, the fixed-point set 
$\mathrm{Fix}(g)$ is a convex subtree of $D_\Gamma$ that is contained in the $1$-skeleton of $D_\Gamma$. 

		Such subtrees are called \textbf{standard trees} of $D_\Gamma$. 
		
		Note in particular that all edges of a standard tree have the same infinite cyclic stabiliser. 
	\end{defn}

We list here a few immediate results:

\begin{lemma}\label{lem:edge_stab_tree}
	Two edges of $D_\Gamma$ have stabilisers that either intersect trivially or are equal. Moreover, if two edges of $D_\Gamma$ have the same non-trivial stabiliser, then they belong to the same standard tree.
\end{lemma}

\begin{proof}
	Since a non-trivial element of $A_\Gamma$ stabilising two points fixes pointwise the unique CAT(0) geodesic between them, this result is a direct consequence of \cite[Lemma~4.3]{MP}.
\end{proof}

\begin{cor}\label{cor:tree_power}
	For every standard generator $a$ and non-zero integer $k\in \Z-\{0\}$, the trees $\mbox{Fix}(a)$ and $\mbox{Fix}(a^k)$ coincide.
\end{cor}

\begin{proof}
	The inclusion $\mbox{Fix}(a)\subseteq\mbox{Fix}(a^k)$ is clear, so let us show the other inclusion. Consider any point $x$ in $\mbox{Fix}(a^k)$, and any edge $e$ of the tree $\mbox{Fix}(a)$. If $x$ lies in $e$, we are done, otherwise we can consider a minimal length geodesic from $x$ to $e$, which is $a^k$-invariant. Since triangles of $D_\Gamma$ have trivial stabilisers, this geodesic is contained in the $1$-skeleton and it intersects an edge $e'$ in a non-trivial subpath containing $x$ (if $x$ is not a vertex we cannot say that the geodesic contains the edge). We have that $e'$ is also $a^k$-invariant, and since $k\neq 0$ we have that the stabilisers of $e$ and $e'$ intersect non-trivially, and therefore by Lemma~\ref{lem:edge_stab_tree} $e'$, whence $x$, belongs to $\mbox{Fix}(a)$.
\end{proof}

\begin{cor}\label{cor:intersect_stab_nontrivial}
	Two points of $D_\Gamma$ have stabilisers that intersect non-trivially if and only if they are contained in a common standard tree. 
\end{cor}

\begin{proof}
 If two points  $x, y$ of $D_\Gamma$ 
have a nontrivial common stabiliser $H$, then $H$ fixes pointwise the unique CAT(0) geodesic $\gamma$  between them.  The geodesic $\gamma$ cannot pass through the interior of a triangle because $H$ is non-trivial, so $\gamma$ is contained in the $1$-skeleton and is contained in a minimal path of edges of the form $e_1, \ldots, e_n$.  So $H$ fixes edges $e_1,\ldots, e_n$, and by Lemma~\ref{lem:edge_stab_tree}, it follows that $e_1,\ldots, e_n$, hence $x$ and $y$, are contained in the same standard tree. (If $n=1$, note that an edge with non-trivial stabiliser belongs to a standard tree.)
\end{proof}

\begin{cor}\label{cor:intersect_trees}
	Two distinct standard trees intersect in at most one vertex.
\end{cor}

\begin{proof}
	This follows directly from Lemma \ref{lem:edge_stab_tree} and the convexity of standard trees.
\end{proof}

\begin{defn}[Coned-off Deligne complex {\cite[Definition 4.8]{MP}}]
	The \textbf{coned-off Deligne complex}, denoted $\widehat{D}_\Gamma$, is obtained from $D_\Gamma$ by coning-off each 
standard tree of $D_\Gamma$. That is, for every standard tree $T$ of $D_\Gamma$ we add a new vertex $v_T$, which we connect 
by an edge to every vertex of $T$. The complex $\widehat{D}_\Gamma$ is then the flag completion of the resulting complex. The action of $A_\Gamma$ on $D_\Gamma$ extends to an action on $\widehat{D}_\Gamma$.
\end{defn}

\begin{notation}
	  For a standard generator $a \in \Gamma$, the standard tree $\mbox{Fix}(a)$ (i.e. the standard tree containing the vertex $\langle a \rangle$ of $D_\Gamma$) will be denoted $T_a$, and the apex of the cone 
$\widehat{T}_a$ over that tree will be denoted $v_a$, see Figure~\ref{coneoff}. With this notation, the vertex $gv_a$ corresponds to the apex of the cone over the standard tree $\mbox{Fix}(gag^{-1})$ containing $g\langle a \rangle$. The apex of a cone over a standard tree will be  called a vertex 
\textbf{of tree type}.
\end{notation}

\begin{figure}[h]
	\begin{center}
		\scalebox{1}{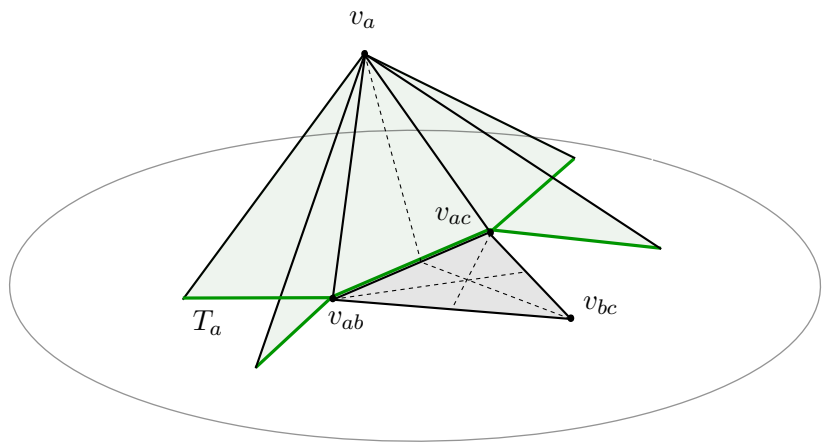}
		\caption{A portion of the coned-off Deligne complex $\widehat{D}_\Gamma$ for an Artin group on three 
			generators $a, b, c$. A fundamental domain for the action of $A_\Gamma$ on $D_\Gamma$ is represented in  grey, and is a 
			subdivided triangle with vertices $v_{ab}, v_{bc}, v_{ac}$. A portion of the standard tree $T_a$ is represented in green. In 
			the coned-off Deligne complex $\widehat{D}_\Gamma$, this tree is the basis of a cone with apex the vertex $v_a$.}
		\label{coneoff}
	\end{center}
\end{figure}

\begin{rem}
	By work of Paris  \cite[Corollary 4.2]{Paris}, two standard generators are conjugated if and only if there is a path in the presentation graph $\Gamma$ consisting of edges with odd labels connecting the corresponding vertices.   Since distinct generators may be conjugated, it may happen that we have an equality of the form $gv_a=hv_b$ for distinct standard generators $a, b$. The following result shows that this is the  only case where such an equality happens.
\end{rem}

\begin{lemma}\label{lem:tree_label_conjugate}
	Let $T$ be a standard tree of $D_\Gamma$, let $a, b$ be two standard generators, and let $g, h \in A_\Gamma$. If the two vertices $g\langle a \rangle$ and $h\langle b \rangle$ of $D_\Gamma$ are  contained in $T$, then $a$ and $b$ are connected in $\Gamma$ by a path with odd labels. (In particular, $a$ and $b$ are conjugated.) 
\end{lemma}

\begin{proof}
	Since $T$ is a connected tree, we consider a geodesic path $e_1, \ldots, e_k$ of $T$ from  $g\langle a \rangle$ to $h\langle b \rangle$. Each edge $e_i$ joins a coset of a cyclic standard parabolic subgroup and a coset of parabolic subgroup of dihedral type. For each $i$, let $a_i$ be the unique standard generator such that $e_i$ contains a vertex that is a coset of $\langle a_i\rangle$. It is enough to show that for every $1 \leq i <k$, $a_i$ and $a_{i+1}$ are either equal or connected by an edge of $\Gamma$ with odd label. Consider the  vertex $v$ of $D_\Gamma$ where $e_i$ and $e_{i+1}$ meet. If $v$ corresponds to a coset of cyclic standard parabolic subgroup, then $a_i = a_{i+1}$. If $v$ is a vertex of dihedral type, then it follows from \cite[Lemma~4.3]{MP} that either $a_i=a_{i+1}$, or $a_i$ and $a_{i+1}$ are adjacent in $\Gamma$ (since $e_i$ and $e_{i+1}$ meet along a vertex that is a coset of $A_{a_i, a_{i+1}}$, which must be an Artin group of finite type by construction of $D_\Gamma$) and the label of that edge is odd. This concludes the proof.
\end{proof}

\begin{rem}\label{conv:metric_Dhat}	It was shown in \cite[Proposition 4.8]{MP} that there exists an $A_\Gamma$--invariant piecewise hyperbolic metric that turns 
$\widehat{D}_\Gamma$ into a CAT($-1$) space. From now on, we will assume that $\widehat{D}_\Gamma$ is endowed such a metric from~\cite{MP}. 

It should be noted that this metric depends on a constant $\varepsilon>0$ that can be chosen arbitrary small, see \cite[Definition~4.7]{MP}. This constant is such that for an edge $e$ of $D_\Gamma$ contained in a standard tree, the triangle of $D_\Gamma$ over $e$ has angles at least $\pi/2-\varepsilon$ at the two vertices of $e$. In this article, the choice of $\varepsilon$ will be mostly irrelevant. We will only need to consider this constant in Lemma~\ref{lem:tree_convex_coneoff} below, where $\varepsilon$ needs to be smaller than a certain constant depending only on the group $A_\Gamma$. 
\end{rem}
 
 We mention here a slight generalisation of the CAT($-1$)-ness of $\widehat{D}_\Gamma$, which will be used in Section~\ref{sec:augment}.
 
\begin{lemma}\label{rem:subspace_coneoff}
	 Let  $Z$ be the full subcomplex of  $\widehat{D}_\Gamma$ whose vertex set is obtained from $\widehat{D}_\Gamma$ by removing some vertices of tree type. Then $Z$ is also CAT($-1$) for the induced metric.
\end{lemma}
	
	\begin{proof}
	 The complex $Z$ can be thought of as being obtained from the simply connected complex $D_\Gamma$ by coning-off only certain (contractible) standard trees of $D_\Gamma$, so $Z$ is also simply connected. Moreover, for a point $x\in Z$, the link $\link_Z(x)$ is a subgraph of $\link_{\widehat{D}_\Gamma}(x)$, since $\widehat{D}_\Gamma$ is a two-dimensional complex. Since $\widehat{D}_\Gamma$ is CAT($-1$), links of points are graphs with a systole of at least $2\pi$, so the same is true for the subgraph $\link_Z(x)$. Thus, $Z$ is locally CAT($-1$). Since $Z$ is simply-connected and locally CAT($-1$), it is CAT($-1$).
	\end{proof}

The standard trees of $D_\Gamma$ are convex in $D_\Gamma$ but become bounded in $\widehat{D}_\Gamma$. We mention the following intermediate result, which will be used in Section~\ref{subsec:augmented}:

\begin{lemma}\label{lem:tree_convex_coneoff}
	We can choose the constant $\varepsilon>0$ from Remark~\ref{conv:metric_Dhat} small enough so that the following holds: Let $T$ be a standard tree of $D_\Gamma$, and let $Z$ be the full subcomplex of $\widehat{D}_\Gamma$ obtained by removing the vertex of tree type associated to $T$. Then $T$ is convex in $Z$ for the induced metric.
	\end{lemma}

\begin{proof}
	It follows from Lemma~\ref{lem:tree_convex_coneoff} that $Z$ is CAT($-1$) for the induced metric. Since $T$ is a subtree of the CAT($-1$)  complex $Z$, it is enough to show that two edges of $T$ that share a vertex $v$ make an angle of at least $\pi$ at $v$. (This follows from the local characterisation of geodesics in a CAT(0) space) This amounts to showing that for distinct vertices $w, w'$ of $T\cap \link_{Z}(v)$, their distance in $\link_{Z}(v)$ is at least $\pi$. By construction, the simplicial graph $\link_{Z}(v)$ is obtained from $\link_{D_\Gamma}(v)$ by coning-off  the sets $T'\cap \link_{D_\Gamma}(v)$ for each standard tree $T'$ other than $T$, i.e. by constructing a simplicial cone over every such sets $T'\cap \link_{D_\Gamma}(v)$. Moreover, each of these new edges has length at least $\pi/2-\varepsilon$ by construction, see Remark~\ref{conv:metric_Dhat}.  Since the action of $A_\Gamma$ on $D_\Gamma$ is cocompact, there is a uniform lower bound on the length of edges in the link of an arbitrary vertex of $D_\Gamma$, and we can choose the constant $\varepsilon<\pi/10$ smaller than half this uniform lower bound, which we now assume. We now claim  that the distance  in $\link_{Z}(v)$ between $w$ and $w'$ is at least $\pi$. Indeed, assume that this were not the case. Since $T$ is a convex subtree of $D_\Gamma$ by construction, two distinct points $w, w'$ of $\link_{D_\Gamma}(v)$ that belong to $T$ are at distance at least $\pi$ in $\link_{D_\Gamma}(v)$. 
	
	Therefore, if there was a geodesic of length less than $\pi$ between $w$ and $w'$ in $\link_{Z}(v)$, it would have to go through at least two of the additional edges of length $\geq \pi/2-\varepsilon$ coming from the cone-off procedure. Moreover, since $w$ and $w'$ are not in the same standard tree (Corollary \ref{cor:intersect_trees}), said geodesic should also contain another edge, which has length at least $2\varepsilon$. Therefore, the geodesic would have length at least 
	$$2\varepsilon + 2(\pi/2 - \varepsilon) \geq\pi,$$
	a contradiction. Thus, $T$ is convex in $Z$.
\end{proof}

\begin{conv}\label{conv:metric_coneoff}
	From now on, we will assume that $\widehat{D}_\Gamma$ is endowed with a CAT($-1$) metric from~\cite{MP} such that Lemma~\ref{lem:tree_convex_coneoff} holds.
	
	Moreover, since we are considering a fixed Artin group $A_\Gamma$, we will from now on simply denote by $D$ and $\widehat{D}$ the complexes 
	$D_{\Gamma}$ and $ \widehat{D}_\Gamma$ . 
\end{conv}

We now describe the stabiliser of the vertices of $\widehat{D}$ of tree type:

\begin{lemma}\label{lem:stab_tree}
	The stabiliser of the vertex of tree type $v_a$ is exactly the centraliser (and normaliser) of the cyclic subgroup $\langle a 
\rangle$. Moreover, this centraliser splits as a direct product of the form 
	$$ \langle a \rangle \times K,$$
	where $K$ is a finitely-generated free group. 
	More precisely, the subgroup $\langle a \rangle$ acts trivially on 
$T_a$, while $K$ acts cocompactly on it with trivial edge stabilisers.
	\end{lemma}

\begin{proof}
	Let us first note that the normaliser and centraliser of $\langle a \rangle$ coincide. Indeed, if $g \in A_\Gamma$ is such that $g^{-1}ag \in \langle a \rangle$, we apply the homomorphism $A_\Gamma \rightarrow \Z$ sending every generator to $1$ and deduce that $g^{-1}ag=a$, hence $g$ centralises $\langle a \rangle$. 
	
	The stabiliser of $v_a$ coincides with the global stabiliser of the standard tree $T_a$. Let us show that an element 
$g \in A_\Gamma$ stabilises $T_a$ if and only if it  normalises $\langle a \rangle$. 
	
	If $g $ stabilises $T_a$, then it sends the vertex $\langle a \rangle \in T_a$ to the vertex $g\langle a \rangle\in 
T_a$, so in particular we have 
	$$a\cdot g\langle a \rangle = g \langle a \rangle,$$
	or in other words $g^{-1}ag \in \langle a \rangle$, hence $g$ normalises $\langle a \rangle$. Conversely, 
let us assume that $g$ normalises $\langle a \rangle$, and let $x$ be a point of $T_a$. Let us show that $gx\in T_a$. We 
have 
	$$a\cdot gx = g(g^{-1}ag)x = gx,$$
	the last equality following from the fact that $x$ is fixed by $\langle a\rangle$ by definition of $T_a$. Thus, $gx$ is fixed by 
$a$, and it follows that $g$ stabilises $T_a$.

The decomposition of the stabiliser of $T_a$ as a direct product as in the statement is a consequence of \cite[Lemma 
4.5]{MP}. Observe that $\mathrm{Stab}(v_a)$ acts cocompactly on $T_a$.  Indeed, $D$ contains finitely many $A_\Gamma$ orbits of edges, and no edge is contained in distinct translates of a 
standard tree by Corollary~\ref{cor:intersect_trees}.

The quotient $K$ was defined in the proof of \cite[Lemma 4.5]{MP} as  the fundamental group of a graph of groups over the graph $T_a/\mathrm{Stab}(v_a)$, with trivial edge stabilisers and vertex stabilisers that are trivial or  
infinite cyclic. Since $T_a/\mathrm{Stab}(v_a)$ is a finite graph by the above, it follows that $K$ is finitely generated.
\end{proof}

Regarding vertices of dihedral type, we have the following similar result:

\begin{lemma}\label{lem:stab_dihedral}
	The stabiliser $A_{ab}$ of the vertex of dihedral type $v_{ab}$ is exactly the centraliser of the cyclic subgroup $\langle z_{ab}^p 
	\rangle$ for all $p\neq 0$, which further coincides with the normaliser.
\end{lemma}

\begin{proof}
	Since the stabiliser of $v_{ab}$ is equal to $A_{ab}$ by construction, and $z_{ab}^p$ is central in $A_{ab}$, it is enough to show that an element in the normaliser of $z_{ab}^p$ fixes $v_{ab}$. Suppose that we have 
an element $h$ such that $h\in N(z_{ab}^p)$.   Then $z_{ab}^q=hz_{ab}^ph^{-1}$ (for some $q\neq 0$) also fixes the vertex $hv_{ab}$, so $z_{ab}^q$ fixes pointwise the unique CAT(0) geodesic of $D$ joining $v_{ab}$ to $hv_{ab}$.  If this 
geodesic is nontrivial, then $z_{ab}^q$ fixes an edge of $D$ containing $v_{ab}$, since triangles in $D$ have trivial stabilisers. This is impossible, as otherwise $z_{ab}^q$ would be contained in an 
edge stabiliser, hence would be conjugate in $A_{ab}$ to a power of $a$ or $b$ (since the edge in question contains $v_{ab}$). Hence $z^q_{ab}$ would be equal to a power of $a$ or $b$ since $z_{ab}$ is 
central in $A_{ab}$, contradicting Corollary~\ref{cor:power_centre}.   So, $hv_{ab}=v_{ab}$, as required.
\end{proof}

We will also mention the following lemma about centralisers, which will be used later in this article (see Lemma~\ref{lem:blow_up_data_existence}):

\begin{lemma}\label{lem:centraliser_power}
	Let $c\in A_\Gamma$ and suppose that either $c$ is a standard generator, or $c=z_{ab}$ for some standard generators $a$ and $b$ generating a dihedral Artin group.  Then the centraliser $C(c)$ 
	satisfies $C(c)=C(c^p)$ for all $p\in\integers-\{0\}$.
\end{lemma}

\begin{proof}
	There are two cases, according to whether $c$ is a standard generator or $c=z_{ab}$ for standard generators $a,b$.
	
	First consider the case where $c=z_{ab}$, where $a,b$ are standard generators generating a dihedral type Artin subgroup $A_{ab}$.  Recall that $\langle z_{ab}\rangle$ fixes a point in $\widehat 
D$, namely the vertex $v_{ab}$.  Suppose that, for some $p\in\integers-\{0\}$, we have an element $h$ such that $h\in C(z_{ab}^p)$.   By Lemma 
\ref{lem:stab_dihedral}, we have that $hv_{ab}=v_{ab}$, whence $h\in A_{ab}$, and  this subgroup is equal to $C(z_{ab})$ by Lemma \ref{lem:stab_dihedral}.
	
	Next consider the case where $c$ is a standard generator $c=a$, and let $e$ be an edge of the corresponding standard tree $T_a$, which has $\langle a \rangle$ as stabiliser. Suppose that, for some 
$p\in\integers-\{0\}$, we have an element $h$ such that $h\in C(a^p)$. Then the stabiliser of the edge $he$ is $\mathrm{Stab}(he)= h\langle a\rangle h^{-1}\supset \langle a^p \rangle$. In particular, 
the stabilisers of $e$ and $he$ intersect non-trivially, and it follows from Lemma~\ref{lem:edge_stab_tree} that they are in the same standard tree $T_a$.  From Corollary~\ref{cor:intersect_trees}, 
we get that $hT_a=T_a$, so $h\in C(a)$.
\end{proof}

\subsection{Links of vertices}

The local structure of Deligne complexes will play an important role in this article, so we now describe it in further detail.

\begin{lemma}\label{lem:link_tree_cocompact}
	Let $v_a$ be a vertex of tree type. The link $\mathrm{Lk}_{\widehat{D}}(v_a)$ is the standard tree $T_a$, and the action of $\mathrm{Stab}(v_a)$ on it is cocompact.
\end{lemma}

\begin{proof}
	The description of the link follows from the construction of the cone-off, and the cocompactness follows from Lemma~\ref{lem:stab_tree}.
\end{proof}

Throughout the rest of this subsection, we fix a vertex of $D$ of dihedral type of the form $v_{ab}$. Before describing the link in the cone-off $\widehat{D}$, we start by describing the link in the original Deligne complex $D$.
The link  $\link_D(v_{ab})$ of that vertex has a simple 
description, which is a direct consequence of the construction of $D$: 

\begin{lemma}\label{lem:link_dihedral_cocompact}
	The link  $\link_D(v_{ab})$ is $A_{ab}$-equivariantly 
isomorphic to  the  the geometric realisation of the poset of cosets of strict standard parabolic subgroups of $A_{ab}$. That 
is, vertices of $\link_D(v_{ab})$ correspond to cosets of the form $g\langle a\rangle, g \langle b \rangle, $ or $g\{1\}$, 
and for every $g \in A_{ab}$, we add an edge between $g\{1\}$ and $g\langle a\rangle$, as well as an edge between  $g\{1\}$ 
and $g \langle b \rangle$. 

In particular, the action of   $\mathrm{Stab}(v_{ab})$ on $\mathrm{Lk}_{D}(v_{ab})$ is cocompact. \qed
\end{lemma} 

This link is a particular case of a general construction that we will  use again in Section~\ref{sec:augment}: 

\begin{defn}[Graph of orbits]\label{def:graph of orbits}
	Let $\calG$ be a graph with an $A_{ab}$-action, such that the subgroups $\langle a \rangle$ and $\langle b \rangle$ 
act freely on it. We define a new graph encoding the pattern of intersections of orbits of $\langle a \rangle$ and $\langle b 
\rangle$ as follows: We put a vertex for every $\langle a \rangle$-orbit and one vertex for every $\langle b \rangle$-orbit. 
If two such orbits have a non-empty intersection, we put an edge between them. The \textbf{graph of orbits} 
$\mathrm{Orbit}_{a, b}\big(\calG\big)$ is defined as the first barycentric subdivision of the graph obtained in this way.
\end{defn}	

\begin{rem}\label{rem:graph_of_orbits}
With that terminology, the link $\link_D(v_{ab})$ is $A_{ab}$-equivariantly isomorphic to the graph of orbits 
$\mathrm{Orbit}_{a, b}\big(\mathrm{Cayley}_{a,b}(A_{ab})\big)$, where $\mathrm{Cayley}_{a,b}(A_{ab}) $ denotes the Cayley 
graph of $A_{ab}$ for the standard generators $a, b$.
\end{rem}

The action of $A_{ab}$ on $\mathrm{Lk}_D(v_{ab})$ has been studied by Vaskou in \cite{Vaskou}. In particular, the following result, which is a geometric counterpart of Lemma~\ref{lem:Vaskou_syll_length} will be useful in Section~\ref{sec:augment}: 

\begin{lemma}[{\cite[Proposition 4.7]{Vaskou}}]  \label{lem:Vaskou_QI_embedded_link}
	Let $g$ be an element  of $A_{ab}$ that can be written only with positive letters and that has syllabic length
greater than $1$. If $m_{ab}\geq 3$, then the $\langle g \rangle$-orbit maps to $\mathrm{Lk}_D(v_{ab})$ are quasi-isometric embeddings.\qed
\end{lemma}

 Since the Garside element $\Delta_{ab}$ either conjugates the generators $a, b$ (when $m_{ab}$ is odd) or centralises each of them (when 
$m_{ab}$ is even), the action of $\langle \Delta_{ab} \rangle$ by \textit{right} multiplication induces an action on the 
\textit{left cosets}. Indeed, we have 
$$g\langle a \rangle \Delta_{ab} = g \Delta_{ab} \langle b \rangle ~~ \mbox{ and }  ~~ g\langle b \rangle \Delta_{ab} = 
g \Delta_{ab} \langle a \rangle  ~~\mbox{ if $m_{ab}$ is odd,}$$
$$g\langle a \rangle \Delta_{ab} = g \Delta_{ab} \langle a \rangle ~~ \mbox{ and }  ~~ g\langle b \rangle \Delta_{ab} = 
g \Delta_{ab} \langle b \rangle  ~~\mbox{ if $m_{ab}$ is even.}$$

We now describe the link of $v_{ab}$ in the coned-off Deligne complex. There is a simple characterisation of the trace of a standard tree on the link of a vertex of dihedral type. The following is a reformulation of \cite[Lemma 4.3]{MP}:

\begin{lemma}\label{lem:Garside orbit}
	Two vertices of the link $\link_D(v_{ab})$ correspond to edges in the same standard tree of $D$ if and only if the 
corresponding cosets are in the same $\langle \Delta_{ab} \rangle$-orbit (for the multiplication on the right).
	
	Moreover, two vertices of the link $\link_D(v_{ab})$ that are \textit{in the same $A_{ab}$-orbit} correspond to edges 
in the same standard tree of $D$ if and only if the corresponding cosets are in the same $\langle z_{ab} \rangle$-orbit (for 
the multiplication on the right).\qed
\end{lemma}

This local characterization of standard trees allows us to simply describe the links of vertices in the coned-off space 
$\widehat{D}$:

\begin{cor}\label{cor: link coneoff}
	 The link $\link_{\widehat{D}}(v_{ab})$ is obtained from $\link_D(v_{ab})$  by coning-off every $\langle \Delta_{ab} 
\rangle$-orbit of vertices corresponding to cosets of the form $g \langle a \rangle$ or $g \langle b \rangle$.

In particular, the action of $\mathrm{Stab}(v_{ab})$ on $\link_{\widehat{D}}(v_{ab})$ is cocompact.
	 \end{cor}
 
 \begin{proof}
 	The description of the link is a consequence of Lemma~\ref{lem:Garside orbit}. The cocompactness of the action of  $\mathrm{Stab}(v_{ab})$ on 
$\link_{D}(v_{ab})$ follows from Lemma~\ref{lem:link_dihedral_cocompact}. Moreover, since there are finitely many $\mathrm{Stab}(v_{ab})$-orbits of edges of $D$ containing $v_{ab}$, there are in particular finitely many orbits of standard trees containing $v_{ab}$, and hence finitely many orbits of apices of standard trees containing $v_{ab}$. The cocompactness of the action of $\mathrm{Stab}(v_{ab})$ on 
$\link_{\widehat{D}}(v_{ab})$ now follows.
 \end{proof}

\section{The commutation graph of an Artin group}\label{sec:commute}

\subsection{The commutation graph}\label{subsec:commutation_graph}

We now construct a complex to which Theorem \ref{thm:main criterion} will be applied, and which turns out to be 
quasi-isometric to the coned-off Deligne complex $\widehat{D}$. This complex is obtained via a general construction that 
encodes the commutation of certain chosen subgroups of a given group that we now describe.

\begin{defn}[Commutation graph]\label{def:commutation graph}
	Let $G$ be a group, and let $\mathcal{H}$ be a set of subgroups of $G$. We define a simplicial graph, called the 
	\textbf{commutation graph} of $\calH$ and denoted $Y_\calH$, as follows. The vertex set of $Y_\calH$ is
	$$\bigsqcup_{H\in\calH} G/N(H),$$
and we put an edge between $gN(H)$ and $hN(H')$ if $gHg^{-1}$ and $hH'h^{-1}$ commute, that is, every element of one subgroup commutes with every element of the other subgroup. Note that this is independent of the choice of coset representatives $g$ and $h$. The group $G$ acts on this graph by left multiplication.
\end{defn}

\begin{rem}
	Commutation graphs have already been considered in the work of Kim--Koberda on right-angled Artin groups under the 
	name of \textit{extension graphs}~\cite{KimKoberda}. Indeed, for a right-angled Artin group $G$ on generators $g_1, \ldots, g_n$, the extension 
	graph they study is exactly the commutation graph for the family $\calH = \big\{ \langle g_1 \rangle, \cdots, \langle g_n 
	\rangle\big\}.$
\end{rem}

We now discuss the family of subgroups of $A_\Gamma$ that yield the correct commutation graph for our purposes.

First, the motivation. The acylindricity of the action of $A_\Gamma$ 
on its coned-off Deligne complex and the CAT($0$) 
geometry of this complex can be shown to imply that the maximal subgroups of $A_\Gamma$ that virtually split as products 
are the dihedral parabolic subgroups (stabilisers of vertices of dihedral type), which are virtual products by 
Lemma \ref{lem:centre dihedral}, and the normalisers of standard generators, which are products by 
Lemma \ref{lem:stab_tree}. It is an exercise, left to the reader as it is not  needed in this 
article, to identify the minimal infinite subgroups obtained by taking arbitrary intersections of such maximal virtual 
products. They come in two families: 
\begin{itemize}
	\item The  cyclic subgroups generated by a conjugate of a standard generator: These subgroups are obtained by taking 
	the intersection of stabilisers of vertices of dihedral type contained in a common standard tree.
	\item  The centres of dihedral parabolic subgroups with $m_{ab}\geq 3$: These subgroups are obtained by taking the intersection of the 
	stabilisers of two  standard trees that share a vertex.
\end{itemize}

\begin{rem}
We know from the work of Paris \cite[Corollary 4.2]{Paris} that two standard generators are conjugate if and only if there is a path in the presentation graph $\Gamma$ consisting of edges with odd labels connecting the corresponding vertices. We can thus choose a set of representatives of conjugacy classes of elements of $V(\Gamma)$, which defines a subset $V_{\mathrm{odd}}(\Gamma)$ of $V(\Gamma)$.
\end{rem}

This motivates the following definition:

\begin{defn}\label{defn:fancy_H}
	 We define the following collection of subgroups of the Artin group  $A_\Gamma$: 
	$$\calH \coloneqq \big\{ \langle a \rangle ~|~ a \mbox{ a standard generator in } V_{\mathrm{odd}}(\Gamma) \big\} \cup \big\{ \langle 
	z_{ab} \rangle ~|~ a, b \mbox{ span an edge of } \Gamma  \mbox{ with } m_{ab}\geq 3 \big\}. $$ 
	We will simply denote by $Y=Y_\Gamma=Y_\calH$  the commutation graph of $\calH$. 
\end{defn}

\begin{lemma}\label{lem:no_conj_H}
	Distinct elements of $\calH$ are in different conjugacy classes. Moreover, if two conjugates of elements of $\calH$ intersect non-trivially, then they are equal.
\end{lemma}

\begin{proof}
	We have to show that for two distinct elements $x, y $ of the form $z_{ab}$ or $a\in V_{\mathrm{odd}}(\Gamma)$, no non-trivial power of $x$ is conjugate to a non-trivial power of $y$. We treat several different cases. 
	
	Let $a$, $b$ be two standard generators of $V_{\mathrm{odd}}(\Gamma)$. If non-trivial powers $a^k$ and $b^\ell$ were conjugate, then using the homomorphism $A_\Gamma \rightarrow \Z$ sending every generator to $1$ we would obtain that $k=\ell$. By \cite[Corollary 5.3]{Paris} (which says that an element conjugates $a^k$ to $b^k$ if and only if it conjugates $a$ to $b$), we get that $a$ and $b$ are conjugate, which is excluded by  construction of $V_{\mathrm{odd}}(\Gamma)$. 
	
	The centraliser of a standard generator $a$ has the form $K\times \langle a\rangle$ for some free group $K$ by Lemma \ref{lem:stab_tree}, while the centraliser of an element of the form $z_{a'b'}$ is the dihedral Artin group $A_{a'b'}$ by Lemma \ref{lem:stab_dihedral}. These are not isomorphic for example because the abelianisations can be isomorphic only if $K$ is trivial (consider the cases of $m_{a'b'}$ odd or even), but $A_{a'b'}$ is not isomorphic to $\mathbb Z$ by Lemma \ref{lem:virtually splits}.

Finally, suppose that non-trivial powers of $z_{ab}$ and $z_{a'b'}$ are conjugate. Then their centralisers are also conjugate. By Lemma \ref{lem:stab_dihedral}, these centralisers are $A_{ab}$ and $A_{a'b'}$. In turn, by Lemma \ref{lem:self-normalising}, these have fixed-point sets respectively consisting of $v_{ab}$ and $v_{a'b'}$ only, so that these fixed point sets are not translates of each other. Therefore, the centralisers cannot be conjugate, and the powers of $z_{ab}$ and $z_{a'b'}$ cannot be conjugate.	
\end{proof}

\begin{notation}
	For a standard generator  $a \in \Gamma$, we denote by $u_a$ the vertex of $Y$ corresponding to $N(\langle a 
	\rangle)$. The $A_\Gamma$-translates of such vertices are said to be  \textbf{of tree type}.  Analogously, for standard 
	generators $a, b$ spanning an edge of $\Gamma$ with $m_{ab}\geq 3$, we denote by $u_{ab}$ the vertex of $Y$ corresponding to $N(\langle z_{ab} 
	\rangle)$. The $A_\Gamma$-translates of such vertices are said to be of \textbf{dihedral type}. 
\end{notation}

In the rest of this section, we construct an equivariant quasi-isometry between the commutation graph $Y$ and the 
coned-off Deligne complex $\widehat{D}$. We start by defining such a map at the level of vertices:

\begin{lemma}\label{lem:iota_vertex}
	We define a map $\iota: Y^{(0)} \rightarrow \widehat{D}^{(0)}$ as follows:
	 \begin{itemize}
	 	\item For a  vertex of $Y$ of dihedral type of the form $gu_{ab}$, we set  $\iota(gu_{ab}) \coloneqq 
gv_{ab}$.
	 	\item For a vertex of $Y$ of tree type of the form $gu_a$, we set $\iota(gu_{a}) \coloneqq gv_{a}$.
	 \end{itemize}
 Then $\iota$  is well-defined, is injective, and realises bijections between the vertices of $Y$ of dihedral type and the  vertices of $\widehat{D}$ of dihedral type, as well as between the vertices of $Y$ of tree type and the  vertices of $\widehat{D}$ of tree type.
\end{lemma}

\begin{proof}
We have to show the following properties:

\begin{itemize}
 \item if $gN(\langle z_{ab}\rangle)=hN(\langle z_{ab}\rangle)$, then $gv_{ab}=hv_{ab}$.
 \item if $gN(\langle a\rangle)=hN(\langle a\rangle)$, where $a\in V_{\mathrm{odd}}(\Gamma)$, then $gv_a=hv_a$.
\end{itemize}

If $gN(\langle z_{ab}\rangle)=hN(\langle z_{ab}\rangle)$, then $g^{-1}h \in N(\langle z_{ab}\rangle)=A_{ab}$, the latter equality following from Lemma~\ref{lem:stab_dihedral}, and we know from Lemma~\ref{lem:stab_dihedral} that $N(\langle z_{ab}\rangle)$ stabilises $v_{ab}$. It thus follows that $hv_{ab} = g(g^{-1}h)v_{ab} = gv_{ab}$, so $\iota$ is well defined on vertices of dihedral type. Moreover, it in fact realises a bijection on vertices of dihedral type as we just saw that in both $Y$ and $\widehat D$ these vertices are cosets of the $A_{ab}$. 

If $gN(\langle a\rangle)=hN(\langle a\rangle)$, then $g^{-1}h \in N(a)$, and we know from Lemma~\ref{lem:stab_tree} that $N(a)$ stabilises $v_{a}$. It thus follows that $hv_{a} = g(g^{-1}h)v_{a} = gv_{a}$, so $\iota$ is well defined on vertices of tree type. 

Let us show that $\iota$ is surjective on vertices of tree type. Let $gv_a$ be a vertex of $\widehat D$ of tree type. By construction, there exists a standard generator $b \in V_{\mathrm{odd}}(\Gamma)$ and an element $x\in A_\Gamma$ such that $xbx^{-1}=a$. It thus follows that $xv_b=v_a$, and in particular $gv_a = gxv_b = \iota(gxu_b)$.

Let us now show that $\iota$ is injective on vertices of tree type. Consider $a,b\in V_{\mathrm{odd}}(\Gamma)$ and $g,h\in A_{\Gamma}$ such that $gv_a=hv_b$. We want to show that $gu_{a}=hu_b$. By construction of vertices of tree type, this means that the standard tree of $D$ containing the vertex $g\langle a \rangle$ and the standard tree of $D$ containing the vertex $h\langle b \rangle$ coincide. 
In particular, the vertices $g\langle a \rangle$ and $h\langle b \rangle$ of $D$ are in the same standard tree, and it follows from Lemma~\ref{lem:tree_label_conjugate} that $a$ and $b$ are connected by a path of $\Gamma$ with odd labels. By definition of $V_{\mathrm{odd}}(\Gamma)$, this implies that $a=b$. We thus have $gv_a=hv_a$, and so $g^{-1}h\in \mbox{Fix}(v_a)=N(a)$, the latter equality following from Lemma~\ref{lem:stab_tree}. We thus have $gN(a)=hN(a)$, hence $gu_{a}=hu_a=hu_b$, which shows injectivity. 

Note that the injectivity of $\iota: Y^{(0)} \rightarrow \widehat{D}^{(0)}$ is now straightforward, since the sets of vertices of tree type and dihedral type of $\widehat{D}$ are disjoint.
\end{proof}

\begin{lemma}
The map $\iota$  is well-defined, is injective, and realises bijections between the vertices of $Y$ of dihedral type and the  vertices of $\widehat{D}$ of dihedral type, as well as between the vertices of $Y$ of tree type and the  vertices of $\widehat{D}$ of tree type.
	
Moreover, two vertices $v, v'$ of $Y$ are adjacent if and only if the following occurs: One of them (say $v$) is of dihedral type, the other (say $v'$) is of 
	tree type, and $\iota(v)$ is contained in the standard tree having $\iota(v')$ as apex. 
	
	In particular, $Y$ is a bipartite graph with respect to the type of vertices.
\label{lem:iota}
\end{lemma}

\begin{proof}
The first statement is exactly Lemma~\ref{lem:iota_vertex}. Let us now characterise the edges of $Y$.

	Consider two vertices $v=gN(H), v'=hN(H')$ of $Y$ that are connected by an edge. The subgroups $gHg^{-1}, hH'h^{-1}$ are infinite cyclic, and we denote by  $z_v, z_{v'}$  the associated generators. The elements $z_v, z_{v'}$ commute by definition of $Y$, and they
generate a $\mathbb{Z}^2$ subgroup of $A_\Gamma$ by Lemma~\ref{lem:no_conj_H}. Since the action of $A_\Gamma$ on $\widehat{D}$ is acylindrical and $\widehat{D}$ is 
CAT(0), this subgroup must fix a point. In particular, the fixed-point sets in $\widehat{D}$ of $z_v$ and $z_{v'}$ have a 
non-trivial intersection. We will need the following standard result from group actions on trees, whose proof we omit,  to show that certain configurations are impossible: 

\medskip

\textbf{Claim:} Let $G$ be a group acting on a simplicial tree $T$ by isometries, let $w, w'$ be two distinct vertices of $T$, and let $g, g'$ be two elements of $G$ such that for all non-zero $k\in \Z$, we have $\mbox{Fix}_T(g^k) = \{w\}$ and $\mbox{Fix}_T((g')^k) = \{w'\}$. Then $g, g'$ generate a non-abelian free subgroup. 

\medskip

 We now consider several cases, depending on the type (dihedral or tree) of the corresponding vertices.

\medskip 
		
Case 1: Suppose by contradiction that $v$ and $v'$ are adjacent vertices of $Y$ of dihedral type. Recall that for an element of the form $z_{ab}$, we have $\mbox{Fix}_{D}(z_{ab}) = \{v_{ab}\}$ by Lemma~\ref{lem:stab_dihedral}. Since $\mbox{Fix}_{\widehat{D}}(z_{v}) \cap \mbox{Fix}_{\widehat{D}}(z_{v'})$ is non-empty, these fixed-point sets intersect along at least (at least) one vertex $w$, which must be of tree type since $\mbox{Fix}_{D}(z_{v}) \cap \mbox{Fix}_{D}(z_{v'})=\varnothing$. This vertex correspond to a standard tree $T$ stabilised by $\langle z_v, z_{v'}\rangle$. Moreover, $T$ contains both vertices $\iota(v)$ and $\iota(v')$: Indeed, $z_v$ stabilises the unique CAT(0) geodesic  between $\iota(v)$ and $T$, and since $z_v$ is not conjugate to a power of a standard generator by Lemma~\ref{lem:no_conj_H}, $z_v$ cannot stabilise an edge of $D$, and it follows that this geodesic is reduced to a point (and similarly for $z_{v'}$). Thus, $\langle z_v, z_{v'}\rangle$ acts on $T$, and since no non-trivial power of $z_v$ or $z_{v'}$ is conjugate to a standard generator by Lemma~\ref{lem:no_conj_H}, it follows that for all non-zero $k\in \Z$, we have $\mbox{Fix}_T(z_v^k) = \{\iota(v)\}$ and $\mbox{Fix}_T((z_{v'})^k) = \{\iota(v')\}$. It now follows from the Claim that $z_v$  and $z_{v'}$ generate a non-abelian free subgroup, a contradiction. 

\medskip 

	Case 2: Let us now assume that $v, v'$ are two adjacent vertices of $Y$ of tree type. By the above argument,  $z_v$ and $z_{v'}$ fix a vertex $w$ of $\widehat{D}$. 
	
	If $w$ is a vertex of $D$, then it belongs to    $\mbox{Fix}_{D}(z_{v}) \cap \mbox{Fix}_{D}(z_{v'})  = T_{\iota(v)}\cap T_{\iota(v')}$. Since distinct standard trees of $D$ meet in at most one vertex by Corollary~\ref{cor:intersect_trees}, we can assume that $w$ is that common vertex. Up to conjugation, we can thus assume that $z_v$ and $z_{v'}$ are two distinct commuting conjugates of standard generators of a dihedral Artin group $A_{ab}$, which contradicts Lemma~\ref{lem:dihedral Z2}.

	\medskip

	Case 3: Finally, we  assume that $v$ is of dihedral type and $v'$ of tree type.  Since $\mbox{Fix}_{\widehat{D}}(z_{v}) \cap \mbox{Fix}_{\widehat{D}}(z_{v'})$ is non-empty,   let $w$ be a vertex of $\widehat{D}$ in that intersection. If $w$ belongs to $D$, then in particular $w$ belongs to $\mbox{Fix}_{D}(z_{v}) \cap \mbox{Fix}_{D}(z_{v'}) = \{\iota(v)\} \cap T_{\iota(v')}$, so $\iota(v)$ is contained in $\mbox{Fix}_{D}(z_{v'})= T_{\iota(v')}$, as required. 
	
	Let us show that this is the only possibility. By contradiction, if $\mbox{Fix}_{\widehat{D}}(z_{v}) \cap \mbox{Fix}_{\widehat{D}}(z_{v'})$ does not contain any vertex of $D$, then these fixed-point sets intersect along (at least) one vertex, which must be of tree type since $\mbox{Fix}_{D}(z_{v}) \cap \mbox{Fix}_{D}(z_{v'})=\varnothing$. This vertex corresponds to a standard tree $T$. Moreover, the same reasoning as in the previous cases shows that $T$ contains $\iota(v)$ and intersects $T_{\iota(v')}$. We thus have that $\langle z_v, z_{v'}\rangle$ induces an action on $T$, $z_v$ fixes the vertex $w\coloneqq {\iota(v)}$, $z_{v'}$ fixes the vertex $w\coloneqq T_{\iota(v')}\cap T$, $w, w'$ are distinct since $\mbox{Fix}_{D}(z_{v}) \cap \mbox{Fix}_{D}(z_{v'})$ is empty by assumption. Moreover, for every non-zero integer $k$, we have $\mbox{Fix}_{T}(z_v^k)=\{w\} $ since no non-trivial power of an element of the form $z_{ab}$ is conjugate to a standard generator by Lemma~\ref{lem:no_conj_H}, and for every non-zero integer $k$, we have that $\mbox{Fix}_{T}(z_{v'}^k)=\{w'\} $ by Lemma~\ref{lem:edge_stab_tree} since $T$ is a standard tree distinct from $T_{\iota(v')}$. It now follows from the Claim that $z_v$  and $z_{v'}$ generate a non-abelian free subgroup, a contradiction. 
\end{proof}

The previous lemma is useful in understand the structure of the graph $Y$. In particular, we have the following results:

\begin{lemma}
	The graph $Y$ does not contain any triangle or square.
	\label{lem:no square}
\end{lemma}

\begin{proof}  Since $Y$ is bipartite with respect to the type of vertices by Lemma~\ref{lem:iota}, it does not contain triangles.

	Let us now show by contradiction that $Y$ does not contain squares. Since $Y$  is bipartite with respect to the type of vertices by Lemma~\ref{lem:iota}, such a square would contain exactly two 
opposite vertices of dihedral type and two opposite vertices of tree type. 

 In $D_\Gamma$, this would correspond to two distinct standard trees intersecting in at least two different vertices, which contradicts Corollary~\ref{cor:intersect_trees}.
\end{proof}

\begin{lemma}\label{lem:Y_conn}
	If $\Gamma$ is connected, then the commutation graph $Y$ is connected.
\end{lemma}

\begin{proof}
	Let us first show that any two vertices of $Y$ of tree type 
	are connected by a path of $Y$. Since $A_\Gamma$ is 
generated by its standard generators, it is enough to show that every pair of vertices of tree type of the form $u_{a}$ 
and $u_{b}$ are connected by a path of $Y$. Let $a$, $b$ be two vertices of $\Gamma$. Since $\Gamma$ is connected, let $c_1 =a, \ldots, c_n=b$ be a sequence of vertices of $\Gamma$ such that $c_i$ 
and $c_{i+1}$ are adjacent for every $i$. Then the sequence of vertices 
$$u_{c_1}, u_{c_1, c_2}, u_{c_2}, \ldots, u_{c_i}, u_{c_i, c_{i+1}}, u_{c_{i+1}}, \dots, u_{c_n}$$
defines a combinatorial path of $Y$ between $u_a$ and $u_b$.  
	
	Moreover, since every vertex of dihedral type of $Y$ is connected by an edge to some vertex of $Y$ of tree type, it 
follows that $Y$ is connected.
\end{proof}

\begin{lemma}\label{lem:cocompact_Y}
	The action of $A_\Gamma$ on $Y$ is cocompact.
\end{lemma}

\begin{proof}
	There are finitely many orbits of vertices of $Y$ by construction, since the family $\calH$ is finite. It follows from Lemma~\ref{lem:iota} and the characterisation of the edges of $Y$ that the set of edges of $Y$ equivariantly embeds into the set of edges of $\widehat{D}$. Since the action of $A_\Gamma$ on $\widehat{D}$ is cocompact, it follows that there are only finitely many $A_\Gamma$-orbits of edges of $Y$, hence the result.  
\end{proof}

Lemma \ref{lem:iota} allows us to extend the map $\iota$ to a map from $Y$ to $\widehat{D}$:

\begin{defn}\label{def:iota}
	 We extend the map $\iota: Y^{(0)} \rightarrow \widehat{D}^{(0)}$  into a  simplicial  and $A_\Gamma$-equivariant map 
$\iota: Y \rightarrow \widehat{D}$ as follows: An edge between a vertex $u$ of dihedral type and a vertex $ u'$ of tree type is sent to the corresponding edge of 
	$\widehat{D}$ between $\iota(u)$ and $\iota(u')$. 
	\end{defn}

\begin{lemma}\label{lem: extending iota}
	Suppose that the graph $\Gamma$ is connected. Then the map $\iota: Y \rightarrow \widehat{D}$ is an $A_\Gamma$- equivariant quasi-isometry. 
More precisely, $\iota$ embeds $Y$ as a coarsely dense subgraph of $\widehat{D}$.
	\label{lem:iota QI}
\end{lemma}

\begin{proof}
	We construct a quasi-inverse $\overline{\iota}: \widehat{D} \rightarrow Y$ as follows. First notice that since 
$\widehat{D}$ has finitely many isometry types of simplices, it is enough to define $\overline{\iota}$ at the level of 
vertices, where the distance between vertices is defined as the length of a minimal path in the $1$-skeleton. 
	
	 By Lemma \ref{lem:iota}, for a vertex $v$ of $\widehat{D}$ that is either of dihedral or tree type, we define 
$\overline{\iota}(v)$ to be the unique vertex $u \in Y$ such that $\iota(u) =v$. Let us define $\overline{\iota}$ on the remaining vertices. A vertex $v$ of 
$\widehat{D}$ corresponding to the coset of a cyclic group generated by a standard generator belongs to a unique standard tree $T$, and we set $\overline{\iota}(v) = \overline{\iota}(v')$, where $v'$ is the vertex of $\widehat{D}$ that is the apex corresponding to $T$. 
For a vertex $v$ corresponding to a coset of the trivial subgroup, we pick a vertex $v'$ of dihedral type of $D$ adjacent to it, and we set $\overline{\iota}(v) = \overline{\iota}(v')$.
	  
	Lemma \ref{lem:iota} implies that $\overline{\iota} \circ \iota$ is the identity on the vertices of $Y$, and that $\iota \circ 
\overline{\iota}$ is the identity on the vertices of $\widehat{D}$ that are either of dihedral type or of tree type. 
For a vertex $v$ of $\widehat{D}$ corresponding either to the coset of a trivial, cyclic, or $\Z^2$-subgroup,  
the construction implies that $\iota \circ \overline{\iota}(v)$ and $v$ are at distance  $1$ in the $1$-skeleton of $D$,  which concludes the 
proof. 
\end{proof}

Note that if $\Gamma$ is disconnected and can be written as the disjoint union of two full subgraphs $\Gamma_1, \Gamma_2$, 
then $A_\Gamma$ splits as the free product $A_{\Gamma_1} * A_{\Gamma_2}$. In particular, since a free product of 
hierarchically hyperbolic groups is itself hierarchically hyperbolic (see \cite[Corollary 8.24]{HHS_II} or~\cite[Theorem 9.1]{HHS_II}), it is enough to consider the case of a connected graph 
$\Gamma$. This motivates the following convention:

\begin{conv}\label{conv:connected}
	In the rest of this article (except the proof of Theorem~\ref{thmintro:main} in Section~\ref{subsec:assembling}), we will assume that the underlying presentation graph $\Gamma$ is connected and not a single vertex.
\end{conv}

\subsection{The graph of proper irreducible parabolic subgroups of finite type}\label{subsec:irreducible_parabolic}

As an aside, we highlight the connection between the commutation graph $Y$ and the graph of proper irreducible parabolic 
subgroups of finite type introduced by Morris-Wright \cite{MorrisWright}, and generalizing a construction in the 
spherical type of Cumplido et al. \cite{CumplidoEtAl} . This graph was proposed as an analogue of the curve graph for all Artin groups. We show that this indeed the 
case for Artin groups of large and hyperbolic type. We emphasise that this subsection is not needed in the rest of the paper and can be omitted in a first reading.

\begin{defn}[{\cite{MorrisWright}}]
The \textbf{graph of irreducible proper parabolic subgroups of finite type} $P$ of  $A_\Gamma$ is  the simplicial  graph defined as 
follows. Vertices correspond to the proper parabolic subgroups of finite type that are \textit{irreducible} (that is, they do 
not decompose as a direct product of proper standard parabolic subgroups). Two vertices $H, H'$ are connected by an edge when either: 
\begin{itemize}
	\item there is a strict inclusion $H \subsetneq H'$, or
	\item we have $H\cap H' = \{1\}$ and $H, H'$ commute.
\end{itemize}
The Artin group $A_\Gamma$ acts on $P$ by conjugation.
\end{defn}

While this definition makes sense for all Artin groups, we recall that we are only dealing in this article (and in particular in this section) with the case of Artin groups $A_\Gamma$ that are large-type and of hyperbolic type.

\begin{prop}\label{prop:irred_parabolic}
	Assume that $\Gamma$ is a connected graph not reduced to a single edge. Then the graph of irreducible parabolic subgroups of finite type of $A_\Gamma$ is equivariantly isomorphic to the commutation graph.
\end{prop}

Under such conditions on $\Gamma$, the action of $A_\Gamma$ on the coned-off Deligne complex is acylindrical and universal by \cite[Theorem A]{MP}. The following is thus a direct consequence of Lemma \ref{lem:iota QI}:

\begin{cor}
	Assume that $\Gamma$ is a connected graph not reduced to a single edge. Then the graph of irreducible proper parabolic subgroups of finite type of $A_\Gamma$ is hyperbolic of infinite diameter, and the action of  $A_\Gamma$ on it is 
	acylindrical and universal.\qed
\end{cor}

In order to prove Proposition~\ref{prop:irred_parabolic}, we will need the following result: 

\begin{lemma}\label{lem:no_conj_norm_H}
	Assume that $\Gamma$ is a connected graph not reduced to a single edge. Then distinct elements of $\calH$ have normalisers that are in different conjugacy classes.
\end{lemma}

\begin{proof}
	The normaliser of a standard generator is of the form $\Z\times F_k$ for $F_k$ a finitely-generated free group by Lemma~\ref{lem:stab_tree}, while the normaliser of an element of the form $z_{ab}$ is equal to the dihedral Artin group $A_{ab}$ by Lemma~\ref{lem:stab_dihedral}.
	We show that these groups are non-isomorphic. This is clear if $k=1$ as a dihedral Artin group with $m_{ab}\geq 3$ is not isomorphic to $\Z^2$ since it virtually contains a non-abelian free group by Lemma~\ref{lem:virtually splits}. For $k\geq 2$, this follows for instance from the description of their abelianisations: The abelianisation of $\Z\times F_k$ is a free abelian group of rank $k+1\geq 3$, while the abelianisation of $A_{ab}$ is generated by two elements since $A_{ab}$ is. Thus, the normaliser of a standard generator and the normaliser of an element of the form $z_{ab}$ are not conjugated.
	 
	Let $z_{ab}$, $z_{a'b'}$ be two elements  associated to two distinct edges of $\Gamma$, and let $g \in A_\Gamma$. By Lemma~\ref{lem:stab_dihedral}, the normalisers of $z_{ab}$ and $gz_{a'b'}g^{-1}$ are the dihedral Artin groups $A_{ab}$ and $gA_{a'b'}g^{-1}$ and their fixed-point sets in $D$ are $v_{ab}$ and $gv_{a'b'}$ respectively. As these points are in different $A_\Gamma$-orbits, it follows that $N(z_{ab})$ and $N(z_{a'b'})$ are not conjugate.
	
	Finally, let $a, b$ two distinct standard generators in $V_{\mathrm{odd}}(\Gamma)$, and suppose by contradiction that $N(a)=gN(b)g^{-1}$ for some $g\in A_\Gamma$. By Lemma~\ref{lem:stab_tree}, the normalisers $a$ and $b$ have the form $\langle a \rangle \times F_k$ and $\langle b \rangle \times F_k$ respectively (note that the rank of the free group factors need to coincide since the subgroups are conjugated). If $k\geq 2$, then since $N(a)$ and $N(b)$ are conjugated, so are their centres $\langle a \rangle$ and $\langle b \rangle$. Using the homomorphism $A_\Gamma\rightarrow \Z$ sending each standard generator to $1$, it follows that $a$ and $b$ are conjugated. By construction of $V_{\mathrm{odd}}(\Gamma)$, it follows that $a=b$, a contradiction.
	
		Assume now that $k\leq 1$. We thus have two standard generators $a, b$ such that $N(a), N(b)$ are isomorphic to $\Z$ or $\Z^2$. We will use the following general claim:
		
		\medskip
		
		\textbf{Claim:}  Let $x$ be a standard generator. If the normaliser  $N(x)$ is isomorphic to $\mathbb Z$, then $x$ corresponds to an isolated vertex of $\Gamma$. If the normaliser $N(x)$ is isomorphic to $\mathbb Z^2$, then $x$ corresponds to a leaf of $\Gamma$.
		
		\medskip
		
		Let us prove the Claim. By Lemma~\ref{lem:stab_tree}, the normaliser $N(x)$ is of the form $\langle x \rangle \times F$ for some finitely-generated group $F$. An explicit basis $\calB$ of $F$ was given in \cite[Remark 4.6]{MP}. In particular, the following holds:
		\begin{itemize}
			\item[(i)] If $x$ is an isolated vertex of $\Gamma$, then $F$ is trivial.
			\item[(ii)] The basis $\calB$ contains the element $z_{xy}$ for every vertex $y$ of $\Gamma$ adjacent to $x$. 
			\item[(iii)] Suppose that $y$ is adjacent to $x$ and the edge of $\Gamma$ between them has an odd label. If $w$ is a vertex of $\Gamma$ adjacent to $y$ and distinct from $x$, then $\calB$ contains a suitable conjugate of $z_{yw}$. We note that such an element cannot be equal to an element described in $(ii)$ by Lemma~\ref{lem:no_conj_H}.
		\end{itemize} 
		
		If $N(x)$ is isomorphic to $\Z$, then $F$ is trivial, and by item $(ii)$ in the above description of a basis of $F$, $x$ must be an isolated vertex. If $N(x)$ is isomorphic to $\Z^2$, then $F$ is isomorphic to $\Z$, and by item $(ii)$ in the above can have at most one neighbour in $\Gamma$. It cannot be an isolated vertex for otherwise $F$ would be trivial by item $(i)$  above. Thus $x$ is a leaf of $\Gamma$. This proves the Claim.
		
		\medskip
		
		Since $N(a), N(b)$ are isomorphic to $\Z$ or $\Z^2$, it follows from the above Claim that $a$ and $b$ are leaves of $\Gamma$ (since $\Gamma$ is connected by assumption). Moreover, since $N(a), N(b) \cong \Z^2$ and $\Gamma$ is not reduced to a single edge, item $(iii)$ above implies that the corresponding edges of $\Gamma$ containing $a$ and $b$ have even labels. 

		 By Lemma~\ref{lem:stab_tree}, $N(a)=gN(b)g^{-1}$ stabilises the standard tree containing the vertex $\langle a \rangle$ and the standard tree containing $g \langle b \rangle$. Because $a, b$ are leaves of $\Gamma$ with corresponding edges having even labels, the corresponding standard trees $T_a$ and $gT_b$ have a very simple structure, which is a direct consequence of \cite[Lemma~4.3]{MP} (see also \cite[Example~4.7]{MP} for a concrete example): Let $a'$ be the unique neighbour of $a$ in $\Gamma$, and let $e$ be the edge of $D$ between the parabolic subgroups $\langle a \rangle$ and $A_{aa'}$. Then $T_a$ is the union of the $\langle z_{aa'}\rangle$-translates of $e$ (such a description was for instance given in \cite[Example 4.7]{MP}). In particular, $T_a$ (and similarly for $gT_b$) is isomorphic to a cone over a countably infinite discrete set.
		 
		 Let us show that the standard trees $T_a$ and $gT_b$ intersect. If these two convex trees were disjoint, then $N(a)=gN(b)g^{-1}$ would also stabilise the unique CAT($-1$) geodesic of $D$ between them, and so both normalisers would be contained in an edge-stabiliser or triangle-stabiliser. This is impossible as  such stabilisers do not contain $\Z^2$, as they are infinite cyclic or trivial respectively. 
		 
		 Thus these two standard trees meet.  The vertices in $T_a$ that are not the central vertex of dihedral type correspond to cosets of $\langle a \rangle$ (and we have an analogous result for $b$). Since $\langle a  \rangle \neq \langle b \rangle$ by assumption, the trees $T_a$ and $gT_b$ necessarily meet at the central vertex of dihedral type. Let $a'$ (respectively $b'$) be the unique neighbour of $a$ (respectively $b$) in $\Gamma$. The unique vertex of $T_a$ of dihedral type corresponds to a coset of $A_{aa'}$. Similarly, the unique vertex of $gT_b$ of dihedral type corresponds to a coset of $A_{bb'}$. As these vertices agree, it follows that $\{a, a'\}= \{b, b'\}$, and since $a \neq b$ this implies $a$ and $b$ are adjacent in $\Gamma$. Since $a$ and $b$ are leaves of $\Gamma$ by the above, it follows that $\Gamma$ is a single edge, contradicting our assumption on $\Gamma$. 
\end{proof}

\begin{lemma}\label{lem:irred_parabolic_vertex}
	Assume that $\Gamma$ is a connected graph not reduced to a single edge.  Consider the  map $\varphi: Y^{(0)} \rightarrow P^{(0)}$ defined as follows: 
	\begin{itemize}
		\item For a vertex $gA_{ab}$ of $Y$ of dihedral type, we set $\varphi(gA_{ab}) = gA_{ab}g^{-1}$.
		\item For a vertex $gN(a)$ of $Y$ of tree type, we set $\varphi(gN(a)) = g\langle a \rangle g^{-1}$.
	\end{itemize}
Then $\varphi$ is well-defined and realises an $A_\Gamma$-equivariant bijection between the vertex sets of $Y$ and $P$.
\end{lemma}

\begin{proof}
	The map $\varphi$ is well-defined on vertices of dihedral type, but we need to check that the definition makes sense for vertices of tree type. It is enough to show that if $a, b$ are standard generators in $V_{\mathrm{odd}}(\Gamma)$ with $N(a)=N(b)$, then $a = b$. This is a direct consequence of Lemma~\ref{lem:no_conj_norm_H}.
	
	Let us show that $\varphi$ is surjective. It is clear from the construction that every conjugate of the form $gA_{ab}g^{-1}=\varphi(gA_{ab})$ is in the image of $\varphi$. Let $a$ be a standard generator of $A_\Gamma$. By construction of $V_{\mathrm{odd}}(\Gamma)$, there exists a standard generator $b \in V_{\mathrm{odd}}(\Gamma)$ and an element $h \in A_\Gamma$ such that $a = hbh^{-1}$. Thus, any conjugate of the form $g\langle a \rangle g^{-1}$ can be written as $g\langle a \rangle g^{-1} = (gh)\langle b \rangle (gh)^{-1}=\varphi(ghN(b)).$ Thus, $\varphi$ is surjective.
	
	Let us show that $\varphi$ is injective. For a given standard generator $a \in V(\Gamma)$, the map $\varphi$ induces a bijection between the conjugates of $\langle a \rangle$ and the cosets of $N(a)$. Similarly, since dihedral parabolic subgroups of $A_\Gamma$  are self-normalizing by Lemma~\ref{lem:self-normalising}, we have that for a given edge $a, b$ of $\Gamma$ with $m_{ab}\geq 3$, the map $\varphi$ induces a bijection between the conjugates of $A_{ab} $ and the cosets of $A_{ab}$. To show that $\varphi$ is injective, it remains to show that for $H, H'$ different elements of $\calH$, we cannot have an equality of the form $\varphi(gN(H))=\varphi(g'N(H'))$ for $g, g' \in A_\Gamma$. Such an equality would amount to an equality  between  parabolic subgroups. Note that a  conjugate of the form $g\langle a \rangle g^{-1}$, which is isomorphic to $\Z$, cannot be isomorphic to a conjugate of a dihedral parabolic subgroup by Lemma~\ref{lem:virtually splits}. There are thus two cases to consider: 
	
	If $H, H'$ correspond to cyclic subgroups of the form $\langle a \rangle, \langle a'\rangle$ respectively, then the equality $\varphi(gN(H))=\varphi(g'N(H'))$ yields the equality $g\langle a \rangle g^{-1} = g' \langle a' \rangle (g')^{-1}$. Since $a\neq a'$, this equality is impossible by Lemma~\ref{lem:no_conj_H}. 
	
	If $H, H'$ correspond to cyclic subgroups of the form $\langle z_{ab} \rangle, \langle z_{a'b'}\rangle$ respectively, then the equality $\varphi(gN(H))=\varphi(g'N(H'))$ and the fact that $A_{ab}=N(z_{ab})$ by Lemma~\ref{lem:stab_dihedral} yield the equality $gN(z_{ab})g^{-1} = g'N(z_{a'b'})(g')^{-1}$. Since $H\neq H'$, this equality is impossible by Lemma~\ref{lem:no_conj_norm_H}.
\end{proof}

\begin{proof}[Proof of Proposition~\ref{prop:irred_parabolic}]
We know from Lemma~\ref{lem:irred_parabolic_vertex} that the map $\varphi$ realises an $A_\Gamma$-equivariant bijection between the vertex sets of $Y$ and $P$.  Let us show that this bijection extends to an equivariant isomorphism between $P$ and $Y$.

 First, it follows from Lemma~\ref{lem:iota} that two vertices that are adjacent in $Y$ are also adjacent in $P$, and we now show the converse. 

 Let $Q$ and $Q'$ be two proper irreducible parabolic subgroups of finite type, i.e. parabolic subgroups on one or two generators with $m_{ab}\geq 3$, and assume that $Q$ and $Q'$ are connected in $P$. There are two cases to consider, depending on the two types of edges of $P$: 
 
 Let us first assume that $Q, Q'$ intersect trivially and commute. Since $A_\Gamma$ is of cohomological dimension 2 by \cite{CharneyDavis} and dihedral parabolic subgroups contain a copy of $\Z^2$, this can only happen if both $Q$ and $Q'$ are infinite cyclic, in which case there is a $Y$-edge between them by construction of $Y$. 
 
 Suppose now that $Q \subsetneq Q'$. First notice that $Q$ and $Q'$ cannot be both dihedral parabolic subgroups by Lemma~\ref{lem:self-normalising}.  
 We can thus assume that $Q$ is infinite cyclic. Let $e$ be an edge in the standard tree $T$ associated to $Q$, and let $ \sigma'$ be a maximal simplex of $\mathrm{Fix}(Q')$ (i.e. $\sigma'$ is a vertex or an edge depending on whether $Q'$ is of dihedral type or infinite cyclic respectively). Since $Q$ fixes both $v$ and $\sigma'$ by assumption, Lemma~\ref{cor:intersect_stab_nontrivial} implies that $v$ and $\sigma'$ are contained in a common standard tree, namely the standard tree $T$. 
 
 Let us show by contradiction that $Q'$ cannot be infinite cyclic.  Let $T'$ be the standard tree corresponding to $Q'$. Then since $T$ and $T'$  both contain $e$, it follows from  Corollary~\ref{cor:intersect_trees} that $T=T'$. Since $Q, Q'$ is the pointwise stabiliser of $T, T'$ respectively by construction, we get $Q=Q'$, contradicting the assumption that $Q \subsetneq Q'$. 
 
 Thus, $Q'$ is a parabolic subgroup of dihedral type. We thus have that the vertex of $D$ corresponding to $Q'$ is contained in the standard tree associated to $Q$. By Lemma~\ref{lem:iota}, this implies that $Q$ and $Q'$ are connected by an edge of $Y$.  
\end{proof}

\begin{rem}
	We note that the alternative description of $P$ given by Proposition~\ref{prop:irred_parabolic} provides an alternative characterisation of the edges of $P$: Two proper irreducible parabolic subgroups of finite type are joined by an edge of $P$ if and only if their centres commute. This generalises results known for Artin groups of finite type \cite[Theorem 2.2]{CumplidoEtAl} and of FC type \cite[Lemma 4.2]{MorrisWright}.
\end{rem}

\subsection{Adjacency in the commutation graph and  intersection of neighbourhoods of cosets} 

In this section, we study the interactions between cosets associated to adjacent vertices of the commutation graph $Y$. Namely, we show that two such adjacent cosets have neighbourhoods that intersect along a quasi-flat (see Lemma~\ref{lem:intersection_of_cosets}). We start with the following result:

\begin{lemma}\label{lem:intersection_normalisers}
	Let $H_1,H_2\in\mathcal H$ and suppose that $gN(H_1),hN(H_2)$ are adjacent vertices of $Y$ for some $g,h\in A_\Gamma$. Then we have 
	$$ gN(H_1)g^{-1}\cap hN(H_2)h^{-1} = \langle gH_1g^{-1},hH_2h^{-1}\rangle .$$
	Moreover, $\langle gH_1g^{-1},hH_2h^{-1}\rangle$ is naturally isomorphic to $gH_1g^{-1}\times hH_2h^{-1}\cong \mathbb Z^2$.
\end{lemma}

\begin{proof}
	By Lemma~\ref{lem:iota}, we can assume that 
	$gN(H_1)$ is mapped under $\iota$ to the apex of a standard tree $T_1$  and $hN(H_2)$ is mapped under $\iota$ to a vertex of dihedral type contained in $T_1$. Notice that, up to a conjugation, it is enough to prove the desired equality when $h$ is trivial, so we will assume that this is the case in the rest of the proof. The subgroup $H_2\in \calH$ is of the form $\langle z_{ab} \rangle$  for some standard generators $a, b$, and so the vertex $v$ corresponds to the trivial coset $A_{ab}$ by construction of $\iota$. Since the standard tree $T_1$ contains the vertex $A_{ab}$, it also contains a vertex of the form $x\langle a \rangle$ or $x \langle b \rangle$ for some $x\in A_{ab}$. Without loss of generality, we can  assume that  $T_1$ contains a vertex of the form $x\langle a \rangle$ for some chosen $x\in A_{ab}$. 
	
	We have $hN(H_2)h^{-1} = N(z_{ab})= A_{ab}$ by Lemma~\ref{lem:stab_dihedral}. Moreover, since two edges of $T_1$ have the same stabiliser by construction, we have $gH_1g^{-1} = x\langle a \rangle x^{-1}$, hence $gN(H_1)g^{-1} = xN(a)x^{-1} = xC(a)x^{-1}$, the latter equality following from Lemma~\ref{lem:stab_tree}. Thus,  the intersection $gN(H_1)g^{-1}\cap hN(H_2)h^{-1}$ can be identified with the centraliser of the element $xax^{-1}$ in $A_{ab}$. By \cite[Lemma 7.$(ii)$]{Crisp}, this is equal to $\langle xax^{-1}, z_{ab} \rangle = \langle gH_1g^{-1}, hH_2h^{-1} \rangle$. 

For the ``moreover'' part, note that the intersection is naturally isomorphic to the product by definition of edges of $Y$ and Lemma~\ref{lem:no_conj_H}.
\end{proof}

\begin{defn}\label{def:coset_neighbourhood}
	 	Let $B(e,r)$ denote the closed ball of radius $r$ in the standard Cayley graph of $A_\Gamma$.
	 	
	 	We fix a constant $r\geq 0$  so that whenever $a,b$ are conjugate standard generators there exists $x\in B(e,r)$ so that $xax^{-1}=b$. The $r$-\textbf{neighbourhood} $N(H)^{+r}$ of $N(H)$ is defined as $$N(H)^{+r}\coloneqq \bigcup_{c\in B(e,r)}N(H)c.$$
\end{defn}

\begin{lemma}\label{lem:intersection_of_cosets}
	There exists a constant $\ell\geq0$ such that the following holds:
	
	Let $H_1,H_2\in\mathcal H$ and suppose that $gN(H_1),hN(H_2)$ are adjacent vertices of $Y$ for some $g,h\in A_\Gamma$.  

	Then $gN(H_1)^{+r}\cap h N(H_2)^{+r}$ is nonempty. Also, the subgroup $gN(H_1)g^{-1}\cap hN(H_2)h^{-1}$ (which is  isomorphic to  
	$\Z^2$ by Lemma~\ref{lem:intersection_normalisers}) acts coboundedly on the  intersection $gN(H_1)^{+r}\cap h N(H_2)^{+r}$, and the quotient of this action has diameter at most $\ell$.
\end{lemma}

\begin{rem}
	We advise the reader to focus on the case of \textit{even} Artin groups (i.e. $m_{ab}$ is even for all standard generators) in a first reading. In that particular case, we have $V_\mathrm{odd}(\Gamma) = V(\Gamma)$ and we can take $r=\ell=0$ in the above statement. In particular,  we are simply looking at intersections of cosets.
	
	Let us explain briefly the need to consider neighbourhoods of cosets in the general case. If $x$ and $y$ are two standard generators connected by an edge of $\Gamma$ with an odd label, then we have $\Delta_{xy}x\Delta_{xy}^{-1}=y$, where $\Delta_{xy}$ is the Garside element of $A_{xy}$. In particular, we have $g\Delta_{xy}N(x)\Delta_{xy}^{-1} = gN(y)$, and so the cosets $g\Delta_{xy}N(x)$ and $gN(y)$ are at bounded Hausdorff distance of one another (they are in a sense ``parallel''). Let us assume that we are looking at vertices $gN(H)$ and $hN(H')$ where $H, H'$ correspond to standard generators $a, b$. As we saw in the proof of  Lemma~\ref{lem:intersection_normalisers}, the intersection $gN(H)g^{-1}\cap hN(H')h^{-1}$ stabilises a vertex of dihedral type. However, while $H, H'$ correspond to standard generators $a, b$ in $V_{\mathrm{odd}}(\Gamma)$, it may happen that this vertex of dihedral type correspond to two different standard generators $a', b'$ that are conjugated to $a, b$ respectively. Thus, while it may not be clear how to study the intersection $gN(a)\cap hN(b)$, it is much simpler to study the corresponding intersection of parallel copies corresponding to cosets of $N(a')$ and $N(b')$. This is why we do not simply consider the cosets $gN(a), hN(b)$ but consider suitable neighbourhoods that contain those parallel copies corresponding to cosets of $N(a')$ and $N(b')$. 
\end{rem}

\begin{proof}[Proof of Lemma~\ref{lem:intersection_of_cosets}]
	Let $H_1,H_2\in\mathcal H$ and let $gN(H_1),hN(H_2)$ be cosets corresponding to adjacent vertices of $Y$. 
  
	By Lemma \ref{lem:iota}, we can assume that $H_1$ correspond to a standard generator $a_1$, $H_2$ corresponds to an element of the form $z_{b_1 b_2}$, and the standard tree $T_1$ with apex $\iota(gN(H_1))$ contains the vertex of dihedral type $\iota(hN(H_2))$. This vertex corresponds to a coset of the form $hA_{b_1b_2}$.   Since the tree $T_1$ contains the vertex $hA_{b_1b_2}$, it also contains a vertex corresponding to a coset of the form $hx \langle b_1\rangle$ or $hx \langle b_2\rangle$  for some $x\in A_{b_1b_2}$, so without loss of generality we will assume that it contains a vertex corresponding to a coset of the form $hx \langle b_1\rangle$.  Since $N(H_2)=A_{b_1b_2}$ by Lemma~\ref{lem:stab_dihedral}, we have that $hxN(H_2)=hN(H_2)$, so without loss of generality we will assume, up to replacing the element $h$ by $hx$, that $T_1$ contains the vertex $h\langle  b_1\rangle$. Moreover, since $T_1$ contains vertices corresponding to cosets of $\langle a_1\rangle$ and $\langle b_1\rangle$, it follows from Lemma~\ref{lem:tree_label_conjugate} that $a_1$ and $b_1$ are conjugated: By construction of $r$, we choose an element $x_1 \in B(e, r)$ so that $x_1a_1x_1^{-1}=b_1$.  
	
	By construction of $x_1$, we have $hx_1v_{a_1}=hv_{b_1}$. Thus, the apex of the standard tree $T_1$ is the coset $gN(a_1) = hx_1N(a_1)$, so we can assume without loss of generality that $g = hx_1$. We thus have $gN(H_1)=hx_1N(a_1)$ and  $hN(H_2)=hA_{b_1b_2}$. Since $x_1 \in B(e,r)$, the  intersection $gN(H_1)^{+r}\cap hN(H_2)^{+r} = hx_1N(a_1)^{+r}\cap hN(z_{b_1b_2})^{+r}$ contains the intersection $$hx_1N(a_1)x_1^{-1}\cap hN(z_{b_1b_2}) = hN(b_1)\cap hN(z_{b_1b_2}) = h\langle b_1, z_{b_1b_2} \rangle,$$
	the latter equality following from Lemma~\ref{lem:intersection_normalisers}.

	Since $B(e,r)$ contains only finitely many elements of $A_\Gamma$ and $\calH$ is finite, it follows from \cite[Proposition~9.4]{Hruska} that there exists a constant $\ell>0$ (that does not depend on the choice of standard generators $a_1, b_1, b_2$ and element $x_1$) such that 
	$$x_1N(a_1)^{+r}\cap N(z_{b_1b_2})^{+r} \subset \big(N(x_1a_1x^{-1})\cap N(z_{b_1b_2})\big)^{+\ell} = \big(N(b_1)\cap N(z_{b_1b_2})\big)^{+\ell} = \langle b_1, z_{b_1b_2} \rangle^{+\ell},$$
	and in particular 
	$$gN(H_1)^{+r}\cap hN(H_2)^{+r} =hx_1N(a_1)^{+r}\cap hN(z_{b_1b_2})^{+r} \subset h\langle b_1, z_{b_1b_2} \rangle^{+\ell}.$$
	Thus, the intersection $gN(H_1)^{+r}\cap hN(H_2)^{+r} $ is at Hausdorff distance at most $\ell$ from $h\langle b_1, z_{b_1b_2} \rangle$. 
	
	The conjugate $h\langle b_1, z_{b_1b_2} \rangle h^{-1}$ is equal to $  \langle ga_1g^{-1}, hz_{b_1b_2}h^{-1} \rangle = gN(H_1)g^{-1}\cap hN(H_2)h^{-1}$ by Lemma~\ref{lem:intersection_normalisers}. Since $k\langle b_1, z_{b_1b_2} \rangle k^{-1}$ acts coboundedly on $k\langle b_1, z_{b_1b_2} \rangle$ and $gN(H_1)^{+r}\cap hN(H_2)^{+r} $ is at Hausdorff distance  at most $\ell$ from $k\langle b_1, z_{b_1b_2} \rangle$, it follows that $gN(H_1)g^{-1}\cap hN(H_2)h^{-1}$ acts coboundedly on $gN(H_1)^{+r}\cap hN(H_2)^{+r} $ and the quotient space has diameter at most $\ell$.
\end{proof}

\begin{defn}\label{def:incident_direction}
	Let $gN(H)$ be a vertex of the commutation graph $Y$.  For each $Y$--adjacent coset $hN(H')$, the $hH'h^{-1}$-orbit of a point of $gN(H)^{+r}\cap hN(H')^{+r}$ is called an \textbf{incident direction} at $gN(H)$.
	
	 Fixing $gN(H)$ and letting $hN(H')$ vary over the vertices of  $Y$ that are adjacent to $gN(H)$, we 
	denote by $\mathcal I(gN(H))$ the set of all 
	such incident directions.  
\end{defn}

We advise the reader to look at Figure~\ref{fig:incident_cosets} for an illustration of this notion.  For later use, we observe:

\begin{lemma}\label{lem:I_finite}
	For each vertex $gN(H)$ of the commutation graph $Y$, there are finitely many $gN(H)g^{-1}$--orbits of  incident directions at $gN(H)$.
\end{lemma}

\begin{proof}
	By Lemma~\ref{lem:link_tree_cocompact} and Corollary~\ref{cor: link coneoff}, there are finitely many $gN(H)g^{-1}$-orbits of vertices of $Y$ adjacent to the vertex $gN(H)$. Thus, it is enough to show that for a fixed vertex $hN(H')$ adjacent to $gN(H)$, there are only finitely many $gN(H)g^{-1}$-orbits of incident directions corresponding to $hH'h^{-1}$-orbits. This is indeed the case, since $gN(H)g^{-1}$ contains $gN(H)g^{-1}\cap hN(H')h^{-1}$, which acts coboundedly, hence cofinitely (since the Cayley graph of $A_\Gamma$ is locally finite) on the  intersection $gN(H)^{+r}\cap hN(H')^{+r}$ by Lemma~\ref{lem:intersection_of_cosets}. 
\end{proof}

For later purposes we also note the following consequence:

\begin{cor}\label{cor:Y_orbits}
 There are finitely many orbits of edges in $Y$.
\end{cor}

\section{The blown-up commutation graph}\label{sec:blowup}

We now introduce a variant of the commutation graph of a group. This is this complex $X$ to which we apply Theorem 
\ref{thm:main 
criterion} for Artin groups of large and hyperbolic type. 

As per Convention \ref{conv:connected}, we are considering an Artin group of large and hyperbolic type $A_\Gamma$ with $\Gamma$ connected, and let $\mathcal H$ be as in Definition~\ref{defn:fancy_H}. 
 For each $H\in\mathcal H$, let $N(H)$ denote 
its normaliser.  Recall that $\bigsqcup_{H\in\mathcal H}A_\Gamma/N(H)$ is the vertex set of the commutation graph $Y$. Our goal is to produce $X$ by modifying  $Y$.  Roughly speaking, each vertex of $Y$ will 
be ``blown up'' to a cone on a certain discrete space quasi-isometric to $\mathbb Z$, and each edge of $Y$ will be blown up 
to the simplicial join of the blow-ups of its vertices.  Figure~\ref{fig:blowup} and the discussion following 
Definition \ref{def:proj_blowup} may help the reader visualise the construction.

In what follows, we equip $A_\Gamma$ with the word metric $\dist_{A_\Gamma}$ coming from the standard generating set given 
by the vertices of $\Gamma$.  When we refer to distances in subspaces of $A_\Gamma$ (e.g. cosets of subgroups), we are using 
the subspace metric inherited from $(A_\Gamma,\dist_{A_\Gamma})$.

\subsection{Blow-up data}\label{subsec:blow_up_data}
The action of $A_\Gamma$ on the commutation graph $Y$ hides too much information to use it as the underlying simplicial complex in a combinatorial HHS structure.  For example, 
subgroups of the form $N(H)$, $H\in\mathcal H$, fix vertices in the commutation graph, so by Lemma~\ref{lem:intersection_normalisers}, edges have $\integers^2$ 
stabilisers.  By Lemma~\ref{lem:no square}, edges are maximal simplices.  Thus any $Y$--graph (see Definition \ref{def:aug_X}) has $\integers^2$ vertex 
stabilisers, and so it is not a quasi-isometry model for $A_\Gamma$.  We remedy this by replacing the commutation graph by a complex $X$ in which the vertices of the 
commutation graph have been ``blown up''. 

We now describe the geometric data necessary to perform such a blow-up construction and 
show that it exists in the case of Artin groups of large and hyperbolic type. 

\begin{defn}[Blow-up data]\label{defn:blow_up_data}
Let $B_0$, $L_0\in\naturals$ be fixed constants, and let $r$ be as in Definition \ref{def:coset_neighbourhood}.  Then \emph{blow-up data} for $A_\Gamma$ and $Y$ consists of the following:
\begin{itemize}
 \item \textbf{(Quasilines associated to cosets.)}  For each coset $gN(H),H\in\mathcal H,g\in A_\Gamma$, let $\Lambda_{gN(H)}$ be a discrete metric space that is $(B_0,B_0)$--quasi-isometric to 
$\mathbb Z$.  Moreover, $gN(H)g^{-1}$ acts by isometries on $\Lambda_{gN(H)}$ with at most $B_0$ orbits of points, fixing the Gromov 
boundary of $\Lambda_{gN(H)}$ pointwise.  
 
 \item \textbf{(Projections to quasilines.)}  For each coset $gN(H),H\in\mathcal H,g\in A_\Gamma$, let
 $$\phi_{gN(H)}:gN(H)^{+r}\to\Lambda_{gN(H)}$$
 be a $gN(H)g^{-1}$--equivariant, $(L_0,L_0)$--coarsely 
Lipschitz map. 
\end{itemize}

Blow-up data must satisfy the following additional conditions:

\begin{enumerate}[(A)]
 \item \textbf{(Equivariance.)}  For each $g\in A_\Gamma$ and each coset $hN(H)$, there is an isometry $g:\Lambda_{hN(H)}\to\Lambda_{ghN(H)}$ such that $\phi_{ghN(H)}(gx)=g\cdot\phi_{hN(H)}(x)$ for 
all $x\in\Lambda_{hN(H)}$.\label{item:phi_equivariance}

\item \textbf{(Unbounded orbits.)}  For each coset $gN(H)$, the action of $gHg^{-1}$ on $\Lambda_{gN(H)}$ has unbounded 
orbits.\label{item:unbounded_orbits}

\item \textbf{(Bounded incident directions.)}  For each incident direction $hH'h^{-1}\cdot x\in\mathcal I(gN(H))$, we have $\diam(\phi_{gN(H)}(hH'h^{-1}\cdot x))\leq 
B_0$.\label{item:bounded_cosets} (Incident directions are defined in Definition \ref{def:incident_direction}.)
\end{enumerate}
\end{defn}

Some intuition for the definition of blow-up data is provided by Figure~\ref{fig:incident_cosets}.

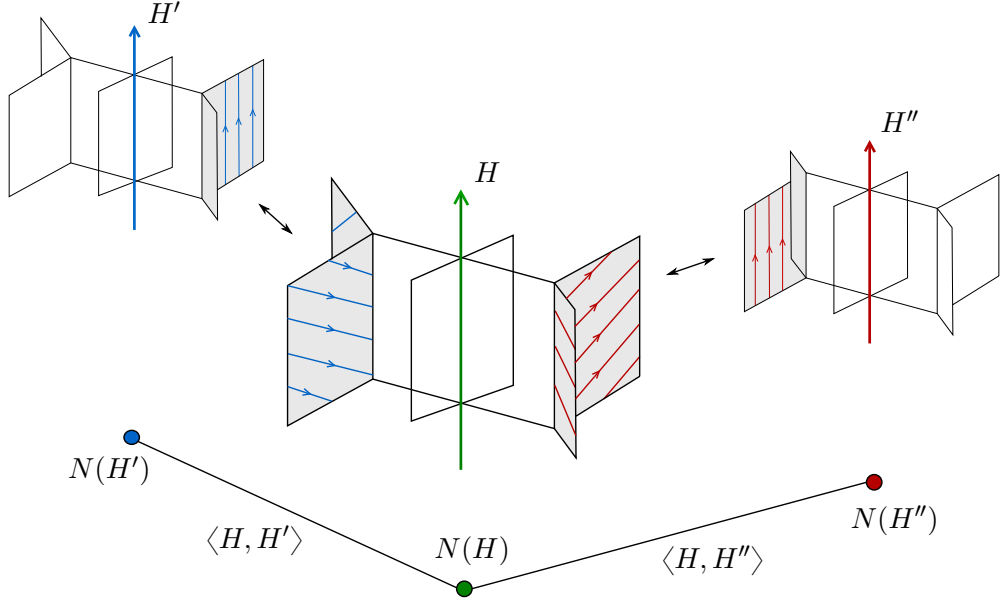
\begin{figure}[h]

	\begin{overpic}[width=0.85\textwidth]{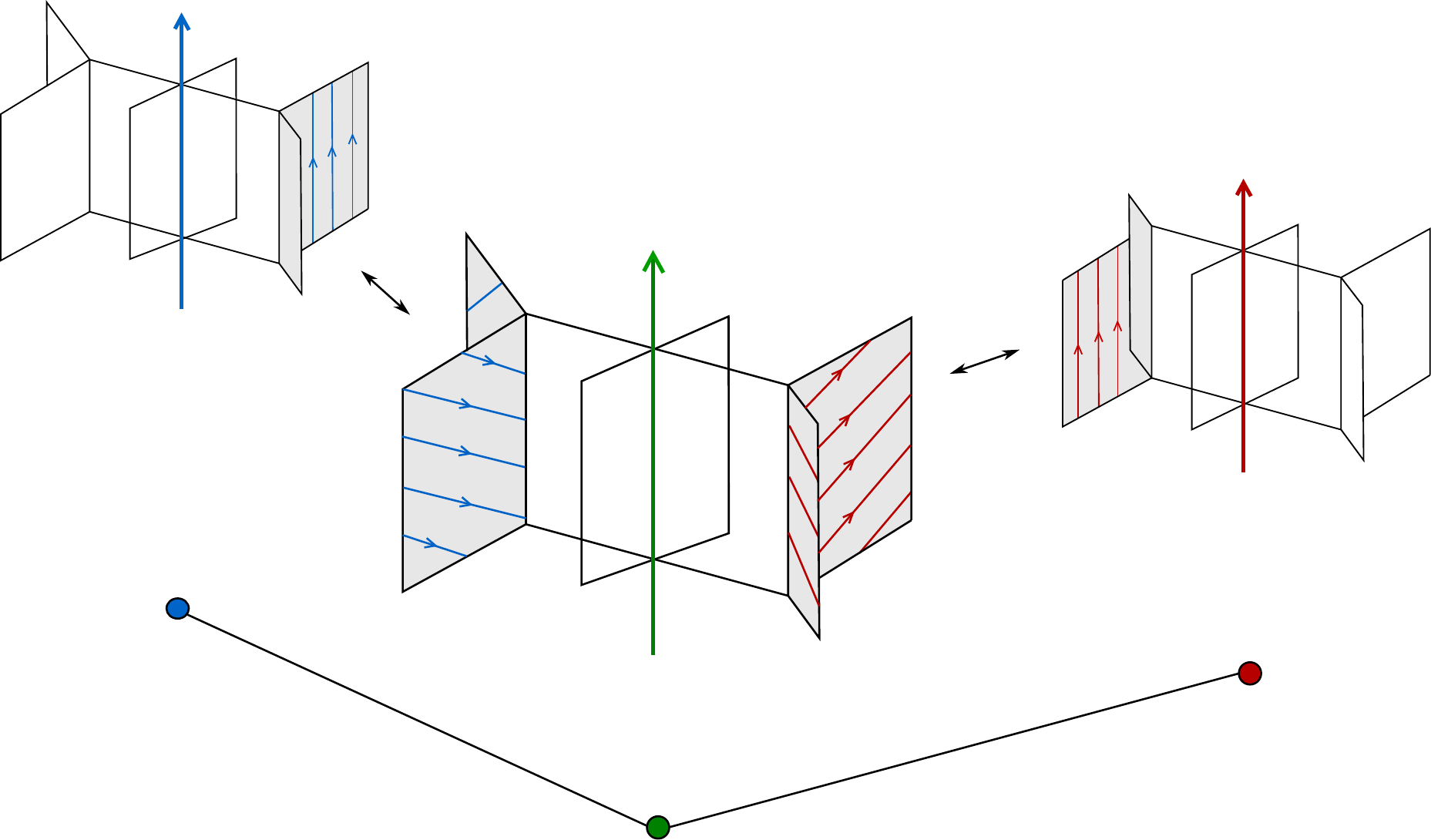}
		\put(14,58){$H'$}
		\put(47,42){$H$}
		\put(88,47){$H''$}
		\put(06,12){$N(H')$}
		\put(43,04){$N(H)$}
		\put(85,07){$N(H'')$}
		\put(20,05){$\langle H, H' \rangle$}
		\put(66,03){$\langle H, H'' \rangle$}
	\end{overpic}
	
\caption{A simplified picture illustrating two adjacent edges of $Y$, with above each vertex a depiction of the corresponding cosets $N(H), N(H'), N(H'')$. Each such coset is quasi-isometric to a product of a tree with a line, and the directions $H, H', H''$ are represented vertically in each coset. Two cosets $N(H)$ and $N(H')$ that are adjacent in $Y$ have uniform neighbourhoods that meet along a subset with a cobounded action of the $\integers^2$-subgroup $H\times H'$ (shaded regions in the picture).  The $H'$- and $H''$-orbits in the corresponding shaded regions have a
certain ``slope''  when seen inside $N(H)$.  The action of $N(H)$ on the  quasiline $\Lambda_{N(H)}$ from the blow-up data is chosen so that $H$ acts with unbounded orbits on it, 
but $H'$ and $H''$ act with bounded orbits on $\Lambda_{N(H)}$. }
\label{fig:incident_cosets}
\end{figure}

The reader may want to read the statement of Lemma~\ref{lem:bounded_sets} already, in order to understand the importance of the previous conditions, especially condition (C). 
While the equivariance condition (A) will be necessary for our constructions, it is slightly redundant in the previous definition, as we can start from a ``non-equivariant'' blow-up data and extend it equivariantly:

\begin{lemma}\label{lem:extending_equivariantly}
	Suppose that for every $H \in \mathcal H$, we have:
	
	\begin{itemize}
		\item  a discrete metric space $\Lambda_{N(H)}$ that is $(B_0,B_0)$--quasi-isometric to $\Z$, with an action of $N(H)$ on it by isometries with at most $B_0$ orbits of points and which fixes the Gromov 
	boundary pointwise.  
	\item a $N(H)$--equivariant, $(L_0,L_0)$--coarsely 
	lipschitz map $$\phi_{N(H)}:N(H)^{+r}\to\Lambda_{N(H)}$$ 
\end{itemize}

that satisfy the following additional conditions:

\begin{enumerate}
\item[(B')] For each $H \in \mathcal H$, the action of $H$ on $\Lambda_{N(H)}$ has unbounded 
orbits.
\item[(C')] For each coset $kH'k^{-1}\cdot x\in\mathcal I(N(H))$, we have $\diam(\phi_{N(H)}(kH'k^{-1}\cdot x))\leq 
B_0$.
\end{enumerate}
Then it is possible to extend equivariantly this data into a blow-up data for $Y$. 
\end{lemma}

\begin{proof}
	Fix a favourite representative $g$ of the coset $gN(H)$.  Let $\Lambda_{gN(H)}$ be a copy of $\Lambda_{N(H)}$. We define $\phi_{gN(H)}(gx)=\phi(x)$ for $x\in N(H)^{+r}$, and the action of $gN(H)g^{-1}$ on $\Lambda_{gN(H)}$ is defined as $gng^{-1}\cdot \lambda=n\lambda$ for $n\in N(H)$ and $\lambda\in\Lambda_{gN(H)}$. (Note that the coset representative $g$ is fixed throughout.)
	
	We have that $\Lambda_{gN(H)}$ is 
	$(B_0,B_0)$--quasi-isometric to $\integers$, $gN(H)g^{-1}$ acts by isometries, it has at most $B_0$ orbits, and the orbits of $gHg^{-1}$ are unbounded. The fact that $\phi_{gN(H)}$ is coarsely Lipschitz is also straightforward.
	
	Let us check equivariance (which is not immediate as we chose coset representatives to define the $\phi$ maps). Let $g\in A_\Gamma$, and consider a coset $hN(H)$. We can assume that $h$ is the favourite representative of its coset, and let $ghn$ be the favourite representative of $ghN(H)$, where $n\in N(H)$. The isometry $g:\Lambda_{hN(H)}\to\Lambda_{ghN(H)}$ is defined as $g(x)=n^{-1} x$ (so that we are using the action of $N(H)$ on $\Lambda_{N(H)}$). For $x\in \Lambda_{hN(H)}$ we have
	
	\begin{eqnarray*}
	\phi_{ghN(H)}(gx) &= & \phi_{ghN(H)}(ghn\ n^{-1}h^{-1}x)\\
	 &=&\phi_{N(H)}(n^{-1}h^{-1}x)\\
	&=& n^{-1} \phi_{N(H)}(h^{-1}x) \\
	&=& n^{-1} \phi_{hN(H)}(h h^{-1}x) \\
	&=& n^{-1} \phi_{hN(H)}(x),\\
	\end{eqnarray*}
	as required.
	
	To show \ref{item:bounded_cosets}, we note that by equivariance, for $x\in gN(H)^{+r}$ and $hN(H')$ adjacent to $gN(H)$ in $Y$, we have $\phi_{gN(H)}(hH'h^{-1}\cdot x)=g\phi_{N(H)}( g^{-1} (h H'h^{-1}\cdot x))$. 
	
	Note that $g^{-1} (h H'h^{-1}\cdot x)=g^{-1} h H'h^{-1}g\cdot (g^{-1}x)\in N(H)^{+r}$ and $g^{-1} h H'h^{-1}g$ commutes with $H$, so that $N(H)$ is adjacent to $g^{-1}hN(H')$. In particular $g^{-1} h H'h^{-1}\cdot x$ is in $\mathcal I(N(H))$, so that $\phi_{N(H)}( g^{-1} (h H'h^{-1}\cdot x))$ has diameter at most $B_0$. Since $g$ is an isometry, the same holds for $\phi_{gN(H)}(hH'h^{-1}\cdot x)$, as required.
\end{proof}

\paragraph{\textbf{Existence of blow-up data.}} We now verify that blow-up data exists for $A_\Gamma$ and $Y$. The rest of this subsection is independent of the rest of the article, as  the arguments in 
the next sections rely on the existence of blow-up data, but 
not on the specific choice of blow-up data.  The reader may thus want to skip the rest of this subsection in a first reading.

As we will see, the key idea is to transform quasimorphisms into actions on 
quasilines. We use a lemma from \cite{ABO} to do so (the idea behind it can be traced further back, see e.g. \cite[Proposition 4.4]{Manning-quasitrees}).  We need the following two general lemmas, which also underpin the strategy followed in the forthcoming paper~\cite{GraphManifolds}.

\begin{lemma}[Initial quasimorphism]\label{lem:quasimorphisms_1}
 Let $1\to \integers\stackrel{\iota}{\longrightarrow}G\stackrel{\pi}{\longrightarrow}F\to 1$ be a central extension corresponding to a bounded class in $H^2(F,\integers)$.  Then there exists a 
quasimorphism $\phi:G\to\integers$ which is unbounded on $\iota(\integers)$.
\end{lemma}

\begin{proof}
Since the central extension corresponds to a bounded class, there is a set-theoretic section $s:F\to G$ of $\pi$ such that the element $s(f_1)s(f_2)s(f_1f_2)^{-1}$ takes finitely many values as 
$f_1,f_2$ vary in $F$.  Since $s$ is a section of $\pi$, we have $s(f_1)s(f_2)s(f_1f_2)^{-1}\in\mathrm{ker}(\pi)=\iota(\integers)$ for all $f_1,f_2\in F$, so for each $f_1,f_2\in F$ we can choose an 
integer $c(f_1,f_2)$ such that $\iota(c(f_1,f_2))=s(f_1)s(f_2)s(f_1f_2)^{-1}$.  Note that $|c(f_1,f_2)|$ is bounded independently of $f_1,f_2$, because $s(f_1)s(f_2)s(f_1f_2)^{-1}$ takes only 
finitely many possible values.

We now define $\phi$.  Let $x\in G$.  Then there exist unique $f_x\in F,t_x\in\integers$ such that $x=s(f_x)\iota(t_x)$.  We set $\phi(x)=t_x$.  We first verify that $\phi$ is a quasimorphism.  
Indeed, let $x,y\in G$.  Then, since $\iota(\integers)$ is central, we have $$xy=s(f_x)\iota(t_x)s(f_y)\iota(t_y)=s(f_x)s(f_y)\iota(t_x+t_y)=s(f_xt_y)\iota(c(f_x,f_y)+t_x+t_y).$$  So 
$$\phi(xy)=t_x+t_y+c(f_x,f_y)=\phi(x)+\phi(y)+c(f_x,f_y),$$ whence $\phi$ is a quasimorphism because of the uniform bound on $|c(f_x,f_y)|$.  

Finally, observe that for all $t\in\integers$, we have $\phi(\iota(t))=t$, proving the last part of the lemma.
\end{proof}

\begin{lemma}[Action on a quasiline]\label{lem:quasimorphisms_2}
Let $1\to \integers\stackrel{\iota}{\longrightarrow}G\stackrel{\pi}{\longrightarrow}F\to 1$ be a central extension with $F$ a hyperbolic group.  Let $\{C_\alpha\}_{\alpha\in I}$ be a finite 
collection of infinite cyclic subgroups of $G$ such that $\{\pi(C_\alpha)\}_{\alpha\in I}$ is a malnormal collection of infinite cyclic subgroups of $F$.  Then $G$ has an (infinite) generating set 
$\mathcal S$ such that:
\begin{itemize}
 \item $\cayley(G,\mathcal S)$ is quasi-isometric to $\integers$;
 \item $\iota(\integers)$ is unbounded in $\cayley(G,\mathcal S)$;
 \item there exists $D\geq 0$ so that for each $x,g\in G$ and each $\alpha$ the diameter of $gC_\alpha g^{-1}x$ is bounded by $D$.
\end{itemize}
\end{lemma}

\begin{proof}
 By Lemma~4.15 in~\cite{ABO}, it suffices to produce a quasimorphism $\psi:G\to\reals$ such that $\psi$ is unbounded on $\iota(\integers)$, and each $\psi(gC_\alpha g^{-1}x)$  is contained in an interval of uniformly bounded size.

For each $\alpha\in I$, let $a_\alpha$ generate the infinite cyclic group $\pi(C_\alpha)$.  So, $\{\langle a_\alpha\rangle\}_{\alpha\in I}$ is a malnormal collection in $F$, and each $a_\alpha$ has 
infinite order.

We now construct a collection of quasimorphisms; the quasimorphism $\psi$ will be a linear combination of these.  

First, let $\phi:G\to\integers$ be the quasimorphism provided by 
Lemma~\ref{lem:quasimorphisms_1}, which applies since $F$ is hyperbolic and therefore every cohomology class is bounded~\cite[Theorem 15]{Mineyev}.  Let $\psi_0$ 
be the homogenisation of $\phi$.  

Second, for each $\alpha\in I$, we can choose a homogeneous quasimorphism $\psi'_\alpha:F\to\reals$ such that $\psi'_\alpha(a_\alpha)=1$ and 
$\psi'_\alpha(a_{\alpha'})=0$ for $\alpha'\in I-\{\alpha\}$.   This is possible because of malnormality of the collection $\{\langle a_\alpha\rangle\}_{\alpha\in I}$ of infinite cyclic 
subgroups (using either the construction in~\cite{EpsteinFujiwara} or combining~\cite[Theorem 4.2]{HullOsin} with~\cite[Corollary 6.6, Theorem 6.8]{DGO}).  So, $\psi_\alpha=\psi'_\alpha\circ\pi:G\to 
\reals$ is a quasimorphism for each $\alpha$.

Now let $$\psi=\psi_0-\sum_{\alpha\in I}\psi_0(a_\alpha)\cdot\psi_\alpha.$$  Then we claim that $\psi$ is a quasimorphism on $G$ with the required properties. Indeed, $\psi|_{\iota(\mathbb Z)}=\psi_0|_{\iota(\mathbb Z)}$, so $\psi$ is unbounded on $\iota(\mathbb Z)$. Next, note that $\psi(C_\alpha)=\{0\}$, so that for any $c\in C_\alpha$ and $g,x\in G$ we have that $|\psi(gcg^{-1}x)-\psi(x)|$ is bounded by 3 times the defect of $\psi$, showing that $\psi(gC_\alpha g^{-1}x)$  is contained in an interval of uniformly bounded size.
\end{proof}

We now use the two preceding lemmas to produce blow-up data.

\begin{lemma}[Existence of blow-up data]\label{lem:blow_up_data_existence}
	There exists blow-up data for $A_\Gamma$ and $Y$. 
\end{lemma}

\begin{proof}
Fix $H\in\mathcal H$.  We have a central extension $H\stackrel{\iota}{\longrightarrow}N(H)\stackrel{\pi}{\longrightarrow}F$, where $H\cong\integers$ and $F$ is virtually free (and in particular, 
hyperbolic).  This follows from Definition~\ref{defn:fancy_H} and either Lemma~\ref{lem:stab_tree} or  Lemma~\ref{lem:virtually splits}, according to whether $H$ is generated by a 
standard generator or by an element of the form $z_{ab}$.

Choose one element of each $N(H)$--orbit in $\mathcal I(N(H))$; there are finitely many by Lemma~\ref{lem:I_finite}.  Each of these is an orbit $g_iH_ig_i^{-1}\cdot x_i$ in $N(H)^{+r}$.

Each $g_iH_ig_i^{-1}$ is infinite cyclic.  Moreover, $g_iH_ig_i^{-1}\cap H=\{1\}$ by Lemma~\ref{lem:intersection_normalisers}, so $\pi(g_iH_ig_i^{-1})$ 
is again infinite cyclic.  

Observe that $\{\pi(g_iH_ig_i^{-1}):\alpha\in I\}$ is a malnormal collection of subgroups of $F$.  Indeed, suppose not, so that for distinct $i,j$, the following holds.  Choose elements $c,c_i,c_j$ 
(either standard generators or elements of the form $z_{ab}$) generating $H,H_i,H_j$ respectively.  Then, for some $k\in N(H)$, there exist positive integers $p,q$ and an integer $r$ such that 
$kc_j^pk^{-1}=c_i^qc^r$.  Now, $c_i^qc^r$ centralises $c_i$, whence $kc_j^pk^{-1}$ centralises $c_i$.  So, $c_i$ centralises $kc_jk^{-1}$, by Lemma~\ref{lem:centraliser_power}.  Hence 
$c_i,kc_jk^{-1}$, and $c$ pairwise commute, contradicting Lemma~\ref{lem:no square}.

So, Lemma~\ref{lem:quasimorphisms_2} yields a generating set $\mathcal S$ of $N(H)$. Recall that $N(H)^{+r}$ is a finite disjoint union of right cosets $N(H)c_i$ of $N(H)$. We can then consider the graph with vertex set $N(H)^{+r}$ with vertices $g,h$ connected by an edge if $g^{-1}h$ or $h^{-1}g$ belong to $\mathcal S\cup \{c_i\}$. By \cite[Theorem 5.1]{CharneyCrisp} (or a direct argument) the vertex set $\Lambda_{N(H)}$ endowed with the induced metric is $N(H)$-equivariantly quasi-isometric to to the quasiline $\mathrm{Cay}(G,\mathcal S)$.

We can define $\phi_{N(H)}:N(H)\to\Lambda_{N(H)}$ to be the identity map, and we observe that $\Lambda_{N(H)}$ and $\phi_{N(H)}$ satisfies the conditions in Lemma~\ref{lem:extending_equivariantly}. 

 Since there are finitely many subgroups $H\in\mathcal H$, we can 
choose constants $B_0,L_0$ such that each $\Lambda_{N(H)}$ is $(B_0,B_0)$--quasi-isometric to $\integers$, the maps 
$\Phi_{N(H)}$ are $(L_0,L_0)$--coarsely lipschitz (where $N(H)$ carries the word metric from the fixed finite generating set 
of $A_\Gamma$), and each element of $\mathcal I(N(H))$ maps to a set in $\Lambda_{N(H)}$ of diameter at most $B_0$, for all 
$H$.  By construction, there are boundedly many $N(H)$--orbits in $\Lambda_{N(H)}$.

We thus have that $\{(\Lambda_{N(H)},\phi_{N(H)}\}$ satisfies 
conditions (B') and (C') from Lemma~\ref{lem:extending_equivariantly}. By Lemma~\ref{lem:extending_equivariantly}, it possible to extend the construction equivariantly to obtain 
condition~~\eqref{item:phi_equivariance}. 
\end{proof}

\subsection{The blown-up commutation graph}\label{subsec:blown_up_graph}
We now define the simplicial complex $X$.  Fix a blow-up data as in Definition~\ref{defn:blow_up_data}.

\begin{defn}[Blown-up commutation graph]
\label{defn:blow_up_commute}
We define a simplicial graph, called the \textbf{blown-up 
commutation graph}, as follows: 
	
\begin{itemize}
	\item The graph has a vertex for each element of  
	$$\bigsqcup_{H \in \calH}G/N(H) ,$$ 
	and for each coset $gN(H)$ in the above set, we add $\Lambda_{gN(H)}$ to the vertex set.
	\item For every $g \in G$ and $H \in \calH$, we put an edge between $gN(H)$ and each vertex of $\Lambda_{gN(H)}$. 
Moreover, whenever $g,h \in G$ 
and $H, H'\in \calH$ are such that $gHg^{-1}$ and $hH'h^{-1}$ are distinct and commute (i.e. when $gN(H)$ and $hN(H')$ are 
$Y$--adjacent), we join every vertex in 
$\{gN(H)\}\cup\Lambda_{gN(H)}$ to every vertex of $\{hN(H')\}\cup\Lambda_{hN(H')}$.
\end{itemize}
	We denote by $X$ the flag completion of this graph. 
	\end{defn} 

Observe that $A_\Gamma$ acts on $X$ by simplicial automorphisms, because of 
Definition~\ref{defn:blow_up_data}.\eqref{item:phi_equivariance}.  

\begin{defn}[Projection, fibre, roots and leaves]\label{def:proj_blowup}
	There is a natural $A_\Gamma$--equivariant simplicial projection $p: X \rightarrow Y$ obtained as follows. At the 
level of vertices, for $H \in \calH$ and $g \in G$, we set 
	$$p(gN(H)) = gN(H), ~~~ p(\Lambda_{gN(H)})=gN(H).$$ 
Note that two adjacent vertices of $X$ are either sent to the same vertex or to two adjacent vertices of $Y$. 
Thus, this definition extends to a map from $X$ to $Y$.  For a simplex $\Delta$ of $X$, we denote by $\overline{\Delta}$ the image simplex $p(\Delta) \subset Y$.

Note that for each vertex $v:=gN(H) \in Y$, the preimage $X_v \coloneqq p^{-1}(v)$, called the \textbf{fibre of} $v$, is a 
tree that is the simplicial cone over $\Lambda_{gN(H)}$: 
$$X_v = \{gN(H)\} \ast \Lambda_{gN(H)},$$
where ``$\ast$'' denotes the simplicial join between two sub-complexes.  
The \textbf{root} of $X_v$ will be the vertex $gN(H)$, while the other vertices of $X_v$ will be called \textbf{leaves}.
\end{defn}

Thus, $X$ can be thought of as being  obtained from  $Y$   by blowing-up each vertex $gN(H)$ of $Y$ into a 
cone over a quasiline, in such a way that fibres $X_v$ and $X_{v'}$ over adjacent 
vertices $v, v'$ of $Y$ span a simplicial join $X_v \ast X_{v'}$ in $X$. 

\begin{figure}[h]
	\begin{center}
		\scalebox{1}{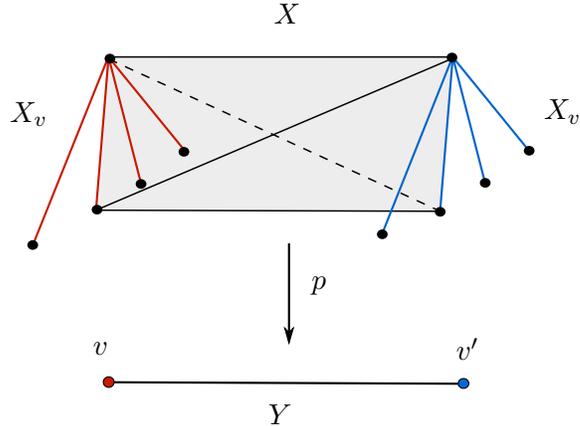}
		\caption{A portion of the blow-up $X$ over an edge of $Y$. The fibres $X_v$ and $X_{v'}$ span a simplicial join in $X$. The figure illustrates how an edge of $X_v$ and an edge of $X_{v'}$ span a tetrahedron of $X$.}
		\label{fig:blowup}
	\end{center}
\end{figure}

\begin{defn}
	We define the map $q: X \rightarrow \widehat{D}$ as $q\coloneqq \iota \circ p$, where $\iota:Y\to \widehat D$ is the map to the coned-off Deligne complex defined in Lemma \ref{lem:iota_vertex} and Definition \ref{def:iota}.
\end{defn}

\begin{lemma}\label{lem:q_qi}
	The map $q: X \rightarrow \widehat{D}$ is an $A_\Gamma$--equivariant simplicial quasi-isometry.
\end{lemma}

\begin{proof}
The map $q$ is simplicial, being a composition of simplicial maps, and is equivariant by construction. We have that $p$ is a surjective simplicial whose fibres have uniformly bounded diameter (the bound being 2), so that $p$ is a quasi-isometry. Also, $\iota$ is a quasi-isometry by Lemma \ref{lem:iota QI}, and hence $q$ is also a quasi-isometry.
\end{proof}

\subsection{Simplices of the blow-up and their links}\label{subsec:simplices_of_X} In this subsection, we describe the simplices of $X$ and their links, and prove  that $X$  satisfies several of the properties from Definition~\ref{defn:combinatorial_HHS} and conditions from Theorem~\ref{thm:main criterion}.  

\begin{defn}[Join decomposition of simplices] A simplex $\Delta$ of $X$ naturally decomposes as a join as follows. 
	For each vertex $v \in \overline{\Delta}$, we define the \textbf{fibre} of $\Delta$ at $v$ by $\Delta_v \coloneqq 
\Delta \cap X_v \subset \Delta$, which is either a single vertex or a single edge of $\Delta$. Then $\Delta$ is naturally the 
join of all its fibres:
	$$\Delta = \bigast_{v \in V(\overline{\Delta})} \Delta_v.$$ 
\end{defn}

\begin{lemma}\label{cor:simplex projection}
	Let $\Delta$ be a simplex of $X$. Then its projection $\overline{\Delta}$ is either a vertex or an edge. In 
particular, $X$ is a $3$-dimensional simplicial complex. \qed
\end{lemma}

\begin{proof}
	Note that the projection $\overline{\Delta}$  is a simplex of $Y$. It thus follows from Lemma \ref{lem:no square} 
that $\overline{\Delta}$ is either a vertex or an edge. Since the simplices of $X_v$ are either vertices or edges, the result 
follows. 
\end{proof}

\begin{lemma}
	Let $\Delta$ be a simplex of $X$. 
	The link $\link_X(\Delta)$ decomposes as the following simplicial join: 
	$$\link_X(\Delta) = p^{-1}\big(\link_Y(\overline{\Delta})\big) \ast \big(\bigast_{v\in V(\overline{\Delta})} 
	\link_{X_v}(\Delta_v)\big).$$
	\label{lem:join decomposition}
\end{lemma}

\begin{proof}
	A vertex in $\link_X(\Delta)$ is adjacent to every vertex of $\Delta$ by definition. Since edges of $X$ either 
	project to vertices or edges of $Y$, we have that $p\big( \link_X(\Delta)\big) \subset \Star_Y(\overline{\Delta})$.  For two 
	adjacent vertices $v, v'$ of $Y$, every vertex of $X_v$ is adjacent to every vertex of $X_{v'}$. It follows that for $v \in 
	\Star_Y(\overline{\Delta})$, a vertex $w \in X_v$ is in $\link_X(\Delta)$ if and only if it is in $\link_{X_v}(\Delta\cap 
	X_v)$. The result follows. 
\end{proof}

Theorem~\ref{thm:main criterion} requires that the action of $A_\Gamma$ on $X$ has finitely many orbits of links, which we 
now verify.  

\begin{lemma}[Finite index set]\label{lem:index_set}
	The action of $A_\Gamma$ on $X$ has finitely many orbits of subcomplexes of the form $\link_X(\Delta)$, where $\Delta$ is a 
	simplex of $X$.
\end{lemma}

In the proof of the lemma we will give a classification of links which will also be useful later on.

\begin{proof}
	Let $\Delta$ be a non-maximal simplex of $X$ and let $\bar\Delta$ be the simplex of $Y$ to which $\Delta$ projects.  We 
	divide into cases according to $\bar\Delta$ and the fibers in $\Delta$ over the vertices of $\bar \Delta$.
	
	\begin{enumerate}
		\item $\bar\Delta$ is a vertex $gN(H)$, and $\Delta$ is the single vertex $gN(H)$.  There are finitely many $A_\Gamma$ 
		orbits of such $\Delta$, since $\mathcal H$ is finite.  Hence there are finitely many orbits of links of such $\Delta$.
		
		\item $\bar\Delta$ is a vertex $gN(H)$, and $\Delta$ is a vertex $\lambda\in\Lambda_{gN(H)}$.  There are boundedly many $gN(H)g^{-1}$ 
		orbits of $\lambda$ for the given $gN(H)$, by Definition~\ref{defn:blow_up_data}, so there are finitely many orbits of such 
		$\Delta$.
		
		\item $\bar\Delta$ is a vertex $gN(H)$, and $\Delta$ is spanned by $\{gN(H),\lambda\}$ for some $\lambda\in\Lambda_{gN(H)}$. 
		Since $gN(H)g^{-1}$ acts on $\Lambda_{gN(H)}$ with boundedly many orbits, and there are finitely many orbits of vertices $gN(H)$, 
		there are finitely many orbits of such simplices $\Delta$.
		
		\item $\bar\Delta$ is an edge joining $gN(H)$ to $hN(H')$, and $\Delta$ is the $1$--simplex spanned by $\{gN(H),hN(H')\}$.  
		By Corollary \ref{cor:Y_orbits}, there are finitely many orbits of such simplices.
		
		\item $\bar\Delta$ is an edge joining $gN(H)$ to $hN(H')$, and $\Delta$ joins $gN(H)$ to some $\mu\in\Lambda_{hN(H')}$.  
		Observe that all simplices that join $gN(H)$ to a point in $\Lambda_{hN(H')}$ have the same link, so there are at most as 
		many $A_\Gamma$ orbits of links of such $\Delta$ as there are orbits of edges in $Y$, i.e. finitely many by Corollary \ref{cor:Y_orbits}.
		
		\item $\bar\Delta$ is an edge joining $gN(H)$ to $hN(H')$, and $\Delta$ joins $\lambda\in \Lambda_{gN(H)}$ to 
		$\mu\in\Lambda_{hN(H')}$.  Then $\link_X(\Delta)$ is the edge joining $gN(H),hN(H')$; there are finitely many orbits of 
		these.
		
		\item $\bar\Delta$ is an edge joining $gN(H),hN(H')$ and $\Delta$ is spanned by $\{gN(H),\lambda,hN(H')\}$, where 
		$\lambda\in\Lambda_{gN(H)}$.  Then $\link_X(\Delta)=\Lambda_{hN(H')}$, and there are finitely many orbits of such 
		subcomplexes since there are finitely many orbits of vertices in $Y$.
		
		\item $\bar\Delta$ is an edge joining $gN(H),hN(H')$, and $\Delta$ is spanned by $\{gN(H),\lambda,\mu\}$, where 
		$\lambda\in\Lambda_{gN(H)}$ and $\mu\in\Lambda_{hN(H')}$.  Then $\link_X(\Delta)=hN(H')$, and there are finitely many orbits 
		of such links because $Y$ has finitely many orbits of vertices.
		
		\item $\Delta=\emptyset$, in which case $\link_X(\Delta)=X$, and there is a single orbit.
	\end{enumerate}
	This exhausts all the cases.
\end{proof}

The list in the previous proof allows to prove that the blow-up $X$ satisfies the finite complexity condition from Definition~\ref{defn:combinatorial_HHS}.

\begin{prop}[Finite complexity]\label{prop:finite_complexity}
	There exists $N$ such that the following holds.  Let 
	$\Delta_1,\ldots,\Delta_n$ be non-maximal simplices of $X$ such that 
	$\link_X(\Delta_i)\subsetneq\link_X(\Delta_{i+1})$ for $1\leq i\leq n-1$.  Then $n\leq N$.
\end{prop}

\begin{proof}
	Suppose that $\Sigma,\Delta$ are non-maximal simplices with $\link_X(\Sigma)\subsetneq\link_X(\Delta)$.  Considering the 
	nine types of simplex explained in the proof of Lemma~\ref{lem:index_set}, we see that links of the same type cannot be properly contained in each other.
 Hence the length of a chain as in the statement is at most $9$.
\end{proof}

In the rest of this subsection is to prove the following:

\begin{prop}[Lattice-ness of links]
	The poset of links of simplices of $X$, ordered by inclusion, has finite height.  Moreover, this poset satisfies the 
	following.
	
	For all simplices $\Delta_1, \Delta_2 $ of $X$ such that there exists a simplex $\Sigma$ of $X$ with unbounded link and with the property that $  \link_{X}(\Sigma) \subset \link_{X}(\Delta_1)\cap \link_{X}(\Delta_2)$ the following holds. There exists a 
	simplex $\Delta$ of $X$ containing $\Delta_1$, and such that for every simplex $\Sigma$ of $X$ with unbounded link and such 
	that  $  \link_{X}(\Sigma) \subset \link_{X}(\Delta_1)\cap \link_{X}(\Delta_2)$, we have $\link_X(\Sigma) \subset 
	\link_X(\Delta)$. 
	\label{prop: links lattice}
\end{prop}

\begin{proof}
	The finite height condition follows from Proposition~\ref{prop:finite_complexity}.
	
	Note that $\link(\Sigma)$ is unbounded if and only if $\Sigma$ is either the empty simplex, an edge contained in a fibre of $X$, or a triangle of $X$ containing exactly two 
	roots, by the classification given in the proof of Lemma \ref{lem:index_set}.
	
	Let $\Delta_1, \Delta_2,\Sigma$ be as in the statement.
	
	The case $\Sigma = \varnothing$ is immediate since in that case $\link_X(\Sigma) =  \link_X(\Delta_1)= 
	\link_X(\Delta_2)= X$, and we can take $\Delta = \varnothing$.  We thus focus on the other two cases. 
	
	\textit{Case 1:} Suppose that there exists such a simplex $\Sigma$ such that $\link_X(\Sigma)\subset 
	\link_X(\Delta_1)\cap \link_X(\Delta_2)$, and such that $\Sigma$ is an edge contained in a fibre of $X$. Then its projection 
	is a single vertex $v\in Y$, and by Lemma \ref{lem:join decomposition} we have $\link_X(\Sigma) = 
	p^{-1}\big(\link_Y(v)\big)$. Since $\link_X(\Sigma)$ is contained in $\link_X(\Delta_1) \cap \link_X(\Delta_2)$, it follows 
	from Lemma \ref{lem:no square} that $\Delta_1, \Delta_2$ are also contained in $X_v$. Since $\Delta_1$ is either a vertex or 
	an edge of $X_v$, we can choose an edge $\Delta$ contained in $\Star_{X_v}(\Delta_1)$. We thus have $\link_X(\Delta) = 
	p^{-1}\big(\link_Y(v)\big)$ by Lemma \ref{lem:join decomposition}, and $\link_X(\Delta)$ satisfies the required property. 
	Indeed, any link contained in $\link_X(\Delta_1)\cap \link_X(\Delta_2)$ and strictly larger than $\link_X(\Delta)$ must 
	contain a vertex of $X_v$, and so this link is a simplicial join, hence bounded.
	
	\textit{Case 2:} Suppose that for every simplex $\Sigma$ such that $\link_X(\Sigma)\subset \link_X(\Delta_1)\cap 
	\link_X(\Delta_2)$, $\Sigma$ is a triangle of $X$ with exactly two roots.  Choose such a simplex $\Sigma$. Then there exists 
	a vertex $v \in Y$ such that $\link_X(\Sigma)$ consists of all the leaves of $X_v$. Moreover, it follows from Lemma 
	\ref{lem:no square} that for every other simplex $\Sigma'$ with unbounded augmented link such that $\link_X(\Delta_1)\cap 
	\link_X(\Delta_2)$, we have $\link_X(\Sigma') \subset X_v$, and in particular $\link_X(\Sigma') = \link_X(\Sigma)$. Now since 
	$\link_X(\Sigma)\subset \link_X(\Delta_1)$, then $\overline{\Delta_1}$ is either a vertex or an edge of $\Star_Y(v)$. 
	
	If $\overline{\Delta_1} = \{v\}$, then we choose a vertex $u$ adjacent to $v$, an edge $f_u$ of $X_u$, and we define 
	$\Delta$ as the triangle of $X$ spanned by $f_u$ and the root of $X_v$. We thus have $\Delta \subset \Star_X(\Delta_1)$, and 
	$\Delta$ satisfies the required conditions. 
	
	If $\overline{\Delta_1} $ is an edge containing $v$, we denote by  $u$ the other vertex of $\overline{\Delta_1} $. We 
	choose an edge $f_u$ of $X_u$ that belongs to $\Star_{X_u}((\Delta_1)_u)$, and we define $\Delta$ as the triangle of $X$ 
	spanned by $f_u$ and the root of $X_v$. We thus have $\Delta \subset \Star_X(\Delta_1)$, and $\Delta$ satisfies the required 
	conditions.
\end{proof}

\subsection{From maximal simplices to elements of $A_\Gamma$}\label{subsec:map_to_group}

In this subsection, we associate to each maximal simplex of $X$ a uniformly bounded subset of $A_\Gamma$. We first describe the maximal simplices of $X$:

\begin{rem}[Description of maximal simplices]\label{rem:max simplex}
	Let $\Delta$ be a maximal simplex of $X$.  Then $\Delta$ has the following form.  First, $\overline{\Delta}$ is an edge of $Y$ 
	joining a vertex $gN(H)$ to a vertex $hN(H')$.  Then, there are vertices $\lambda\in\Lambda_{gN(H)}$ and 
	$\mu\in\Lambda_{hN(H')}$ such that $\Delta$ is the join of the edges $\{gN(H),\lambda\}$ and $\{hN(H'),\mu\}$.
\end{rem}

\begin{defn}\label{defn:max_simplex}
	For simplices $\sigma_1, \sigma_2$ of $X$ projecting to vertices of $Y$  and spanning a simplex of $X$, we denote 
	that 
	simplex by $\Delta(\sigma_1,  \sigma_2)$.
\end{defn}

Recall that one of our goals is to define an $X$--graph $W$ quasi-isometric to $A_\Gamma$, and whose vertex set is the set of maximal simplices of $X$, see Definition~\ref{defn:combinatorial_HHS}. We 
start by defining the vertices of this graph and we construct a map to $A_\Gamma$. The edges of $W$ and the quasi-isometry  $W\rightarrow A_\Gamma$ will be constructed in Section~\ref{sec:augment}.

\begin{defn}
	We denote by $W^{(0)}$ the set of maximal simplices of $X$.
\end{defn}

The blow-up data was constructed to obtain the following crucial lemma:

\begin{lemma}[Bounded sets in $A_\Gamma$]\label{lem:bounded_sets}
	There exists $B_1\geq 0$ such that the following holds for all $B\geq B_1$.  Let $gN(H),hN(H')$ be $Y$--adjacent 
	cosets.  Let $\lambda\in\Lambda_{gN(H)},\mu\in\Lambda_{hN(H')}$.  Then the subset
	$$\phi_{gN(H)}^{-1}(N_{B}(\lambda))\cap \phi_{hN(H')}^{-1}(N_{B}(\mu))$$ of $A_\Gamma$ is nonempty and has diameter bounded 
	in terms of $B$. Moreover, $\phi_{gN(H)}(gN(H)^{+r}\cap hN(H')^{+r})$ is $B_1$-dense.
\end{lemma}

We prove the lemma after the following two auxiliary lemmas (coboundedness in the first lemma will only be needed in the next section).

\begin{lemma}\label{lem:product_quasilines}
 Suppose $C_0,C_1$ are infinite cyclic groups and $\Lambda_0$ and $\Lambda_1$ are quasilines. Suppose that $C_0\times C_1$ acts on both $\Lambda_i$ with the property that $C_i$ acts with unbounded 
orbits on $\Lambda_i$ and with bounded orbits on $\Lambda_{i+1}$ (where the action of $C_i$ is the restriction of the action to the subgroup $C_i$ of $C_0\times C_1$). Then the diagonal action of 
$C_0\times C_1$ on $\Lambda_0\times\Lambda_1$ is proper  and cobounded (where the product is given the $\ell^1$ metric).
\end{lemma}

\begin{proof} 
Fix $\mu_0\in\Lambda_0,\mu_1\in\Lambda_1$.  Let $L\geq 0$ be given and suppose that $(c_0,c_1),(d_0,d_1)\in C_0\times C_1$ have the 
property that $\dist_{\Lambda_0\times \Lambda_1}((c_0,c_1)\cdot(\mu_0,\mu_1),(\mu_0,\mu_1))\leq L,$ i.e.
$$\dist_{\Lambda_0}((c_0,c_1)\cdot\mu_0,\mu_0)+\dist_{\Lambda_1}((c_0,c_1)\cdot\mu_1,\mu_1)\leq L.$$

Let $B$ be the bound on $C_i$ orbits in $\Lambda_{i+1}$.  Then
$\dist_{\Lambda_0}((c_0,c_1)\cdot\mu_0,(c_0,1)\cdot\mu_0)\leq B$ and
$\dist_{\Lambda_1}((c_0,c_1)\cdot\mu_1,(1,c_1)\cdot\mu_1)\leq B.$
So $$\dist_{\Lambda_0}((c_0,1)\cdot\mu_0,\mu_0)\leq L+B$$ and $$\dist_{\Lambda_0}((1,c_1)\cdot\mu_1,\mu_1)\leq L+B.$$
Since $C_i$ has unbounded orbits on the quasiline $\Lambda_i$, there are finitely many such $c_i$, for $i\in\{0,1\}$.  So 
$(c_0,c_1)$ is one of finitely many elements of $C_0\times C_1$, as required.

Regarding coboundedness, the action of the cyclic group $C_i$ on the quasiline $\Lambda_i$ is, say, $B_i$-cobounded, which easily implies that the action of $C_0\times C_1$ is $(B_0+B_1+2B)$-cobounded.
\end{proof}

\begin{lemma}\label{lem:f_proper}
 Let $X$ and $Y$ be metric spaces, and let $f:X\to Y$ be a coarsely Lipschitz map. Suppose that a group $G$ acts metrically properly on $Y$ and coboundedly on $X$, and that $f$ is 
$G$--equivariant. Then $f$ is uniformly metrically proper, that is, for every $L$ there exists $D$ such that the preimage under $f$ of any ball of radius $L$ in $Y$ has diameter at most $D$.
\end{lemma}

\begin{proof}
 Let $K$ be such that $f$ is $(K,K)$--coarsely 
lipschitz.  Fix $x\in X$ and $R\geq 0$ such that $G\cdot B_R^X(x)=X$.  Let $L\geq 0$.  By properness, there exists a finite set $F_L\subset G$ such that $\dist_Y(hf(x),f(x))\leq L+2(KR+K)$ implies 
$h\in F_L$.

We will show that if two points of $X$ map $L$-close in $Y$ then they are $D(L)$-close in $X$; this is easily seen to imply the lemma (after increasing $D$).

Suppose $x_0,x_1\in X$ satisfy $\dist_Y(f(x_0),f(x_1))\leq L$.  Choose $g_0,g_1\in G$ such that $\dist_X(x_i,g_ix)\leq R$ for $i\in\{0,1\}$.  Then $\dist_Y(g_0f(x),g_1f(x))\leq L + 2(KR+K)$, so 
$\dist_Y(g_1^{-1}g_0f(x),f(x))\leq L+2(KR+K)$.  Thus $g_1^{-1}g_0\in F_L$, so $$\dist_X(x_0,x_1)\leq 2R+\dist_X(g_0x,g_1x)\leq 2R+\max_{h\in 
F_L}\dist_X(x,hx),$$
as required.
\end{proof}

\begin{proof}[Proof of Lemma~\ref{lem:bounded_sets}]
	Fix $Y$--adjacent cosets $gN(H),hN(H')$.  Recall from Lemma \ref{lem:intersection_normalisers} that $\langle gHg^{-1}, hH'h^{-1}\rangle$ is naturally isomorphic to $gHg^{-1}\times hH'h^{-1}$. By definition of the blow-up data, we have a $gHg^{-1}\times hH'h^{-1}$--equivariant map $$\phi_{gN(H)}\times 
	\phi_{hN(H')}:gN(H)^{+r}\cap hN(H')^{+r}\to\Lambda_{gN(H)}\times\Lambda_{hN(H')}.$$
	Equipping $gN(H)^{+r}\cap hN(H')^{+r}$ with the word metric and $\Lambda_{gN(H)}\times\Lambda_{hN(H)}$ with the $\ell_1$--metric, we 
	see that $\phi_{gN(H)}\times \phi_{hN(H')}$ is $(2L_0,2L_0)$--coarsely 
	lipschitz.
	
	Since $gHg^{-1}$ stabilises $hN(H')^{+r}$ (because it is contained in $hN(H')h^{-1}$ by definition of the edges of $Y$), we have that it also acts on $gN(H)^{+r}\cap hN(H')^{+r}$. Since $gHg^{-1}$ has unbounded orbits in $\Lambda_{gN(H)}$ by Condition~\eqref{item:unbounded_orbits}, we have that $\phi_{gN(H)}$ 
	restricts to a $B_1$--coarsely surjective map on $gN(H)^{+r}\cap hN(H')^{+r}$ for some $B_1$ (proving the moreover part). By equivariance (Condition ~\eqref{item:phi_equivariance} ) and the fact that there are finitely many orbits of edges of $Y$ (Corollary \ref{cor:Y_orbits}), the constant 
	$B_1$ can be chosen independently of the cosets in question.
	
	Let us prove the ``nonempty'' part. Fix $\lambda\in \Lambda_{gN(H)}$ and $\mu\in\Lambda_{hN(H')}$. By coarse surjectivity, there exists $x\in 
	gN(H)^{+r}\cap hN(H')^{+r}$ such that $\dist_{\Lambda_{gN(H)}}(\phi_{gN(H)}(x),\lambda)\leq B_1$. In view of Condition \eqref{item:bounded_cosets}, the $hH'h^{-1}$--orbit of $\phi_{gN(H)}(x)$ is contained in the $(B_0+B_1)$--neighbourhood of $\lambda$. Since the action of $hH'h^{-1}$ on $\Lambda_{hN(H)}$ is cobounded (with uniform constant), up to enlarging $B_1$ we can find an element in $hH'h^{-1}$--orbit of $x$ in $gN(H)^{+r}\cap hN(H')^{+r}$ that maps $B_1$--close to both $\lambda$ and $\mu$, as required.

We now make the following claim, which will also be used later.

\begin{claim}\label{claim:quasilines_proper}
 Fix $Y$--adjacent cosets $gN(H),hN(H')$. Then the map
 $$\phi_{gN(H)}\times 
	\phi_{hN(H')}:gN(H)^{+r}\cap hN(H')^{+r}\to\Lambda_{gN(H)}\times\Lambda_{hN(H')}$$
	is proper. More precisely, there exists $f_0:\reals_{\ge0}\to\reals_{\ge0}$ independent of $gN(H),hN(H')$ such that, for all $s\in \reals_{\ge0}$, the preimage 
	under $\phi_{gN(H)}\times \phi_{hN(H')}$ of any $s$--ball in 
	$\Lambda_{gN(H)}\times\Lambda_{hN(H')}$ has diameter at most $f_0(s)$. Moreover, the action of $\langle gHg^{-1}, hH'h^{-1}\rangle$ on $\Lambda_{gN(H)}\times\Lambda_{hN(H')}$ is cobounded.
\end{claim}

\begin{proof}
	By Lemma \ref{lem:product_quasilines} and Conditions \eqref{item:unbounded_orbits} and \eqref{item:bounded_cosets}, the action of $gHg^{-1}\times hH'h^{-1}$ on $\Lambda_{gN(H)}\times\Lambda_{hN(H')}$ is proper and cobounded.
	Then, by Lemma \ref{lem:f_proper} (with $f=\phi_{gN(H)}\times \phi_{hN(H')}$) and Lemma \ref{lem:intersection_of_cosets} we have there exists $f_0:\reals_{\ge0}\to\reals_{\ge0}$ such that, for all $s\in \reals_{\ge0}$, the preimage under $\phi_{gN(H)}\times \phi_{hN(H')}$ of any $s$--ball in 
	$\Lambda_{gN(H)}\times\Lambda_{hN(H')}$ has diameter at most $f_0(s)$.  Finiteness of $\mathcal H$ and 
	condition~\eqref{item:phi_equivariance} imply that this $f_0$ can be chosen independently of the 
	cosets $gN(H),hN(H')$.
\end{proof}

	Now, if $x,y\in gN(H)^{+r}\cap hN(H')^{+r}$ map $B$--close 
	to the image of $x$ under both $\phi_{gN(H)}$ and $\phi_{hN(H')}$, then $\dist_{A_\Gamma}(x,y)\leq f_0(2B)$, yielding the 
	required bound on the diameter of $\phi_{gN(H)}^{-1}(N_{B}(\lambda))\cap \phi_{hN(H')}^{-1}(N_{B}(\mu))$.
\end{proof}

We are almost ready to fix constants for the rest of the paper, but we first need lemma about the geometry of quasilines.

\begin{lemma}\label{lem:order_quasiline}
 Given a quasiline $\Lambda$, there exists a constant $M_0$ such that for all $M\geq M_0$ the following holds. There exists a partial order $\prec$ on $\Lambda$ with the following properties:
 \begin{itemize}
  \item whenever $d_{\Lambda}(x,y)\geq M$ then $x$ and $y$ are $\prec$-comparable.
  \item for each $z$ there exist $z'_1$ and $z'_2$ such that for all $z_i$ with $d_{\Lambda}(z'_i,z_i)\leq M/10$ we have $z_1\prec z \prec z_2$ and $d_{\Lambda}(z,z_i)\in [M, 2M]$.
  
  \item if $x\prec y \prec z$ and $d_{\Lambda}(x,y),d_{\Lambda}(y,z)\geq M$ then any geodesic from $x$ to $z$ passes within $M/10$ of $y$.
 \end{itemize}
\end{lemma}

\begin{proof}
Consider a quasi-isometry $f:\Lambda\to \mathbb R$ and define $x\prec y$ if $f(x)<f(y)$. The details are left to the reader.
\end{proof}

\begin{conv}\label{conv:constants}
 We fix $B_0$ as in the definition of blow-up-data, then choose $B_1$ as in Lemma \ref{lem:bounded_sets}, with the follwing additional properties:
 \begin{itemize}
  \item for each $Y$-adjacent pair $gN(H), hN(H')$ the action of each $gHg^{-1}$ on $\Lambda_{gN(H)}$ is $B_1$-cobounded (see Claim \ref{claim:quasilines_proper}),
  \item $\phi_{gN(H)}(gN(H)^{+r}\cap hN(H')^{+r})$ is $B_1/100$-dense,
   \item $B_1\geq 10M_0$, where $M_0$ satisfies Lemma \ref{lem:order_quasiline} for all quasilines $\Lambda_{gN(H)}$.
 \end{itemize}

 When $B=B_1$, we denote by $B_2$ the bound provided by Lemma \ref{lem:bounded_sets}.
\end{conv}

Fix $gN(H)$.  Recall that any $kH'k^{-1}\cdot x\in\mathcal I(gN(H))$ has image in $\Lambda_{gN(H)}$ of diameter 
bounded by $B_0$ (and $B_0$ does not depend on $H,H',g,k,x$), by 
Definition~\ref{defn:blow_up_data}.\eqref{item:bounded_cosets}.

Let $\delta\in W^{(0)}$ be a maximal simplex spanned by 
vertices $gN(H),hN(H')$ and $\lambda\in\Lambda_{gN(H)},\mu\in \Lambda_{hN(H')}$.  Recall from Lemma~\ref{lem:bounded_sets} that 
$$\phi_{gN(H)}^{-1}(N_{B_1}(\lambda))\cap \phi_{hN(H')}^{-1}(N_{B_1}(\mu))$$
is a nonempty subset of $A_\Gamma$, and the corresponding subset with $B_1$ replaced by $B_1+B_0$ has diameter (in the word metric) at most $B_2$.

\begin{defn}\label{defn:little_w}
	Let $\delta\in W^{(0)}$ be a maximal simplex spanned by 
	vertices $gN(H),hN(H')$ and $\lambda\in\Lambda_{gN(H)},\mu\in hN(H')$. Let 
	$$\sigma(\lambda)= \phi_{gN(H)}^{-1}(N_{B_1}(\lambda))$$
	and
	$$w(\delta)=\sigma(\lambda)\cap \sigma(\mu).$$ 
	This defines a coarse map $w:W^{(0)}\to A_\Gamma$.  
\end{defn}

\begin{lemma}[Basic properties of $w$]\label{lem:basic_w_properties}
The coarse map $w$ is $A_\Gamma$--equivariant and sends each maximal simplex to a nonempty subset of $A_\Gamma$ of diameter 
at most $B_2$.  Moreover, if $\delta\in W^{(0)}$ projects to the edge of $Y$ joining cosets $gN(H),hN(H')$, then 
$w(\delta)\subset gN(H)^{+r}\cap hN(H')^{+r}$. 
\end{lemma}

\begin{proof}
 Let $\delta$ be a maximal simplex, as above.  By construction, $w(\delta)$ is $$\phi_{gN(H)}^{-1}(N_{B_1}(\lambda))\cap \phi_{hN(H')}^{-1}(N_{B_1}(\mu)),$$ which is nonempty by our choice of 
$B_1$ (which we made in Convention \ref{conv:constants}). 

Also, by the choice of $B_2$, $w(\delta)$ has diameter at most $B_2$.

The ``moreover'' assertion holds by construction, as does $A_\Gamma$--equivariance (see Condition \ref{item:phi_equivariance}).
\end{proof}

We also have:

\begin{lemma}[Surjection $W\to A_\Gamma$]\label{lem:surjection_W}
Each $x\in A_\Gamma$ satisfies $x\in w(\delta)$ for some $\delta\in W^{(0)}$.   
\end{lemma}

\begin{proof}
This follows since $w$ is $A_\Gamma$--equivariant and the action of $A_\Gamma$ on itself has one orbit.
\end{proof}

\section{The augmented complex}\label{sec:augment}

As per Convention \ref{conv:connected}, we are considering an Artin group of large and hyperbolic type $A_\Gamma$ with $\Gamma$ connected, and we work with the commutation graph $Y$ and blow-up $X$ constructed in the preceding sections. In this section we 
construct a combinatorial HHS with underlying simplicial complex $X$.  

\subsection{Construction of the augmented complex}\label{subsec:augmented}
Let $r_1,r_2\geq 0$ be constants to be determined.  The first step is to define a graph $W_{r_1,r_2}$ whose vertex set is 
the set $W^{(0)}$ of maximal simplices of $X$; recall that Remark \ref{rem:max simplex} and Definition \ref{defn:max_simplex} describe maximal simplices. The edges of $W_{r_1,r_2}$ are as follows.

\begin{defn}[$W_{r_1,r_2}$-edges]\label{defn:W_edges} 
Let $\delta_0,\delta_1$ be 
maximal simplices of $X$, where $\delta_i=\Delta(\alpha_i,\beta_i)$.  Here, $\alpha_i$ is an edge joining $g_iN(H_i)$ to 
$\lambda_i\in\Lambda_{g_iN(H_i)}$ and $\beta_i$ is an edge joining $h_iN(H'_i)$ to $\mu_i\in\Lambda_{h_iN(H_i')}$.

	\begin{itemize}
		\item $W$\textbf{-edges of type 1}:  Suppose $g_0N(H_0)=g_1N(H_1)$ and 
$h_0N(H_0')=h_1N(H_1')$, and $\mu_0=\mu_1$.  We declare $\delta_0$ to be $W_{r_1,r_2}$--adjacent to $\delta_1$ if 
$$\dist_{\Lambda_{g_0N(H_0)}}(\lambda_0,\lambda_1)\leq r_1.$$
Note that simplices joined by an edge of type 1 differ on a single vertex and  project to the same edge of $Y$. 
		
		\item $W$\textbf{-edges of type 2}: In this case, $h_0N(H_0')=h_1N(H_1')$ and $\mu_0=\mu_1$, while 
$g_0N(H_0)\neq g_1N(H_1)$.  (So, $p(\delta_0\cup\delta_1)$ is the union of two distinct edges of $Y$ that share a vertex.)  
In this case, we declare $\delta_0$ and $\delta_1$ to be $W_{r_1,r_2}$--adjacent if 
$$\dist_{A_\Gamma}(\sigma(\lambda_0),\sigma(\lambda_1))\leq 
r_2.$$
Note that simplices joined by an edge of type 2 intersect in an edge of the fibre over 
$h_0N(H_0')=h_1N(H_1')$.
	\end{itemize}
\end{defn}

The action of $A_\Gamma$ on $W^{(0)}$ extends to an action of $A_\Gamma$ on $W$ by graph automorphisms.  

The (finitely many) conditions needed on $r_1,r_2$ for our constructions to work appear in the proofs of the various results below.

Morally, the definition of $W$-edges of type 2 should be the inequality in the following lemma. However, we need the definition exactly as stated to show fullness of links (see Proposition~\ref{prop: C=C0}).

\begin{lemma}\label{lem:close_levels}
 For every $r_2$ there exists $r_3$ with the following property. Fix the notation of the ``$W$-edges of type 2'' part of Definition \ref{defn:W_edges}. Then
 $$\dist_{A_\Gamma}(w(\delta_0),w(\delta_1))\leq r_3.$$
\end{lemma}

\begin{proof}
 We first make two preliminary claims.
 
 \begin{claim}
  There exists a function $s$ such that the following holds. Given cosets $g_0N(H_0)$, $hN(H),$ and $g_1N(H_1)$ that, as vertices of $Y$, form a path, and $x_i\in g_iN(H_i)^{+r}$ we have 

  $$\dist_{A_{\Gamma}}(x_0, g_1N(H_1)^{+r}\cap hN(H)^{+r})\leq s\Big(\dist_{A_\Gamma}(x_0,x_1)\Big).$$
 \end{claim}

 \begin{proof}
Set $d=\dist_{A_\Gamma}(x_0,x_1)$. Up to multiplying all objects on the left by $x_0^{-1}$, we can assume that $x_0$ is the identity. Note that there are $N=N(d)<+\infty$ cosets $gN(H')$ with $H'\in \mathcal H$ within distance $d+r$ of $x_0$. Therefore, it suffices to consider fixed cosets $g_0N(H_0),g_1N(H_1)$ as in the statement, and prove the claimed inequality for $x_i$ in those particular $g_iN(H_i)^{+r}$. Note that since $Y$ does not contain triangles (Lemma \ref{lem:no square}), for each pair $g_0N(H_0),g_1N(H_1)$ there is at most one $hN(H)$ at distance $1$ in $Y$ from both. In particular, there are at most finitely many $hN(H)$ that can occur, so that once again we can fix one. But at this point, the left hand side is a number, so we are done.
 \end{proof}
 
\begin{claim}
 There exists a function $t$ such that the following holds. Suppose that $gN(H'), hN(H)$ are adjacent vertices of $Y$ and $\lambda\in \Lambda_{gN(H')}$. Given $x\in \sigma(\lambda)$ there exists $y\in \sigma(\lambda)\cap hN(H)^{+r}$ such that  $\dist_{A_\Gamma}(x,y)\leq t(\dist_{A_{\Gamma}}(x, gN(H')^{+r}\cap hN(H)^{+r}))$.
\end{claim}

\begin{proof}
Similarly to the proof of the previous claim, we can assume that $x$ is the identity and that $gN(H'), hN(H)$ are fixed. However, $\lambda$ is not; there are infinitely many possible ones. To overcome this, we will find two candidates for $y$, and we will see that at least one of them is in $\sigma(\lambda)\cap hN(H)^{+r}$, which suffices.

Towards this, recall the order $\prec$ from Lemma \ref{lem:order_quasiline}. From Convention \ref{conv:constants} we know that the lemma applies with $M=B_1/10$, and that $\phi_{gN(H')}(gN(H')^{+r}\cap hN(H)^{+r})$ is $B_1/100$-dense. Hence, we can find $y_1,y_2\in gN(H')^{+r}\cap hN(H)^{+r}$ such that $z_i=\phi_{gN(H')}(y_i)\in \Lambda_{gN(H')}$ satisfy the following:
\begin{itemize}
 \item $\dist_{\Lambda_{gN(H')}} (z,z_i)\in [B_1/10,B_1/5]$ where $z=\phi_{gN(H')}(x)$,
 \item $z_1\prec z \prec z_2$.
\end{itemize}

There are now two cases. If $\dist_{\Lambda_{gN(H')}}(z,\lambda)\leq B_1/2$, then either $z_i$ lies in $N_{B_1}(\lambda)$, that is, either $y_i$ lies in $\sigma(\lambda)$, and we are done. If not, any two points $\lambda, z, z_1,z_2$ are either $\prec$-comparable by construction, or at distance at least $B_1/10$, whence, again, $\prec$-comparable. If $\lambda\prec z_1$, then it is readily seen from the last item of Lemma \ref{lem:order_quasiline} that $z_1\in N_{B_1}(\lambda)$, since $z$ lies in said ball and a geodesic from $\lambda$ to $z$ passes within distance $B_1/100$ of $z_1$, so that
$$d_{\Lambda_{gN(H')}}(\lambda,z_1)\leq d_{\Lambda_{gN(H')}}(\lambda,z)-d_{\Lambda_{gN(H')}}(z,z_1)+2 B_1/100\leq B_1-B_1/10 + B_1/50\leq B_1.$$
Similarly, if $z_2\prec \lambda$ then $z_2\in N_{B_1}(\lambda)$. On the other hand, we cannot have $z_1\prec \lambda \prec z$ or $z\prec \lambda \prec z_2$, for otherwise we would have $\dist_{\Lambda_{gN(H')}}(z,\lambda)<B_1/2$. We now covered all cases.
\end{proof}

We now argue that for every $r_2$ there exists $r_{2.5}$ so that if $\dist_{A_\Gamma}(\sigma(\lambda_0),\sigma(\lambda_1))\leq 
r_2$ then 
$$\dist_{A_\Gamma}(\sigma(\lambda_0)\cap h_0N(H'_0)^{+r},\sigma(\lambda_1)\cap h_0N(H'_0)^{+r})\leq r_{2.5}$$
(That is, if the $\sigma$ come close, then they come close in the subspace of $A_{\Gamma}$ corresponding to the middle vertex $h_0N(H'_0)=p(\delta_0)\cap p(\delta_1)$.)

Indeed, for $x_i\in \sigma(\lambda_i)$ such that $\dist_{A_\Gamma}(x_0,x_1)\leq r_2$, we can find $y_i$ as in the second claim above, and in view of both claims we have $\dist_{A_\Gamma}(y_1,y_2)\leq 2 t(s(r_2))+r_2$.

Pick $p\in \sigma(\lambda_0)\cap h_0N(H'_0)^{+r}$ and $q\in\sigma(\lambda_1)\cap h_0N(H'_0)^{+r})$ with $\dist_{A_\Gamma}(p,q)\leq r_{2.5}$. By translating the pair $p,q$ by an element $h_0H'_0h^{-1}$, we find points $p',q'$ so that
\begin{itemize}
 \item $\dist_{A_\Gamma}(p',q')\leq r_{2.5}$ (since we multiplied by an element of $A_\Gamma$,
 \item $\dist_{\Lambda_{g_0N(H_0)}}(\phi_{g_0N(H_0)}(p),\phi_{g_0N(H_0)}(p'))\leq B_0$ and $\dist_{\Lambda_{g_1N(H_1)}}(\phi_{g_1N(H_1)}(q),\phi_{g_1N(H_1)}(q'))\leq B_0$ (Condition \ref{item:bounded_cosets}),
 \item $\dist_{\Lambda_{h_0N(H'_0)}}(\mu_0,\phi_{h_0N(H'_0)}(p'))\leq B_1$ (Convention \ref{conv:constants}, coboundedness property of $B_1$).
 
 Furthermore, by Definition \ref{defn:blow_up_data} (blow-up data), the first and third items yield
$\dist_{\Lambda_{h_0N(H'_0)}}(\mu_0,\phi_{h_0N(H'_0)}(q))\leq L_0r_{2.5}+L_0+B_1.$
\end{itemize}

Pick $p''\in w(\delta_0)$ and $q''\in w(\delta_1)$. The we see that $p'$ and $p''$ are both points in $g_0N(H_0)^{+r}\cap h_0N(H'_0)^{+r}$ that map within bounded distance of $(\lambda_0,\mu_0)\in \Lambda_{g_0N(H_0)}\times \Lambda_{h_0N(H'_0)}$ under $\phi_{g_0N(H_0)}\times \phi_{h_0N(H'_0)}$. By the properness part of Claim \ref{claim:quasilines_proper}, there is a bound on $\dist_{A_{\Gamma}}(p'',p')$. A similar argument applies to $q''$, and therefore we get a bound on $\dist_{A_{\Gamma}}(p'',q'')$, as required.
\end{proof}

To use 
Theorem~\ref{thm:main criterion}, we will need the following: 

\begin{prop}\label{prop:W_qi_to_A_Gamma}
There exist $R_1,R_2\ge0$ such that the following holds for all $r_1\geq R_1,r_2\geq R_2$.  Let $W=W_{r_1,r_2}$.  Then $W$ is connected and the action of 
$A_\Gamma$ on $W$ is proper and cobounded. In particular, any 
orbit map $A_\Gamma\to W$ is a quasi-isometry.
\end{prop}

\begin{proof}

\textbf{The action is proper.}  Recall the $A_\Gamma$--equivariant coarse map $w:W^{(0)}\to A_\Gamma$ from Definition \ref{defn:little_w}. Using Lemma~\ref{lem:surjection_W}, choose 
$\delta_0\in W^{(0)}$ such that $1\in w(\delta_0)$. We show that $w$ is coarsely Lipschitz, for any $r_1,r_2$, which proves properness of the action since it provides a bound on $\dist_{A_{\Gamma}}(w(\gamma_1\delta_0),w(\gamma_2\delta_0))\geq \dist_{A_{\Gamma}}(\gamma_1,\gamma_2)-2B_2$ when $\dist_W(\gamma_1\delta_0,\gamma_2\delta_0)$ is bounded.

Let $r_1,r_2\geq 0$. Let $\delta_0,\delta_1$ be $W$--adjacent maximal simplices of $X$, where 
$\delta_i=\Delta(\alpha_i,\beta_i)$ and $\alpha_i,\beta_i$ are as in Definition~\ref{defn:W_edges}.

First suppose that $\delta_0,\delta_1$ are joined by a $W$--edge of type 2.  Then, by Lemma \ref{lem:close_levels}, we have 
$\dist_{A_\Gamma}(w(\delta_0),w(\delta_1))\leq r_3$.  

Second, suppose that $\delta_0,\delta_1$ are joined by a $W$--edge of type 1.  Then, by definition, we have the following: 
associated to the coset $h_0N(H_0')=h_1N(H_1')$, we have $\mu_0=\mu_1$, and 
associated to the coset $g_0N(H_0)=g_1N(H_1)$, we have $\lambda_0,\lambda_1\in\Lambda_{g_nN(H_0)}$ with 
$\dist_{\Lambda_{g_0N(H_0)}}(\lambda_0,\lambda_1)\leq r_1$. 

Consider the map $\phi_{g_0N(H_0)}\times \phi_{h_0N(H_0')}:g_0N(H_0)\cap h_0N(H_0')\to \Lambda_{g_0N(H_0)}\times 
\Lambda_{h_0N(H_0')}$, and let $f_0:\mathbb R_{\ge0}\to\mathbb R_{\ge0}$ be the (properness) function as in Claim \ref{claim:quasilines_proper} (which says that $\phi_{g_0N(H_0)}\times \phi_{h_0N(H_0')}$ is proper).  

Letting $x\in w(\delta_0)$ and $y\in w(\delta_1)$, we have by Lemma \ref{lem:basic_w_properties}:
$$\dist_{\Lambda_{g_0N(H_0)}}(\phi_{g_0N(H_0)}(x),\lambda_0)\leq B_1+B_0$$ and 
$$\dist_{\Lambda_{g_0N(H_0)}}(\phi_{g_0N(H_0)}(y),\lambda_1)\leq B_1+B_0.$$  In $\Lambda_{h_0N(H_0')}$, the images of $x,y$ 
are 
both $(B_1+B_0)$--close to $\mu_0=\mu_1$.

So, $$\dist_{A_\Gamma}(x,y)\leq f_0(4B_1+4B_0+r_1).$$  This completes the proof that $w$ is coarsely lipschitz.  As a side remark, the 
constants 
depend on $A_\Gamma$, the blow-up data, and the choice of $r_1,r_2$, but we have not yet needed to impose any restriction on 
$r_1, r_2$.

\textbf{Connectedness and coboundedness.}  We will show at the same time that, for sufficiently large $r_1,r_2$,  $W_{r_1,r_2}$ is connected and the action of $A_\Gamma$ is cobounded. Note that a set of representatives for the $A_\Gamma$-orbits can be taken to be $\{\delta'_i=\Delta(\alpha_i,\beta_i)\}$ where

\begin{itemize}
 \item the set $\{p(\delta'_i)\}$ of edges of $Y$ is finite (we can arrange this since there are finitely many orbits of edges of $Y$ by Corollary \ref{cor:Y_orbits}), and
 \item there exists a constant $C$ so that if $p(\delta'_i)=p(\delta'_j)$ then the following holds. Suppose that the endpoints of $p(\delta'_i)$ are $gN(H)$ and $hN(H')$ then $\dist_{\Lambda_{gN(H)}}(\lambda_i,\lambda_j)\leq C$, where $\alpha_k=\{gN(H),\lambda_k\}$, and similarly in $\Lambda_{hN(H')}$.
\end{itemize}

The second item can be arranged because the action of $\langle gHg^{-1},hH'h^{-1}\rangle$ on $\Lambda_{gN(H)}\times \Lambda_{h'N(H')}$ is cobounded by Claim \ref{claim:quasilines_proper} (``moreover'' part).

Since $Y$ is connected by Lemma \ref{lem:Y_conn} (recalling that we are assuming that $\Gamma$ is connected throughout), to show connectedness and coboundedness it suffices to show that there is a path in $W_{r_1,r_2}$ connecting $\delta_0$ to $\delta_1$ when either
 \begin{enumerate}
  \item $\delta_0$, $\delta_1$ are maximal simplices with $p(\delta_0)=p(\delta_1)$, or 
  \item $\delta_0$, $\delta_1$ are maximal simplices such that $p(\delta_0)$ intersects $p(\delta_1)$ at a vertex $gN(H)$,
 \end{enumerate}

 and that moreover in the first case the length of the path is bounded if we are in the situation of the second bullet point above.
 
 In the first case, it is easily seen that we can in fact connect the maximal simplices only using edges of type 1 (roughly, changing one vertex at a time moving a bounded amount in one of the relevant quasilines). This requires $r_1$ to be sufficiently large to move in the quasilines.

 Since there are only finitely many $gN(H)g^{-1}$-orbits of edges of $Y$ emanating from $gN(H)$ and $N(H)$ is finitely generated, there is a sequence of edges $p(\delta_0)=e_0,\dots,e_n=p(\delta_n)$ emanating from $gN(H)$ so that the possible pairs $(e_i,e_{i+1})$ belong to a fixed set of $A_\Gamma$-orbits of pairs of edges of $Y$. From here, we see that we can change the orders of quantifiers, meaning that it suffices to show that given maximal simplices $\delta_0,\delta_1$ of $X$ as in case 2, there exist $r_1$ and $r_2$ such that $\delta_0$ and $\delta_1$ are $W_{r_1,r_2}$-adjacent. 
 
 To show that this holds, denote $h_iN(H'_i)$ the other endpoint of $p(\delta_i)$, and observe that since the distance between the sets $gN(H)^{+r}\cap h_iN(H'_i)^{+r}$ is at most some $R_2$, for any $r_2\geq R_2$ there is a $W$-edge of type 2 connecting maximal simplices $\delta'_0,\delta'_1$ with $p(\delta_i)=\delta'_i$. We know already that $\delta_i$ can be connected to $\delta'_i$, so we are done.
\end{proof}

To save notation, we often denote $W_{r_1,r_2}$ by $W$, even though we will not fix $r_1,r_2$ until later.

\subsection{Fullness of links}
On our way to proving that $(X,W)$ is a combinatorial HHS, so as to apply Theorem~\ref{thm:main criterion}, we need to verify condition~\eqref{item:fullness} from 
Definition~\ref{defn:combinatorial_HHS}, which we do in the following proposition.

\begin{prop}[$W$-fullness of links]
	Let $\Delta$ be a non-maximal simplex of $X$, and let $v_1,v_2$ be two distinct non-adjacent vertices of 
$\link_X(\Delta)$  that are contained in $W$-adjacent maximal simplices. Then $v_1, v_2$ are contained in $W$-adjacent maximal 
simplices of $\Star_X(\Delta)$. 
	\label{prop: C=C0}
\end{prop}

Since $W$-edges come in two types, we split Proposition~\ref{prop: C=C0} in two. 

\begin{lemma}[Type 1 fullness]
Let $\Delta$ be a non-maximal simplex of $X$, and let $v_1,v_2$ be two distinct non-adjacent vertices of $\link_X(\Delta)$. 
Suppose that there exist maximal simplices $\Delta_1, \Delta_2$ of $X$ such that $v_1 \in \Delta_1, v_2 \in \Delta_2$, and 
$\Delta_1, \Delta_2$ are connected by a $W$-edge of type $1$. Then $v_1, v_2$ are contained in $W$-adjacent maximal simplices 
of  $\Star_X(\Delta)$. 
	\label{prop: C=C0 W1}
\end{lemma}

\begin{proof}
		Recall that $W$-edges of type $1$ connect maximal simplices with the same projection in $Y$.  Let $e=\bar\Delta_1=\bar\Delta_2$ be the edge of $Y$ to which $\Delta_1,\Delta_2$ project, and denote the endpoints of $e$ by $u$ and $v$.  Then $\Delta_i=\Delta(f^i_u,f^i_v)$, where $f^i_u$ is an edge with an endpoint $\lambda^1_u\in \Lambda_u$, and similarly for $v$.
Also, we have say, $f^1_u=f^2_u$ and $\dist_{\Lambda_v}(\lambda^1_v,\lambda^2_v)\leq r_1$.

Since $v_1,v_2$ are not adjacent in $X$, we have $v_1=\lambda_v^1$ and $v_2=\lambda_v^2$, since those are the only non-common vertices of $\Delta_1$ and $\Delta_2$.  Since $v_1,v_2\in\link_X(\Delta)$, Lemma~\ref{lem:join decomposition} implies that there exist a vertex $u'$ of $Y$ adjacent to $v$  
and an edge $f_{u'}$ of $X_{u'}$ such that $f_{u'}$ is contained in $\Star_X(\Delta)$.  By Definition~\ref{defn:W_edges}, the maximal simplices $\Delta(f_v^1,f_{u'})$ and $\Delta(f_v^2,f_{u'})$ are 
joined by a type $1$ edge in $W$.  These simplices respectively contain $v_1,v_2$ and lie in the star of $\Delta$, as 
required.
\end{proof}

\begin{lemma}[Type 2 fullness]
	Let $\Delta$ be a non-maximal simplex of $X$, and let $v_1,v_2$ be two distinct non-adjacent vertices of 
$\link_X(\Delta)$. Suppose that there exist maximal simplices $\Delta_1, \Delta_2$ of $X$ such that $v_1 \in \Delta_1, v_2 
\in \Delta_2$, and $\Delta_1, \Delta_2$ are connected by a $W$-edge of type $2$. Then $v_1, v_2$ are contained in 
$W$-adjacent maximal simplices of  $\Star_X(\Delta)$. 
	\label{prop: C=C0 W2}
\end{lemma}

\begin{proof}
	Since $\Delta_1$ and $\Delta_2$ are connected by a $W$-edge of type 2, the edges $\overline{\Delta_1}$ and 
$\overline{\Delta_2}$ share exactly one vertex $u$ of $Y$. Let us denote by $u_1, u_2$ the  vertices of 
$\overline{\Delta_1}, 
\overline{\Delta_2}$, respectively, that differ from $u$. 

By the definition of $W$--edges of type 2, there is an edge $f_u$ of $X_u$ and edges $f_{u_1},f_{u_2}$ of $X_{u_1},X_{u_2}$ such that $\Delta_i=\Delta(f_{u_i},f_u)$, for $i=1,2$.  Letting $\lambda_i$ be the endpoint of $f_{u_i}$ which is not $u_i$,  we also have $\dist_{A_\Gamma}(\sigma(\lambda_1),\sigma(\lambda_2))\leq r_2$.

\medskip

\textbf{Candidate simplices:}  Now, since $v_1,v_2$ are not adjacent in $X$, we have $v_i\in f_{u_i}$ for $i=1,2$.  Since $Y$ does not contain triangles, by Lemma~\ref{lem:no square}, the vertices 
$u_1,u_2$ are at distance $2$ in $Y$.  Since $v_1,v_2\in\link_X(\Delta)$, we therefore have by Lemma~\ref{lem:join decomposition} that $\overline{\Delta}$ is a 
vertex.  Since, by Lemma~\ref{lem:no square}, $Y$ does not contain squares, it follows 
that $\overline{\Delta}=u$. 

Let $f_u'$ be an edge of $X_u$ containing $\Delta$.  Let $\delta_i=\Delta(f_u',f_{u_i})$ for $i=1,2$.  So, $\delta_1,\delta_2$ are maximal simplices that share the edge $f'_u$ of $X_u$ and project to 
$\bar\Delta_1,\bar\Delta_2$ respectively.  Note that $\delta_1$ and $\delta_2$ contain $f_u'$ and hence contain $\Delta$.  Thus $\delta_1,\delta_2\subset\Star_X(\Delta)$.

\medskip

\textbf{$W$--adjacency:} We are only left to observe that there is a $W$-edge of type 2 connecting $\delta_1$ and $\delta_2$, since the $\sigma$ sets associated to $\delta_i$ and $\Delta_i$ coincide. (That is, roughly, in the definition of edges of type $2$ the ``middle'' edge does not play a role.)
\end{proof}

Now the proposition is immediate:

\begin{proof}[Proof of Proposition~\ref{prop: C=C0}]
Combine Lemma~\ref{prop: C=C0 W1} and Lemma~\ref{prop: C=C0 W2}.
\end{proof}

\subsection{Maps between augmented complexes}

Recall that the augmented complex $X^{+W}$ was defined in Definition~\ref{def:aug_X}.  
We will also use a 
corresponding construction in $Y$:

\begin{defn}
	We define the \textbf{augmented graph} $Y^{+W}$ which is obtained from $Y$ by adding a $W$-edge between two vertices 
$v, v'$ of $Y$ whenever their fibres $X_v, X_{v'}$ contain vertices that are connected by a $W$-edge.
In that case, we say that $v, v'$ are \textbf{connected by a $W$-edge}.
	Similarly, for a full subgraph $Y_0$ of $Y$, we denote by $Y_0^{+W} $  the full subgraph of $Y^{+W}$ induced by 
$Y_0$.
\end{defn}

\begin{rem}\label{rem:2-Lipschitz}
Note that vertices of $Y$ connected by a $W$-edge lie at distance at most $2$ from each other, by definition of the $W$-edges.
\end{rem}

The simplicial map $p$ mentioned below is defined in Definition \ref{def:proj_blowup}, while $\iota$ is defined in  Lemma \ref{lem:iota_vertex} and Definition \ref{def:iota}. Recall from Lemma \ref{lem:iota QI} that $\iota$ is a simplicial quasi-isometry, and recall also that $q$ is just the composition of $p$ and $\iota$.

\begin{defn}\label{defn:augmented_maps}
	The simplicial maps 
	$$p: X \rightarrow Y, ~~~ q: X \rightarrow \widehat{D}, ~~~ \iota: Y \rightarrow \widehat{D} $$
	extend to maps that we still denote by 
	$$p: X^{+W} \rightarrow Y^{+W}, ~~~ q: X^{+W} \rightarrow \widehat{D}, ~~~ \iota: Y^{+W} \rightarrow \widehat{D}.  $$
	The extended map $p: X^{+W} \rightarrow Y^{+W}$ is still simplicial, while the other two can be taken to $2$-Lipschitz by Remark \ref{rem:2-Lipschitz}

Again by Remark \ref{rem:2-Lipschitz} the inclusion $Y\hookrightarrow Y^{+W}$ is an $A_\Gamma$--equivariant quasi-isometry, so 
$\iota:Y^{+W}\to\widehat D$ is an $A_\Gamma$--equivariant quasi-isometry.  Similarly, $p:X^{+W}\to Y^{+W}$ is an 
$A_\Gamma$--equivariant quasi-isometry.  Hence, so is $q:X^{+W}\to\widehat D$.
\end{defn}

We will need the following preliminary lemma, in a similar spirit to Proposition \ref{prop:W_qi_to_A_Gamma}.

\begin{lemma}\label{lem:links_connected}
 For any sufficiently large $r_2$ the following holds. Let $v$ be a vertex of $Y$ and let $\Delta$ be an edge of $X$ contained in a fibre $X_v$. Then the augmented links $\link_Y(v)^{+W}$ and $\link_X(\Delta)^{+W}$ are connected, and $p$ restricts to a quasi-isometry of said augmented links.
\end{lemma}

\begin{proof} 
\textbf{Connectedness:} We start by showing that $\link_Y(v)^{+W}$ is connected. 

Set $v'\coloneqq \iota(v)$. By Lemmas~\ref{lem:link_tree_cocompact} and~\ref{lem:link_dihedral_cocompact} , the action of $\stabilizer(v')$ on $\link_{\widehat{D}}(v')$ is cocompact.  We choose a finite connected subcomplex $K$ of $\link_{\widehat{D}}(v')$ such that its $\stabilizer(v')$-orbits cover $\link_{\widehat{D}}(v')$. We can further assume that $K$ satisfies the following additional property: Whenever $K$ contains a vertex $w$ that is not of tree or dihedral type, then it also contains every vertex of tree or dihedral type adjacent to it. Indeed, vertices of $\widehat D$ that are not of dihedral or tree type are cosets of the form $g\{1\}$ or $g\langle a \rangle$, and they have only finitely many neighbours in $D$ that are vertices of dihedral type by construction of the modified Deligne complex. Moreover, the cone-off procedure adds no additional neighbour for vertices of the form $g\{1\}$ and add only one additional neighbour for vertices of the form $g\langle a \rangle$ (namely, the apex of the standard tree $gT_a$). Without loss of generality, we can thus assume that $K$ contains all such neighbours.

Let $S$ be the set of vertices of $K$ that are of dihedral or tree type. The set $\iota^{-1}(S)$ forms a finite set of vertices of $Y$ that are adjacent to $v$ by Lemma~\ref{lem:iota}. Let us choose an edge $e$ of $X_v$. For each $w \in \iota^{-1}(S)$, we choose a maximal simplex $\Delta_w$ of $X$ above the edge $vw$ of $Y$ and whose fibre over $v$ is the edge $e$. Since the action of $A_\Gamma$ on $Y$ is cocompact by Lemma~\ref{lem:cocompact_Y}, we can now require the constant $r_2$ in the definition of $W$-edges (see Definition~\ref{defn:W_edges}) to be large enough so that for every distinct $w, w'$ in $\iota^{-1}(S)$, we have that $\Delta_w$ and $\Delta_{w'}$ are connected by a $W$-edge of type $2$. In particular, all pairs of vertices of $\iota^{-1}(S)$ are connected by an edge of $Y^{+W}$. 

Let us now show that $\link_Y(v)^{+W}$ is connected. Indeed, if $x, x'$ are two vertices of $\link_Y(v)^{+W}$, then by definition of the subcomplex $K$ we can find a finite sequence of elements $g_1, \ldots, g_n \in A_\Gamma$ such that $\iota(x) \in g_1K$, $\iota(x')\in g_nK$ and for every $1\leq i <n$, $g_iK \cap g_{i+1}K \neq \varnothing$. By the additional property imposed on $K$, it follows that each such $g_iK \cap g_{i+1}K $ contains a vertex $v_i'$ of $\widehat{D}$ that is either of tree or dihedral type. By Lemma~\ref{lem:iota}, we can define $v_i \coloneqq \iota^{-1}(v_i')$. We thus have a sequence of vertices $v_0\coloneqq x, v_1, \ldots, v_{n-1}, v_n\coloneqq x'$ and by construction, for every $0\leq i <n$ we have that $g_{i+1}^{-1}v_i$ and $g_{i+1}^{-1}v_{i+1}$ are in $\iota^{-1}(S)$, hence  $v_i$ and $v_{i+1}$ are either equal or connected by an edge of $Y^{+W}$.   This show that $\link_Y(v)^{+W}$ is connected.

Let us now show that $\link_X(\Delta)^{+W}$ is connected. Since $\Delta$ is an edge contained in the fibre $X_v$, it follows from Lemma~\ref{lem:join decomposition} that $\link_X(\Delta)$ is the disjoint union of all the fibres $X_w$ for $w\in \link_Y(v)$. For every vertex $w$ of $\link_Y(v)$, the fibre $X_w$ is connected by construction. Since in addition $\link_Y(v)^{+W}$ is connected by the above, it is now straightforward to check that $\link_X(\Delta)^{+W}$ is also connected.

\medskip

\textbf{Quasi-isometry:} 	By construction of $Y^{+W}$, for every $e$ edge of $\link_Y(v)^{+W}$, there exists an edge $e'$ in $\link_X(\Delta)^{+W}$ such that $p(e')=e$. Thus, we can construct a set-theoretic section $\overline{p}: \link_Y(v)^{+W} \rightarrow \link_X(\Delta)^{+W}$ by choosing for every point  $x\in \link_Y(v)^{+W}$ a point in $p^{-1}(x)$. Let us show that  $\overline{p}$ is a quasi-inverse of $p$. 

We have $p\circ \overline{p} = Id$ by construction. Since all fibres of the form $X_w$ are connected of diameter $2$ by construction, it follows that for every vertex $w \in \link_Y(v)$ and every point $x \in X_w$, we have $\dist(x, \overline{p}\circ p(x))\leq 2$. Moreover,  for every other point $x \in \link_X(\Delta)^{+W}$, $x$ is thus contained in an edge with endpoints in two distinct fibres $X_w$ and $X_{w'}$, so we have that $\overline{p}\circ p(x)$ is contained in a possibly different edge between $X_w$ and $X_{w'}$. It follows that $$\dist(x, \overline{p}\circ p(x))\leq  2 + \mbox{diam}(X_w)\leq 4.$$ 
	
	Moreover, both $p$ and $\overline{p}$ are coarsely Lipschitz. That $p$ is coarsely Lipschitz follows from the fact that it sends an edge of $\link_X(\Delta)^{+W}$  to an edge or a vertex of $\link_Y(v)^{+W}$. A similar argument as above shows that if $x, y$ are two points belonging to some edge of $Y$, then $\dist(\overline{p}(x), \overline{p}(y))\leq 4$. Thus, $\overline{p}$ is coarsely Lipschitz, and it then follows that $p$ is a quasi-isometry.
\end{proof}

The following lemma is crucial as it allows us to compare links in $X$ (which are the ones we are interested in) to links in $\widehat D$ (which are the ones we understand).

\begin{lemma}\label{lem: QI link}
	For any sufficiently large $r_1,r_2$ the following holds for $W=W_{r_1,r_2}$: Let $v$ be a vertex of $Y$. Then the quasi-isometry $$\iota: Y^{+W} \rightarrow 
\widehat{D}$$ induces a 
quasi-isometry between $\link_Y(v)^{+W}$ and $\link_{\widehat{D}}(\iota(v))$. 

 Moreover, let $\Delta$ be an edge of $X$ contained in a fibre $X_v$, for some vertex $v$ of $Y$.  Then the quasi-isometry $$q: X^{+W} \rightarrow \widehat{D}$$ restricts to a 
quasi-isometry  $q: (\link_X(\Delta)^{+W})^{(0)}\rightarrow \link_{\widehat{D}}(\iota(v))$. 
\end{lemma}

\begin{proof}
Using the actions of $A_\Gamma$, we can and will assume that $v=N(H)$ for some $H\in\mathcal H$. We first notice that the map $\iota$ induces a $\mathrm{Stab}(v)$--equivariant bijection between the vertices of  $\link_Y(v)^{+W}$ and the vertices of 
$\link_{\widehat{D}}(\iota(v))$ of tree or dihedral type. Indeed, this follows from Lemma~\ref{lem:iota}, the fact that the bijection between tree-type/dihedral-type sets of vertices of $Y$ and $\widehat{D}$ stated in Lemma~\ref{lem:iota} passes to links is a direct consequence of the characterisation of edges of $Y$ therein.  

 We remark that $\link_{\widehat{D}}(\iota(v))$ is connected, since it is either a tree when $v$ is of tree type or we can use Remark \ref{rem:graph_of_orbits} and Corollary \ref{cor: link coneoff}. We can thus extend the restriction $\iota: \link_Y(v)\rightarrow \link_{\widehat{D}}(\iota(v))$ into a map $\iota: \link_Y(v)^{+W}\rightarrow \link_{\widehat{D}}(\iota(v))$ by sending a $W$-edge between two distinct vertices  $w, w' \in \link_Y(v)$ to a geodesic path between $\iota(w)$ and $\iota(w')$. Let us show that this extension is coarsely Lipschitz.

Since $\iota$ is $\mathrm{Stab}(v)$-equivariant, it suffices to show that there are finitely many $\mathrm{Stab}(v)$-orbits of $W$-edges in $\link_Y(v)^{+W}$ (note that a $G$-equivariant map between graphs $\Gamma_1\to\Gamma_2$ is coarsely Lipschitz provided that $\Gamma_2$ is connected and there are finitely many $G$-orbits of edges in $\Gamma_1$). Since there are finitely many $\mathrm{Stab}(v)$-orbits of vertices in $\link_Y(v)$ (Lemma \ref{lem:I_finite}), it suffices to fix a vertex $z=h_1N(H_1)\in \link_Y(v)$ and show that there are finitely many $(\mathrm{Stab}(v)\cap \mathrm{Stab}(z))$-orbits of $W$-edges in $\link_Y(v)^{+W}$ containing $z$. Consider some vertex $z'=h_2N(H_2)\in \link_Y(v)$ connected by a $W$-edge to $z$, and notice that the $W$-edge needs to be of type 2 (since $W$-edges of type 1 connect maximal simplices that project to the same edge in $Y$). By definition of $W$-edges, we must have
$$\dist_{A_\Gamma}(N(H)^{+r}\cap h_1N(H_1)^{+r},N(H)^{+r}\cap h_2N(H_2)^{+r})\leq r_2.\ \ \ (*)$$
Since $K=\mathrm{Stab}(v)\cap \mathrm{Stab}(z)=N(H)\cap h_1N(H_1)h_1^{-1}$ acts coboundedly on $N(H)^{+r}\cap h_1N(H_1)^{+r}$ (Lemma \ref{lem:intersection_of_cosets}) and $\mathcal H$ is finite, there are finitely many $K$-orbits of cosets $h_2N(H_2)$ satisfying $(*)$, there must be finitely many $K$-orbits of vertices $z'$ as required.

Let us now construct a coarsely Lipschitz quasi-inverse. By Lemma~\ref{lem:link_dihedral_cocompact} or Lemma~\ref{lem:link_tree_cocompact}, the action of $\stabilizer(\iota(v))$ on $\link_{\widehat D}(\iota(v))$ has finitely many orbits of edges and $\link_Y(v)$ is connected by Lemma \ref{lem:links_connected}, and as above this suffices. 

It thus follows that $\iota: \link_Y(v)^{+W}\rightarrow \link_{\widehat{D}}(\iota(v))$ is a quasi-isometry.

Regarding the moreover clause, in view of Lemma \ref{lem:links_connected}, $q=\iota\circ p$ restricts to a composition of quasi-isometries of the relevant links.  
\end{proof}

In this section we will often make the following abuse of notation: 

\begin{conv}Let $A$ be a subcomplex of a simplicial complex $B$. Then by $B-A$ we mean the full subcomplex of $B$ with vertex set $B^{(0)}-A^{(0)}$. Note for example that with this abuse we have $\big(X-\sat_X(\Delta)\big)^{+W}=\big(X^{(0)}-\sat_X(\Delta)\big)^{+W}$.
\end{conv}

We also mention the following related result that will be needed in the Section~\ref{subsec:QIembedded}:

\begin{lemma}\label{lem:coarseLipSat}
	For any sufficiently large $r_1,r_2$ the following holds. Let $\Delta$ be an edge of $X$ contained in a fibre $X_v$, for some vertex $v$ of $Y$. 
		If  the presentation graph $\Gamma$ is a single edge, we have $X-\sat_X(\Delta) = \link_X(\Delta).$
	Otherwise, $X-\sat_X(\Delta) $ is the full subcomplex of $X$ spanned by $X-X_v$, and the  restriction $$q: X-\sat_X(\Delta)\rightarrow \widehat{D} - \iota(v)$$ is well-defined and extends to a coarsely Lipschitz map $$\overline{q}: (X-\sat_X(\Delta))^{+W}\rightarrow \widehat{D} - \iota(v).$$ 
\end{lemma}

\begin{proof}
\textbf{Case where $\Gamma$ is a single edge:}	Let us first consider the case where $\Gamma$ is a single edge between the two standard generators $a, b$. Then by construction, $Y$ contains a unique vertex of dihedral type, namely $u_{ab}$. 
	
	Let us first assume that $\Delta$ is contained in the fibre of a vertex of tree type. Since $Y$ is bipartite with respect to the vertex type (tree or dihedral) by Lemma~\ref{lem:iota}, it follows that for every vertex $u$ of $Y$ of tree type, its link consists of the single vertex $u_{ab}$. It now follows from Lemma~\ref{lem:join decomposition} that the link of any edge contained in the fibre of a vertex of tree type is equal to the fibre $X_{u_{ab}}$. Moreover, no vertex of $X_{u_{ab}}$ has $X_{u_{ab}}$ in its link, so we get that  $$X-\sat_X(\Delta) = X_{u_{ab}}= \link_X(\Delta).$$
	
	Let us now assume that $\Delta$ is contained in the fibre $X_{u_{ab}}$. It follows from Lemma~\ref{lem:join decomposition} that the link of any edge of $X_{u_{ab}}$ is exactly $X-X_{u_{ab}}$. It also follows from the above discussion that no other simplex of $X$ has link $X-X_{u_{ab}}$, hence  
	$$X-\sat_X(\Delta) = X-X_{u_{ab}}= \link_X(\Delta).$$

\textbf{General case:} Let us now consider the case where $\Gamma$ is not a single edge. Let us first show that no other vertex of  $Y$ has the same link as $v$. Assume by contradiction that there exists another such vertex $v'$. Since $Y$ does not contain squares by Lemma~\ref{lem:no square}, we have that $\link_Y(v) = \link_Y(v')$ is a single vertex. Since a vertex of $Y$ of dihedral type $gv_{ab}$ contains at least two vertices in its link, namely $gv_a$ and $gv_b$ by Lemma~\ref{lem:iota}, it follows that $v$ and $v'$ are of tree type, say $v=gv_a$ and $v'=hv_b$. 
This now forces $a$ and $b$ to be leaves of $\Gamma$: Indeed, if $a$ had at least two neighbours $c, d$ in $\Gamma$, then $\link_Y(gv_a)$ would contain the two distinct vertices $gv_{ac}$ and $gv_{ad}$. Thus, $a$ and $b$ are leaves of $\Gamma$. Moreover, since $gv_a$ and $hv_{b}$ have a common neighbour that is of dihedral type, this implies that this common neighbour is of the form $kv_{ab}$, and in particular $a$ and $b$ are adjacent in $\Gamma$. Since $\Gamma$ is connected and $a$ and $b$ are two adjacent leaves of $\Gamma$, it follows that $\Gamma$ is a single edge, a contradiction. Thus, no other vertex of $Y$ has the same link as $v$.

By Lemma~\ref{lem:join decomposition}, the link of $\Delta$ is the disjoint union of all the fibres of the form $X_w$ for $w$ a vertex of $ \link_Y(v)$. Moreover, since distinct vertices of $Y$ do not have the same link by the above, a simplex of $X$ has the same link as $\Delta$ if and only if it is an edge of $X_{v}$. Thus,  $X-\sat_X(\Delta)$ is the full subcomplex of $X$ obtained by 
removing the fibre  $X_{v}$. 

It follows from the above description of $X-\sat_X(\Delta)$ that the restriction map
$$q: X-\sat_X(\Delta)\rightarrow \widehat{D} - \iota(v)$$
is well-defined (recall that $q$ is simplicial on $X$, so it is enough to check that the restriction is well-defined on the 0-skeleton).

Since $\widehat{D} - \iota(v)$ is connected (since the link of $\iota(v)$ is connected as explained in the proof of Lemma \ref{lem: QI link}), we extend this map to a map 
$$\overline{q}: (X-\sat_X(\Delta))^{+W}\rightarrow \widehat{D} - \iota(v)$$ 
by sending edges to geodesics in $\widehat{D} - \iota(v)$.

To show that this map is coarsely Lipschitz, it is enough to find a uniform constant $C$ such that for two distinct vertices $w, w'$ of $Y-v$ connected by a $W$-edge, we have $\dist_{\widehat{D} - \iota(v)}(\iota(w), \iota(w'))\leq C$. Note that since $w$ and $w'$ are distinct, this $W$-edge must be of type $2$, and in particular $w$ and $w'$ are at distance at most $2$ in $Y$. Since $Y$ does not contain triangles and squares by Lemma~\ref{lem:no square}, this distance is exactly $2$ and we can consider the midpoint $w''$. 

If $w'' \in Y-v$, then $\iota(w'')\in \widehat{D}-\iota(v)$ since $\iota$ is injective by Lemma \ref{lem:iota}, and we thus have $\dist_{\widehat{D} - \iota(v)}(\iota(w), \iota(w'))\leq 2$.

 Otherwise, $w''=v$ and so $w, w' $ are adjacent vertices of $\link_Y(v)^{+W}$. 
It follows from Lemma~\ref{lem: QI link} that the maps $p: \link_X(\Delta)^{+W}\rightarrow  \link_Y(v)^{+W}$ and  $q: \link_X(\Delta)^{+W}\rightarrow \link_{\widehat{D}}(\iota(v))$ are quasi-isometries. Since $w, w' $ are adjacent vertices of $\link_Y(v)^{+W}$, it follows that there exists a constant $C$ (depending only on $A_\Gamma$) such that $\dist_{\link_{\widehat D}(\iota(v))}(\iota(w), \iota(w'))\leq C$, and hence $\dist_{\widehat{D} - \iota(v)}(\iota(w), \iota(w'))\leq C$. This shows that the map $\bar q: (X-\sat_X(\Delta))^{+W}\rightarrow \widehat{D} - \iota(v)$ is coarsely Lipschitz.
\end{proof}

We will later prove a similar statement for simplices $\Delta$ that are triangles of $X$, see Lemma~\ref{lem:coarseLipSatTriangle}.

\section{The geometry of the augmented complex}\label{sec:augment_geom}

In this section we finish the proof of Theorem \ref{thmintro:main}, by verifying that all conditions of Theorem \ref{thm:main criterion} apply to the pair $(X,W)$ where $X$ is the blown-up commutation graph (Definition \ref{defn:blow_up_commute}) and $W=W_{r_1,r_2}$ is described in Subsection \ref{subsec:augmented}. Up until Subsection \ref{subsec:assembling} (the final argument) we work under Convention \ref{conv:connected}, that is, we work with an Artin group of large and hyperbolic type $A_\Gamma$ with $\Gamma$ connected and not a single vertex. 

\subsection{Hyperbolicity of the augmented links}
The goal of this section is to prove the following proposition, verifying part of condition~\eqref{item:hyperbolic_links} from Definition~\ref{defn:combinatorial_HHS} for the candidate combinatorial 
HHS $(X,W)$:

\begin{prop}
	For all sufficiently large $r_1,r_2$ the following holds. The augmented links $\link_X(\Delta)^{+W}$ of all non-maximal simplices  $\Delta$ of $X$ are connected and 
hyperbolic.
	\label{prop:hyp_link}
\end{prop}

Before the proof, we assemble the ingredients.

Let $\Delta$ be a simplex of $\Delta$. If $\link_X(\Delta)$ is connected and of bounded diameter, so is its augmented link 
and there is nothing to prove. Thus, we only need focus on unbounded links. 

\begin{lemma}
	Let $\Delta$ be a non-maximal simplex of $X$. Then its link (and hence its augmented link) is connected and of 
bounded diameter, except possibly when $\Delta$ is one of the following:
	\begin{itemize}
		\item the empty simplex,
		\item an edge of $X$ contained in a fibre,
		\item a triangle of $X$ containing exactly two roots.
	\end{itemize}
	\label{lem:disconnected links}
\end{lemma}

\begin{proof}
	Let $\Delta$ be a non-empty simplex of $X$. Note that if a link decomposes as a non-trivial join, it is automatically 
connected and bounded.  In view of Lemma \ref{lem:join decomposition}, $\link_X(\Delta)$ is connected and bounded, by virtue of decomposing non-trivially as a join, except possibly in one of the following two situations.

	\textbf{Case 1:} $\link_Y(\overline{\Delta}) \neq \varnothing$ and all the links of fibres are empty. In that case, 
$\overline{\Delta}$ is a single vertex $v$ of $Y$, and its associated fibre $\Delta_v = \Delta$ is an edge (and this case is listed in the statement).
	
	\textbf{Case 2:} $\link_Y(\overline{\Delta}) = \varnothing$ and exactly one of the fibres of $\Delta$ has an empty link. In 
that case, $\overline{\Delta}$ is an edge, and $\Delta$ is a triangle (since it is non-maximal). If $\Delta$ contains only one root, then $\link_X(\Delta)$ is a single vertex, so it is bounded. (The remaining case of two root is listed in the statement.)
\end{proof}

We now treat the remaining cases separately.

\begin{lemma}\label{lem:maximal_Deligne}
	The augmented link of the empty simplex of $X$ is equivariantly quasi-isometric to $\widehat{D}$, and in particular is hyperbolic.
\end{lemma}

\begin{proof}
	By definition, the link of the empty simplex of $X$ is $X$ itself, which is quasi-isometric to $\widehat{D}$ by Lemma \ref{lem:q_qi}. Note that $X^{+W}$ is quasi-isometric to $X$ since $W$-edges connect vertices at distance at most $2$ in $X$ by Definition \ref{defn:W_edges}.
\end{proof}

\begin{lemma}\label{lem:quasiline}
	Suppose $r_1\geq 1$. Let $\Delta$ be a triangle of $X$ containing exactly two  roots. Then its augmented link $\link_X(\Delta)^{+W}$ is  
hyperbolic, and in fact quasi-isometric to $\mathbb Z$.
\end{lemma}

\begin{proof}
	Let $v$ be the vertex of $Y$ such that $\Delta_v$ is a single root. By Lemma \ref{lem:join decomposition}, the link 
of $\Delta$ is the disjoint union of the leaves of $X_v$. By Proposition \ref{prop: C=C0}, the augmented link 
$\link_X(\Delta)^{+W}$ is obtained from $\link_X(\Delta)$ by adding an edge between all vertices of $\link_X(\Delta)$ that 
are connected by a $W$-edge coming from maximal simplices of $\Star_X(\Delta)$. By construction of $W$-edges of type 1, this 
graph is connected and quasi-isometric to $\mathbb{Z}$, provided $r_1\geq 1$.
Indeed, given a graph and $C\geq1$, the graph with the same vertex set and edges connecting pairs of vertices within distance $C$ of each other in the original graph is quasi-isometric to the original graph.
\end{proof}

\begin{lemma}\label{lem:quasitree}
	Let $\Delta$ be an edge of $X$ contained in a fibre $X_v$ for some vertex $v$ of $Y$. Then its augmented link $\link_X(\Delta)^{+W}$ is a connected quasi-tree. Moreover, it is unbounded except when $v$ corresponds to a standard generator of $A_\Gamma$ that comes from a leaf of $\Gamma$ contained in an edge of $\Gamma$ with even label.
\end{lemma}

\begin{proof}
By Lemma \ref{lem: QI link} we have that $\link_X(\Delta)^{+W}$ is quasi-isometric to $\link_{\widehat{D}}(\iota(v))$ where $v\coloneqq \overline{\Delta}$. Therefore, we have to prove that $\link_{\widehat{D}}(\iota(v))$ is a quasi-tree. There are two cases.

\textbf{Tree type:} Suppose that $v$ is a vertex of $Y$ of tree type.
By construction of $\widehat{D}$ (or Lemma \ref{lem:link_tree_cocompact}) we have that $\link_{\widehat{D}}(\iota(v))$ is a standard tree of $D$, and the characterisation of when standard trees are unbounded follows directly from the graph of groups structure given in  \cite[Remark~4.6]{MP}. In particular, it follows from that description that when $v$ corresponds to a standard generator of $A_\Gamma$ that comes from a leaf of $\Gamma$ contained in an edge of $\Gamma$ with even label, then the corresponding standard tree has diameter $2$.
\\

\textbf{Dihedral type:}  Suppose that $v$ is the vertex $v_{ab}$ of $Y$ of dihedral type for some $a, b \in \Gamma$. 

By Corollary~\ref{cor: link coneoff}, the action of $\stabilizer(v)=A_{ab}$ on $\link_{\widehat{D}}(\iota(v))$ is cocompact. Let us describe the maximal stabilisers for this action. Since the action is without inversion, it is enough to describe vertex stabilisers. By Lemma~\ref{lem:link_dihedral_cocompact}, the stabiliser of a vertex of $\link_D(\iota(v))$ is either trivial or conjugate to $\langle a \rangle$ or $\langle b \rangle$. Moreover, it follows from Lemma~\ref{lem:stab_tree} that the stabiliser of the vertex of tree type corresponding to the standard tree $T_a$ is equal to $A_{ab} \cap C(a)$, i.e. equal to the centraliser of $a$ in $A_{ab}$. By \cite[Lemma~7]{Crisp}, this subgroup is equal to $\langle a, z_{ab} \rangle$. Thus, the stabilisers of vertices of tree type in  $\link_{\widehat{D}}(\iota(v))$ are conjugate (again in $A_{ab}$) to $ \langle a, z_{ab}\rangle $ or $  \langle b, 
z_{ab}\rangle$.

It follows from the above that the maximal stabilisers for the cocompact action of $A_{ab}$ on $\link_{\widehat{D}}(\iota(v)) $ are the conjugates of $ \langle a, z_{ab}\rangle $ and $  \langle b, 
z_{ab}\rangle$. It now follows from 
	 \cite[Thm 5.1]{CharneyCrisp}  that $\link_{\widehat{D}}(\iota(v))$ is quasi-isometric to the Cayley graph 
$\mbox{Cayley}(A_{ab};  \langle a, z_{ab}\rangle ,   \langle b, z_{ab}\rangle ).$
	Note that this Cayley graph is quasi-isometric to 
	$\mbox{Cayley}(A_{ab}/ \langle z_{ab}\rangle;  \langle a\rangle  ,  \langle b\rangle ).$ Moreover, since $A_{ab}$ 
contains a finite-index subgroup isomorphic to $ \langle z_{ab}\rangle \times F$ with $F$ a finitely-generated free group by Lemma~\ref{lem:virtually splits}, the central  quotient  $A_{ab}/ \langle  
z_{ab}\rangle $ is virtually free. 
By \cite[Theorem 2.19]{Balasubramanya_quasitree} (which applies since $\langle a\rangle  $  and $\langle b\rangle$ are quasiconvex in the Cayley graph of $A_{ab}/ \langle  
z_{ab}\rangle $), the Cayley graph $\mbox{Cayley}(A_{ab}/ \langle 
z_{ab}\rangle ;  \langle a\rangle  ,  \langle b\rangle )$ is a quasitree, hence so are $\link_{\widehat{D}}(\iota(v))$ and  $\link_X(\Delta)^{+W}$. 

Unboundedness follows from \cite[Lemma 5.12]{Osin} and results in \cite{DGO} using the same argument as in \cite[Corollary 3.22]{Balasubramanya_quasitree}.
\end{proof}

\begin{cor}\label{cor:unbounded link}
	The simplices of $X$ that have an unbounded augmented link are exactly: \begin{itemize}
		\item the empty simplex of $X$,
		\item the triangles of $X$ containing exactly two roots.
		\item the edges contained in a fibre $X_v$ of $X$, except when $v$ corresponds to a standard generator of $A_\Gamma$ that comes from a leaf of $\Gamma$ contained in an edge of $\Gamma$ with even label.
	\end{itemize}
All other simplices of $X$ have an augmented link of diameter at most $3$. \qed
\end{cor}

\subsection{Quasi-isometric embeddings of augmented links}\label{subsec:QIembedded}
It remains to verify that $(X,W)$ satisfies the quasi-isometric embedding part of condition~\eqref{item:hyperbolic_links} from Definition~\ref{defn:combinatorial_HHS}.  This is achieved by the 
following proposition:

\begin{prop}
	For every non-maximal simplex $\Delta$ of $X$, the augmented link $\link_X(\Delta)^{+W}$ is quasi-isometrically 
embedded in the augmented complex $\big(X^{(0)}-\sat_X(\Delta)\big)^{+W}.$
	\label{prop:QI}
\end{prop}

The proposition is clear when the augmented link is bounded. We treat the remaining cases, listed in Corollary \ref{cor:unbounded link}, separately.  Note that there is nothing to do in the case of 
the empty simplex, as the inclusion $\link_X(\emptyset)^{+W} 
\hookrightarrow \big(X-\sat_X(\emptyset)\big)^{+W}$ is the identity map of $X^{+W}$.\\

\textbf{Overview.} Let us give an overview of the general strategy. We will use the quasi-isometries between $X^{+W}, Y^{+W}$ and 
$\widehat{D}$ (see Lemma \ref{lem:iota QI} and Definition \ref{defn:augmented_maps}) to construct models for $\link_X(\Delta)^{+W}$ and $ \big(X-\sat_X(\Delta)\big)^{+W}$ that are easier to work 
with. More precisely, we will construct a  diagram of the form 

	\[	\xymatrixcolsep{7pc}\xymatrix{
	\link_X(\Delta)^{+W} \ar@{^{(}->}[d] \ar[r] & U_\Delta    \ar@{^{(}->}[d] \\
	\big(X-\sat_X(\Delta)\big)^{+W} \ar[r]  & V_\Delta} \]

where 
\begin{itemize}
 \item the complexes $U_\Delta, V_\Delta$ are built out of subcomplexes of $\widehat{D}$, 
 \item the top horizontal map is a quasi-isometry (in the spirit of Lemma \ref{lem: QI link}),
  \item the bottom horizontal map is coarsely lipschitz (in the spirit of Lemma \ref{lem:coarseLipSat}).
  \item both horizontal maps coincide with the map $q$ when restricted to vertex sets,
  	\item the diagram commutes up to bounded error. This is due to some bounded choices appearing when constructing the various horizontal arrows (and more precisely extending the map $q$ to $W$-edges), see the proof of Lemma~\ref{lem: QI link}, \ref{lem:coarseLipSat}, and~\ref{lem:coarseLipSatTriangle}.
\end{itemize}

Then, it will be enough to show that $U_\Delta \hookrightarrow V_\Delta$ is a quasi-isometric 
embedding.  
This in turn will be done using either the CAT(0) geometry 
of $\widehat{D}$ to construct a coarsely Lipschitz retraction $V_\Delta \rightarrow U_\Delta$ (in Lemma 
\ref{lem:QI2}), or by using additional properties of dihedral Artin groups when the complexes $U_\Delta, V_\Delta$ are 
defined using neighbourhoods of vertices of dihedral type (in Lemma \ref{lem:QI3} and \ref{lem:QI4}).\\

We start by introducing the main tool for constructing coarsely Lipschitz retractions in $\widehat{D}$, and more generally in 
CAT(0) complexes:

\newcommand{\No}[1]{\mathring{N}(#1)}

\begin{defn}[Combinatorial neighbourhood] Let $C$ be a subcomplex of a complex $\calX$.  The \textbf{combinatorial 
neighbourhood} of $C$, denoted $N(C)$, is the subcomplex consisting of the union of all the simplices of $\calX$ intersecting $C$. The 
\textbf{combinatorial sphere} around $C$, denoted $\partial N(C)$, is the full subcomplex with vertex set $N(C)^{(0)}-C^{(0)}$. Finally, we let $\No{C}=N(C)-\partial N(C)$.
\end{defn}

For $C$ a subcomplex of a metric complex $\calX$, note that $\calX-\No{C}$ is the full subcomplex spanned by $\calX-C$, meaning the full subcomplex with vertex set the vertices of $\calX$ that are not contained in $C$.

In the statement below we regard $\calX-\No{C}$ as endowed with its intrinsic path metric, and similarly for $\partial N(C)$. 

\begin{lemma}\label{lem:coarse retraction}
	Let $\calX$ be a CAT(0) simplicial complex with finitely many isometry types of simplices, and let $C$ be a convex 
subcomplex of $\calX$. Then there exists a coarsely Lipschitz retraction  from $\calX-\No{C}$ to $\partial N(C)$. In particular, $\partial N(C)$ is quasi-isometrically embedded in $\calX- \No{C}$. 
\end{lemma}

\begin{proof}
	We construct a projection from $\calX-\No{C}$ to $\partial N(C)$ as follows. Since the CAT(0) space $\calX$ has finitely 
many isometry types of simplices, we can choose a number $\varepsilon>0$ such that the closed metric neighbourhood $\calN(C, \varepsilon)$ (for the CAT(0) 
metric) is contained in $N(C)$. Since $C$, whence $\calN(C, \varepsilon)$, is convex for the CAT(0) metric, the closest-point projection $\pi: \calX-\No{C} 
\rightarrow \calN(C, \varepsilon)$ is well-defined and $1$-Lipschitz. Given a point $x\in \calX - \No{C}$, we  first compute the 
projection $\pi(x)$, we choose a simplex $\Delta$ of $N(C)$ containing $\pi(x)$ (such a simplex cannot be contained in $C$), 
and we define $\pi'(x)$ to be any vertex of $\Delta$ not in $C$. This  defines the map $\pi': \calX-\No{C}\rightarrow\partial 
N(C)$.
	
	Let $x, y \in \calX-\No{C}$ be two points at distance at most $\varepsilon/3$ in $\calX-\No{C}$, and therefore in $\calX$. Their closest-point projections $\pi(x), 
\pi(y)$ are also at distance at most $\varepsilon/3$. Thus, the CAT(0) geodesic $\gamma$ between $\pi(x)$ and $\pi(y)$ is of length at 
most $\varepsilon/3$, and in particular $\gamma$ does not intersect $C$ since its endpoints are at distance $\varepsilon$ from $C$. Also, $\gamma$ is contained in the combinatorial neighbourhood $N(C)$ since it is even contained in the convex subspace $\calN(C, \varepsilon)$. By \cite[Corollary 7.30]{Bridson-Haefliger}, there is a uniform bound $k$ depending only on $\calX$ and $\varepsilon$ 
but not on $x, y$ such that $\gamma$ intersects at most $k$ simplices of $N(C)$ (none of them contained in $C$). This now 
implies that there is a path in $\partial N(C)$ connecting $\pi'(x)$ and $\pi'(y)$ consisting of a concatenation of at most $k$ paths each contained in a simplex of $\partial N(C)$ (indeed, note that simplices $\Delta_1, \Delta_2$ of $N(C)$ that intersect outside of $C$ also satisfy $\Delta_1\cap\Delta_2\cap \partial N(C)\neq \emptyset$).
 Since $\calX$ has 
finitely many isometry types of simplices, the $k$ paths can be taken to have uniformly bounded length, and this implies that the map $\pi': \calX-\No{C}\rightarrow\partial N(C)$ is coarsely Lipschitz.
	
	Moreover, again since $\calX$ has finitely many isometry types of simplices, a similar reasoning implies that there exists a constant 
$k'$ such that for any vertex $v$ of $\partial N(C)$, the CAT(0) geodesic between $v$ and $\pi(v)$ meets at most $k'$ 
simplices, which implies that $v$ and $\pi'(v)$ lie within bounded distance in $\calX-\No{C}$. Thus, $\pi'$ 
is a coarse retraction.
	\end{proof}

\begin{lemma}
	Let $\Delta$ be an edge of $X$ contained in the fibre of a vertex of $Y$. Then the augmented 
link $\link_X(\Delta)^{+W}$ is quasi-isometrically embedded in  $\big(X-\sat_X(\Delta)\big)^{+W}.$
	\label{lem:QI2}
\end{lemma}

\begin{proof} 
	If $\Gamma$ is a single edge, then we have  $\link_X(\Delta) = X-\sat_X(\Delta)$ by Lemma~\ref{lem:coarseLipSat}, hence $\link_X(\Delta)^{+W} = \big(X-\sat_X(\Delta)\big)^{+W}$, and the inclusion is a quasi-isometric embedding.
	
	Let us now consider the case where $\Gamma$ is not reduced to a single edge. The edge $\Delta$ is contained in the fibre over a vertex that we call $u \in Y$, and we set $v \coloneqq \iota(u)\in\widehat{D}$.
	 It follows from Lemma~\ref{lem:coarseLipSat} that  $X-\sat_X(\Delta)$ is the full subcomplex of $X$ obtained by removing the fibre $X_{u}$. Recall also from the same lemma the map $\bar q:(X-\sat(\Delta))^{+W}\to \widehat{D}-\{v\}$. Also, Lemma \ref{lem: QI link} gives us a quasi-isometry $q':\link_X(\Delta)^{+W}\to \link_{\widehat{D}}(v)$ (note that, in the statement of the lemma, the map $q$ is restricted to the $0$-skeleton of $\link_X(\Delta)^{+W}$, and $q'$ here is an extension of $q$ across all edges).

	Consider now the following diagram: 
	
	\[	\xymatrix{
		\link_X(\Delta)^{+W} \ar@{^{(}->}[d] \ar[r]^{q'} & \partial N(v)  \ar@{^{(}->}[d] \\
		\big(X-\sat_X(\Delta)\big)^{+W} \ar[r]_{\ \ \ \ \ \ \ \bar q}  & \widehat{D} - \{v\}  } \]

Note that the diagram commutes at the level of vertices, so that it commutes up to bounded arrow since all maps are coarsely Lipschitz. 

By Lemma \ref{lem:coarse retraction}, there exists a coarsely 
Lipschitz retraction $\pi: \widehat{D} - \{v\} \rightarrow \partial N(v)$. Thus,  the inclusion $\partial 
N(v)\hookrightarrow \widehat{D} - \{v\} $ is a quasi-isometric embedding.  Since the top and right arrows are 
quasi-isometric embeddings, so is their composition.  Since the left arrow is coarsely Lipschitz, and so is the bottom arrow, the left 
arrow is a quasi-isometric embedding because the diagram commutes up to bounded error.  In other words, $\link_X(\Delta)^{+W}\hookrightarrow 
\big(X-\sat_X(\Delta)\big)^{+W}$ is also a quasi-isometric embedding.
\end{proof}

Before considering the final two cases where $\Delta$ is a triangle of $X$, we prove the following result that will allow us to show that the bottom map of our (almost) commutative diagram is coarsely Lipschitz. Recall that we have simplicial maps $p:X\to Y$, and $\iota: Y\to\widehat{D}$, and we set $q=\iota\circ p$.

\begin{lemma} \label{lem:coarseLipSatTriangle}
	Let $\Delta$ be a triangle of $X$ that contains an edge in the fibre of a vertex of $Y$ and an additional root $u$. Then we have
	$$X-\sat_X(\Delta) =  p^{-1}\big(Y- \Star_{Y}(u))\big) \sqcup  \Lambda_{u}.$$
	Moreover, let $Z=Z_u$ be the graph 
	obtained from the disjoint union $$\big(\widehat{D} - \Star_{\widehat{D}}(\iota(u))\big) \sqcup \Lambda_u^{+W}$$ as follows: We add an edge between a vertex $\lambda 
	\in \Lambda_u$ and a vertex $v\in \widehat{D} -\Star_{\widehat{D}}(\iota(u))$ if there exists a $W$-edge of $X^{+W}$ between $\lambda$ and a 
	vertex of $X_{v}$. 
	
	Then the restriction 
	$$q: X-\sat_X(\Delta) \rightarrow Z $$ 
	is well-defined and extends to a coarsely Lipschitz map 
	$$q': (X-\sat_X(\Delta))^{+W}\rightarrow Z.$$ 
	
\end{lemma}

\begin{proof}
	It follows from Lemma \ref{lem:join decomposition} that 
	$\sat_X(\Delta)$ consists of all the triangles $\Delta'$ of $X$ such that $(\Delta')_{u} = \Delta_{u}$ (that is, $\Delta'$ is the join of the root $u$ and some edge). The structure of $X-\sat_X(\Delta)$ follows immediately. Since $\iota$ is injective by Lemma \ref{lem:iota}, as well as simplicial, we have $\iota^{-1}(\Star_{\widehat{D}}(\iota(u))\subseteq \Star_Y(u)$, and therefore $q^{-1}(\Star_{\widehat{D}}(\iota(u))\subseteq p^{-1}(\Star_Y(u))$. Therefore, the restriction of $q$ as in the statement is well-defined.
	
	The required coarsely Lipschitz extension exists provided that any two distinct vertices  $v, v' \in X-\sat_X(\Delta)$ joined by an edge get mapped to vertices of $Z$ that lie within uniformly bounded (in particular, finite) distance.

	First notice that, by construction, edges that are not $W$-edges between two distinct vertices of $ X-\sat_X(\Delta)$ are either mapped to an edge of $Z$ or collapsed to a point. Moreover, if $v\in X-\sat_X(\Delta)-\Lambda_u$ and a point $\lambda \in \Lambda_{u}$ are joined by a $W$-edge, then their images under $q$ are joined by an edge in $Z$ by definition of $Z$. Thus, it is enough to show that there exists a uniform constant $C$ such that if $v, v' \in \big(X-\sat_X(\Delta)\big) - \Lambda_{u}$ are distinct vertices joined by a $W$-edge, then $\dist_Z(q(v), q(v'))\leq C$. 
	
	Let $v, v' \in \big(X-\sat_X(\Delta)\big) - \Lambda_{u}$ be two  distinct vertices joined by a $W$-edge. First notice that $p(v), p(v')$ belong to $Y - \Star_Y(u)$ by the first part of the statement. It is enough to consider the case where $p(v)$ and $p(v')$ are not connected by an edge in $Y$. Since $p(v)$ and $p(v')$ are connected by a $W$-edge, they are either equal (for $W$-edges of type $1$) or at distance $2$ in $Y$ (for $W$-edges of type $2$). It is enough to consider the second case, so let $\delta$, $\delta'$ be maximal simplices of $X$, containing $v, v'$ respectively, that are connected by a $W$-edge of type 2. 
	
	Let $u'$ be the midpoint between $p(v)$ and $p(v')$. If $u'\in Y - \Star_Y(u)$, then $q(v), \iota(u'), q(v')$ is a combinatorial path of length $2$ in $Z$, so $\dist_Z(q(v), q(v'))\leq 2$. 
	
	Otherwise, $u'$ must be a vertex  of $\link_Y(u)$ (since $u'\neq u$ for otherwise $v$ and $v'$ would lie in $\sat(\Delta)$). Up to the action of $A_\Gamma$, we can assume that $u, u'$ are of the form $u_a, u_{ab}$ (up to permutation) and in particular there are finitely many possible pairs to consider. Hence, it suffices to fix $u,u'$ and produce a constant for those.
	By construction, we have that  $w(\delta), w(\delta')$ are contained in the ``thickened'' coset $N(z_{u'})^{+r}$. Moreover, since $\delta$ and $\delta'$ are connected by a $W$-edge, it follows from  Lemma~\ref{lem:close_levels} that $\dist_{A_\Gamma}(w(\delta), w(\delta'))\leq r_3$. 
	
	To show that $q(v), q(v')$ are at uniformly bounded distance in $Z$, it will be enough to show that they are at uniformly bounded distance in the full subgraph $Z'$ of $Z$ spanned by $(\link_{\widehat{D}}(\iota(u'))- \Star_{\widehat{D}}(\iota(u)))\cup \Lambda_{u}$.

	We modify the complex $Z'$ in an equivariant way to obtain a graph on which  $N(z_{u'})$ acts naturally. Let $Z''$ be the full subcomplex of $X^{+W}$ spanned by the quasi-lines $\Lambda_w^{+W}$ for $w\in \link_Y(u')$. By construction of the augmented link $\link_Y(u')^{+W}$, the projection $p: X^{+W}\rightarrow Y^{+W}$ induces a surjective simplical map $p: Z'' \rightarrow \link_Y(u')^{+W}$. 
	 Moreover, since the link $\link_Y(u')^{+W}$ is connected by Lemma~\ref{lem:links_connected} and the quasi-lines $\Lambda_w^{+W}$ (that is, the fibers of $p$ over the vertices) are connected by Lemma~\ref{lem:quasiline}, it follows that $Z''$ is connected. Since $\stabilizer_Y(u')=N(z_{u'})$ stabilises $\link_Y(u')^{+W}$ and preserves the family of quasi-lines $\Lambda_w$ by construction of the action $A_\Gamma \curvearrowright X^{+W}$, it follows that there is an induced action of $N(z_{u'})$ on the connected graph $Z''$.

	 In particular, for $x \in Z''$,  the orbit map
	 $$\omega_x: N(z_{u'})\rightarrow Z'', g \mapsto g \cdot x$$
	 is coarsely Lipschitz (with respect to any given word metric on $N(z_{u'})$), and the Lipschitz constant does not depend on $x\in Z''$ since the action is cobounded, see Lemmas \ref{lem:I_finite} and \ref{lem:blow_up_data_existence}.
	 
	 We can extend $\omega_x$ to a map $\hat\omega_x: N(z_{u'})^{+2r}\rightarrow Z''$ which is still $N(z_{u'})$-equivariant and coarsely Lipschitz when $N(z_{u'})^{+2r}$ is endowed with the metric inherited from $A_\Gamma$.
	 
	 Moreover, by construction of $Z$, the projection $q: X^{+W}\rightarrow \widehat{D}^{+W}$ induces a Lipschitz map $q: Z'' \rightarrow Z'$ obtained by collapsing each quasi-line $\Lambda_w^{+W}$ other than $\Lambda_u$ to the corresponding vertex $\iota(w)$. 
	Thus, the composition map 
	$N(z_{u'})\overset{\omega_x}{\longrightarrow} Z'' \overset{q}{\longrightarrow} Z'$ is coarsely Lipschitz, where the coarse Lipschitz  constant can be chosen to be uniform.

	Let us now bound above the distance $\dist_{Z'}(q(v), q(v'))$. After perturbing $w(\delta)$ up to distance $r$, we can assume that it is contained in $yN(z)\cap N(z_{u'})^{+2r}$, where $yN(z)=q(v)$ for some element $z$ belonging to the finite set of elements of the form $a$ or $z_{ab}$ (for some $a, b \in \Gamma$). Similarly, we perturb $w(\delta')$ and define $y',z'$ analogously for $\delta'$.	
	We pick as basepoint for the orbit map the point $x \coloneqq N(z)$. The  coarsely Lipschitz map $ q\circ \hat\omega_x$ sends $w(\delta)$ to a set (uniformly bounded by \ref{lem:basic_w_properties}) containing $yN(z) = q(v)$. Moreover, if we denote by $k$ a uniform upper bound on the distance in $ \link_{Y}(u')$ between vertices of the form $N(z_1)$ and $N(z_2)$, for $z_1, z_2$  of the form $a$ or $z_{ab}$ ($a, b \in \Gamma$), then  the map $ q\circ \hat\omega_x$ sends $w(\delta')$ to a set (again uniformly bounded by \ref{lem:basic_w_properties}) containing $y'N(z)$, which is at distance at most $k$ from $y'N(z') = q(v')$. Moreover, since $q\circ \hat\omega_x$ is coarsely Lipschitz and  $\dist_{A_{\Gamma}}(w(\delta), w(\delta'))\leq r_3+4r$ (recall the $r$-perturbation), it follows that there exists a constant $C'$ such that $\dist_{Z'}(q(v), q(v')) \leq C'$, hence $\dist_{Z}(q(v), q(v')) \leq C'$.  This constant $C'$ depends on $k,r,r_3$ and the coarse Lipschitz constant for $\hat\omega_x$ (in particular, it is independent of $v,v',u'$).
	  
	  We thus have that $\dist_{Z}(q(v), q(v'))$ is uniformly bounded above, and it follows that  the map $q'$ is coarsely Lipschitz.
\end{proof}

\begin{lemma}
	Let $\Delta$ be a triangle of $X$ that contains two roots, and an edge in the fibre of a vertex of tree type of 
$Y$. 
 Then the augmented link $\link_X(\Delta)^{+W}$ is quasi-isometrically embedded in the augmented complex 
$\big(X-\sat_X(\Delta)\big)^{+W}.$
	\label{lem:QI3}
\end{lemma}

\begin{proof} Up to the action of $A_\Gamma$, we can assume that there exist adjacent generators $a, b \in\Gamma$ such that 
$\overline{\Delta}$ is the edge between $u_{a}$ and $u_{ab}$. It follows from Lemma \ref{lem:join decomposition} that 
$\sat_X(\Delta)$ consists of all the triangles $\Delta'$ of $X$ such that $(\Delta')_{u_{ab}} = \Delta_{u_{ab}}$ (that is, $\Delta'$ is a join of $u_{ab}$ and some edge). Moreover, by Proposition \ref{prop: C=C0}, the augmented link $\link_X(\Delta)^{+W}$ consists of the quasi-line 
$\Lambda_{u_{ab}}$, equipped with its $W$--edges (this augmented link is still quasi-isometric to $\mathbb Z$ by 
Lemma~\ref{lem:quasiline}). 

Let $q': (X-\sat_X(\Delta))^{+W} \rightarrow Z$ be the coarsely Lipschitz map constructed in Lemma~\ref{lem:coarseLipSatTriangle}. 

We have the 
following  diagram that commutes up to bounded error (as it commutes exactly on vertices and all maps are coarsely Lipschitz): 
	
	\[		\xymatrixcolsep{6pc}\xymatrix{
	\link_X(\Delta)^{+W} \ar@{^{(}->}[d]  \ar[r]^{=}   & \Lambda_{u_{ab}}  \ar@{^{(}->}[d] \\
	\big(X-\sat_X(\Delta)\big)^{+W} \ar[r]_{q'}  & Z } \]
	 
	 In particular, to show that  $\link_X(\Delta)^{+W} \hookrightarrow \big(X-\sat_X(\Delta)\big)^{+W}$  is a 
quasi-isometric embedding, it is enough to show that the inclusion $\Lambda_{u_{ab}} \hookrightarrow Z$ is a quasi-isometric 
embedding.
By construction of the $W$-edges, the inclusion $\Lambda_{u_{ab}} \hookrightarrow Z$ is a quasi-isometric 
embedding provided the $\langle z_{ab} \rangle$-orbits are quasi-isometrically embedded in $Z$.

Let $Z'$ be the full subcomplex of $\widehat{D}$ generated by $Z\cup \{v_{ab}\}$. By construction of $Z$ (see Lemma~\ref{lem:coarseLipSatTriangle}), the subcomplex $Z'$ is obtained from $\widehat{D}$ by removing the apices of all the standard trees containing $v_{ab}$. Since  $Z'$ is obtained from $\widehat{D}$ by removing a family of open cones, it is still CAT(0) 
for the induced metric by Lemma~\ref{rem:subspace_coneoff}. In particular, it follows from Lemma \ref{lem:coarse retraction} that there is a coarsely Lipschitz 
retraction $$Z = Z' - \{v_{ab}\}  \rightarrow \link_{Z'}(v_{ab}).$$
	 Thus, it is enough to show that the $\langle z_{ab} \rangle$-orbits are quasi-isometrically embedded in 
$\link_{Z'}(v_{ab})$. But by construction of $Z'$, $\link_{Z'}(v_{ab})$ is $A_{ab}$-equivariantly isomorphic to 
$\link_D(v_{ab})$. Thus, it is enough to show that the $\langle z_{ab} \rangle$-orbits are quasi-isometrically embedded in 
$\link_D(v_{ab})$, where $D$ is the original Deligne complex. Since $A_{ab}$ is of large type, this now follows from Lemma~\ref{lem:Vaskou_QI_embedded_link} applied to $g=z_{ab}$.
	\end{proof}

\begin{lemma}
	Let $\Delta$ be a triangle of $X$ that contains two roots, and an edge in the fibre of a dihedral vertex of $Y$.  Then the augmented 
link $\link_X(\Delta)^{+W}$ is quasi-isometrically embedded in the augmented complex $\big(X-\sat_X(\Delta)\big)^{+W}.$
	\label{lem:QI4}
\end{lemma}

	The proof of this lemma is rather long, but follows a similar strategy of reduction. The rest of this section will be 
devoted to it.

Up to the action of the group, we can assume that the projection $\overline{\Delta}$ is an edge of $Y$ between the vertices 
$u_a$ and $u_{ab}$. Since the fibre of $\Delta_{u_a}$ is the apex of $X_{u_a}$, the saturation $\mbox{Sat}_X(\Delta)$ 
consists of the union of all the triangles of $X$ that have the same fibre at $u_{a}$ as $\Delta$. In particular, we have 
$$X-\mbox{Sat}(\Delta) = \Lambda_{u_a} \sqcup q^{-1}\big(\widehat{D} - \widehat{T}_a\big),$$
and $\big(X-\mbox{Sat}(\Delta) \big)^{+W}$ is obtained from the previous graph by adding all the $W$-edges whose endpoints 
are in $X-\mbox{Sat}(\Delta) $. 

\medskip
	
\noindent 	\textbf{Step 1:} Let $q'=\pi_1: \big(X-\sat_X(\Delta)\big)^{+W} \rightarrow  Z$ be the coarsely Lipschitz map constructed in Lemma~\ref{lem:coarseLipSatTriangle}. The the following diagram commutes up to bounded error: 
		\[	\xymatrixcolsep{6pc}\xymatrix{
		\link_X(\Delta)^{+W} \ar@{^{(}->}[d]  \ar[r]^{=}   & \Lambda_{a}  \ar@{^{(}->}[d] \\
		\big(X-\sat_X(\Delta)\big)^{+W} \ar[r]_{\pi_1}  & Z } \]
	
	In particular, to show that the inclusion $$	\link_X(\Delta)^{+W}\hookrightarrow	
\big(X-\sat_X(\Delta)\big)^{+W}$$ is a quasi-isometric embedding, it is enough to show that the inclusion 
$$\Lambda_a\hookrightarrow Z$$ is a quasi-isometric embedding.

\medskip 
	
	We reduce the problem further, by using the CAT(0) geometry of $\widehat{D}$ to show that $Z$ coarsely retracts 
onto a neighbourhood of $\widehat{T}_a$. Consider the combinatorial neighbourhood $\partial N(\widehat{T}_a) = N(T_a) - v_a$.
	We define the complex $Z_a$ as the full subcomplex of $Z$ spanned by 
	$\Lambda_a \sqcup \partial N(\widehat{T}_a) $. \\
	
\noindent 	\textbf{Step 2:} There exists a coarsely Lipschitz retraction $\pi_2: Z \rightarrow Z_a$ such that the following 
diagram  commutes: 
		\[	\xymatrixcolsep{5pc}\xymatrix{
		\Lambda_a  \ar@{^{(}->}[d]  \ar[r]^{=}   & \Lambda_{a}  \ar@{^{(}->}[d] \\
	Z \ar[r]_{\pi_2}  & Z_a } \]
		In particular, to show that the inclusion $\Lambda_a \hookrightarrow Z$ is a quasi-isometric embedding, it is enough to 
	show that   the inclusion $\Lambda_a \hookrightarrow Z_a$ is a quasi-isometric embedding.
	
\begin{proof}[Proof of Step 2:]  Since $T_a$ is convex in $\widehat{D} - v_a$ by Lemma~\ref{lem:tree_convex_coneoff}, there exists a coarse retraction of $\widehat{D} - \widehat{T}_a$ on $\partial 
	N(\widehat{T}_a)$ by Lemma \ref{lem:coarse retraction}. We define the map $\pi_2: Z\rightarrow Z_a$ as follows: $\pi_2$ is the identity on $Z_a$ and is the coarse retraction $\widehat{D} - \widehat{T}_a \rightarrow \partial N(\widehat{T}_a)$ on $Z-\Lambda_a$. This definition makes sense as $Z_a \cap (Z-\Lambda_a) = \partial N(\widehat{T}_a)$ and $\pi_2$ is defined as the identity on $\partial N(\widehat{T}_a)$  in both cases. To conclude that this defines a coarse retraction $Z \rightarrow Z_a$, let us check that an edge between a vertex $x$ of $Z - \Lambda_a$ and a vertex $y$ of $Z_a$ is either contained in $Z - \Lambda_a$ or in $Z_a$. If $y \notin \Lambda_a$, then $x, y $ belong to $ Z - \Lambda_a$, and so does the edge between them. If $y \in \Lambda_a$, then by construction the edge $e$ between $x$ and $y$ comes from a $W$-edge of type $2$ of $X^{+W}$.  But by construction of $W$-edges of type $2$, this forces $x$ to belong to $\partial N(T_a)$, and thus $x, y$, and $e$ belong to $Z_a$. 
\end{proof}

	We now want to show that the inclusion  $\Lambda_a \hookrightarrow Z_a$ is a quasi-isometric embedding. 	For a vertex $v$ of $T_a$, we denote by $Z_v$ the subgraph of $Z_a$ spanned by $\Lambda_a$ and $\partial N(\widehat{T}_a)\cap 
N(v)$. In particular, $Z_a$ is the union of all the $Z_v$, where $v$ runs among the vertices of $T_a$. For a generator $b \in 
\Gamma$ adjacent to $a$, we simply denote by $Z_{ab}$ the graph $Z_{v_{ab}}$. To show that the inclusion  $\Lambda_a \hookrightarrow Z_a$ is a quasi-isometric embedding, we will first show that each inclusion $\Lambda_a \hookrightarrow Z_v$ is a quasi-isometric embedding (Steps 3 and 4) and then combine these quasi-isometric embeddings to show that $\Lambda_a \hookrightarrow Z_a$ is a quasi-isometric embedding (Step 5). 
	
	Before reducing the situation further, we need further information about  the graph of orbits introduced in Definition \ref{def:graph of orbits} and some variants. As the graphs we are looking at have been obtained from $\link_D(v_{ab})$ by removing the vertices corresponding to 
the tree $T_a$, we introduce a variant of the graph of orbits:
	
\begin{defn}
	Let $\calG$ be either  $\mathrm{Cayley}_{a,b}(A_{a,b})$ or $\mathrm{Cayley}_{a,b}(A_{a,b})/\langle \Delta_{ab}\rangle 
$ (where as usual, $\langle \Delta_{ab}\rangle$ acts by multiplication on the right). We define a new graph, called the 
\textbf{blown-up graph of orbits} and denoted $\mathrm{Orbit}_{a, b}^*\big(\calG\big)$ as follows.
	
	We denote by $L_a$ the simplicial line in $\calG$ spanned by the $\langle a \rangle$-orbit of the identity element. We put a 
vertex for every $\langle a \rangle$-orbit of $\calG$ except the one corresponding to $L_a$,  one vertex for every $\langle b 
\rangle$-orbit in $\calG$, and one vertex for each vertex of $L_a$. If two such subsets of $\calG$ have a non-empty 
intersection, we put an edge between them. Moreover, for every vertex $v$ of $L_a$, we also put an 
edge between $v$ and $av$. 
\end{defn}	
	
	This construction can be thought of as obtained from the original graph of orbits  $\mathrm{Orbit}_{a, 
b}^*\big(\calG\big)$ by blowing up the vertex corresponding to the coset $\langle a \rangle$, and replacing it with a copy of $L_a$. \\

	\noindent	\textbf{Step 3:}  There exists a coarsely Lipschitz map 
				$$ \pi_3: Z_{ab} \rightarrow  \mathrm{Orbit}_{a, b}^*\big(	
\mathrm{Cayley}_{a,b}(A_{a,b})\big)\ /\ \langle \Delta_{ab} \rangle $$ 
		and a quasi-isometry $\Lambda_a \rightarrow L_a$ such that the following diagram commutes up to bounded error and the inclusion $\Lambda_a \hookrightarrow Z_{ab}$ coarsely Lipschitz: 
		
				\[	\xymatrixcolsep{10pc}\xymatrix{
			\Lambda_a  \ar@{^{(}->}[d]  \ar[r]^{\cong}   & L_{a}  \ar@{^{(}->}[d] \\
			Z_{ab} \ar[r]_{\pi_3}  & \mathrm{Orbit}_{a, b}^*\big(	\mathrm{Cayley}_{a,b}(A_{a,b})\big) / \langle 
\Delta_{ab} \rangle  } \]

		In particular, to show that the inclusion $\Lambda_a \hookrightarrow Z_{ab}$ is a quasi-isometric embedding, it is 
enough to show that the inclusion 
		$$L_a \hookrightarrow\mathrm{Orbit}_{a, b}^*\big(	\mathrm{Cayley}_{a,b}(A_{a,b})\big) / \langle 
\Delta_{ab} \rangle $$
		is a quasi-isometric embedding. \\
		
\begin{proof}[Proof of Step 3]First notice that one can identify the subgraph $Z_{ab} - \Lambda_a$ with the subgraph $\link_{\widehat{D}}(v_{ab}) -  \link_{\widehat{D}\cap T_a}(v_{ab})$.

Now consider the (coarsely Lipschitz) quotient map $Z_{ab}\rightarrow Z_{ab}/\langle \Delta_{ab}\rangle$ for the action of $\langle \Delta_{ab}\rangle$ on the right. 
The image in that quotient of the apex of a standard tree is a vertex of valence $1$. The graph  $Z_{ab}/\langle 
\Delta_{ab}\rangle$ thus retracts via a Lipschitz map onto the subgraph obtained by removing all these valence $1$ 
vertices. Using the identification of $\link_D(v_{ab})$ with the graph of orbits  $\mathrm{Orbit}_{a, 
	b}\big(\mathrm{Cayley}(A_{ab})\big)$, one sees that the graph obtained is  the blown-up graph of orbits $\mathrm{Orbit}_{a, 
	b}^*\big(	\mathrm{Cayley}_{a,b}(A_{a,b})\big) / \langle \Delta_{ab} \rangle $.  	The composition $$\pi_3: Z_{ab} \rightarrow Z_{ab}/\langle \Delta_{ab}\rangle \rightarrow \mathrm{Orbit}_{a, 
	b}^*\big(	\mathrm{Cayley}_{a,b}(A_{a,b})\big) / \langle \Delta_{ab} \rangle $$ satisfies the required property. Moreover, $\pi_3$ is $\langle a \rangle$--equivariant. 

Pick a bounded set $B$ whose $\langle a \rangle$-translates cover $\Lambda_a$ containing exactly one element of each orbit. There is a unique $\langle a \rangle$-equivariant map $\Lambda_a \rightarrow L_a$ that sends $B$ to the identity element. This is a quasi-isometry since $\langle a \rangle$ acts coboundedly on both $\Lambda_a$ and $L_a$. We use this as the top arrow in the diagram. Since the other three maps are also $\langle a \rangle$-equivariant, the diagram commutes up to bounded error. 

Note that the map $\Lambda_a \hookrightarrow Z_{ab}$ is simplicial, hence coarsely Lipschitz.
\end{proof}

\noindent 	\textbf{Step 4:} The inclusion 
		$$L_a \hookrightarrow \mathrm{Orbit}_{a, b}^*\big(	\mathrm{Cayley}_{a,b}(A_{a,b})\big) / \langle 
\Delta_{ab} \rangle $$
is a quasi-isometric embedding. 

\begin{proof}[Proof of Step 4]Let us start with a preliminary remark. Starting from the Cayley graph of $A_{ab}$, there are two operations one can perform: 
We can mod out by the right-action of $\langle \Delta_{ab} \rangle$, or we can construct the aforementioned blown-up graph of 
orbits  $ \mathrm{Orbit}_{a, b}^*\big(	\mathrm{Cayley}_{a,b}(A_{a,b})\big)$. These two constructions commute. More precisely, there exists an isomorphism (in dotted lines below) that makes the following diagram commute: 

\[	\xymatrix{
	& \mathrm{Orbit}_{a, b}^*\big(	\mathrm{Cayley}_{a,b}(A_{a,b})/\langle \Delta_{ab}\rangle \big) \ar@{.>}[dd]^{\cong} 
\\  L_a \ar@{^{(}->}[rd] \ar@{^{(}->}[ru]  &  \\
	&  \mathrm{Orbit}_{a, b}^*\big(	\mathrm{Cayley}_{a,b}(A_{ab})\big) / \langle \Delta_{ab} \rangle  }\]

Thus, it is enough to show that the inclusion 
$$L_a \hookrightarrow  \mathrm{Orbit}_{a, b}^*\big(	\mathrm{Cayley}_{a,b}(A_{a,b})/\langle \Delta_{ab}\rangle \big) $$
is a quasi-isometric embedding.  The graph $\mathrm{Cayley}_{a,b}(A_{a,b})/\langle \Delta_{ab}\rangle$ (for the action of $\langle 
\Delta_{ab} \rangle$ on the right) is a tree, as it is obtained from the quasi-tree $\calT_{ab}$ defined in Section \ref{subsec: 
toolbox} by removing the edges of $\calT_{ab}$ labelled by an atom on more than one generator. Moreover, the various
$gL_a$ and $gL_b $ are  quasi-lines in the tree $\mathrm{Cayley}_{a,b}(A_{a,b})/\langle \Delta_{ab}\rangle$ that are $g\langle a \rangle g^{-1}$-invariant and $g\langle b \rangle g^{-1}$-invariant respectively. Moreover, for any $r \geq 0$ there is a constant $f(r)$ such that the  intersection of the $r$-neighbourhoods of any two such distinct quasi-lines has diameter at most $f(r)$. The fact that $L_a \hookrightarrow  \mathrm{Orbit}_{a, b}^*\big(	\mathrm{Cayley}_{a,b}(A_{a,b})/\langle \Delta_{ab}\rangle \big) $ is a quasi-isometric embedding is now a consequence of Step 3 and the following Claim, applied to the case $Z=  \mathrm{Cayley}_{a,b}(A_{a,b})/\langle \Delta_{ab}\rangle$, $L_{i_0} = L_a$,  and the $L_i$ are the $gL_b$ and the $gL_a$ other than $L_a$. 

\medskip

\noindent \textbf{Claim.} For all $\delta,\kappa$ and functions $f:\reals_+\to\reals_+$, there exists $D\geq 0$ such that the following holds.  Let $Z$ be a $\delta$--hyperbolic space and let $\{L_i\}$ be a family of $(\kappa,\kappa)$--quasilines in $Z$ such that for all distinct $i,j$ and all $r$, the intersection of the $r$--neighbourhoods of $L_i$ and $L_j$ has diameter at most $f(r)$. Then the following holds for all $L_{i_0}$.  Let $\widehat Z$ be obtained from $Z$ by coning off each $L_i,i\neq i_0$.  Then the image of $L_{i_0}$ under the inclusion $Z\to \widehat Z$ is a $(D,D)$--quasi-line.

\begin{proof}[Proof of Claim.]
The closest-point projection $proj: Z\to L_{i_0}$ has the property that the image of all $L_i, i\neq i_0$ is uniformly bounded. Therefore, $proj$ gives a coarsely Lipschitz coarse retraction of $\widehat Z$ onto the image of $L_{i_0}$ in $\widehat Z$. The existence of a coarse retraction implies that $L_{i_0}$ is quasi-isometrically embedded in $\widehat Z$, as required.
\end{proof}

This concludes the proof of Step~4.
\end{proof}
	
It follows from Steps 3 and 4 that the various inclusions $\Lambda_a \hookrightarrow Z_v$ are
quasi-isometric embeddings, for 
vertices $v $ of $T_a$. Note that because the action of $\mathrm{Stab}(T_a)$ on $T_a$ is cocompact by Lemma~\ref{lem:link_tree_cocompact}, we can choose a constant $K \geq 0$ such that all such inclusions $\Lambda_a \hookrightarrow Z_v$ are $(K,K)$-quasi-isometric embeddings. We now want to combine these various quasi-isometric embeddings into a global quasi-isometric embedding $\Lambda_a 
\hookrightarrow Z_a$.\\
	
\noindent 	\textbf{Step 5:} The inclusion 
	$\Lambda_a \hookrightarrow Z_a$
	is a quasi-isometric embedding. \\
	
\begin{proof}[Proof of Step 5] For every distinct vertices $v, w$ of $T_a$, notice that vertices in $Z_v \cap Z_w$ are at (combinatorial) distance at most $1$ from $\Lambda_a.$
	In particular, given a finite combinatorial geodesic $\gamma$ in $Z_a$ between two points $x, x'$ of $\Lambda_a$, there 
exist points $x_0\coloneqq x, x_1, \ldots, x_n\coloneqq x'$ of $\Lambda_a$, and distinct points $y_0\coloneqq x, y_1, \ldots, y_n 
\coloneqq x'$ of $\gamma\subset Z_a$ such that the portion of $\gamma$ between $y_i$ and $y_{i+1}$ is contained in some 
$Z_v$, and each $y_i$ is either equal or adjacent to $x_i$. From Steps 3 and 4 and the definition of the $x_i$ and $y_i$, it 
follows that for each $i$, we get the following inequality between combinatorial length: $$\dist_{\Lambda_a}(x_i, x_{i+1})\leq K\dist_{Z_v}(y_i, 
y_{i+1}) + K +2 \leq   (2K+2)\dist_{Z_v}(y_i,y_{i+1}),$$
	the last inequality following from the fact that $\dist_{Z_v}(y_i, 	y_{i+1})\geq 1$. We thus have 
	
	$$d_{\Lambda_a}(x,x')\leq \sum_i\dist_{\Lambda_a}(x_i,x_{i+1})\leq  (2K+2)\sum_i\dist_{Z_v}(y_i,y_{i+1})= 
	(2K+2)\dist_{Z_a}(x,x').$$

This shows that the inclusion 
	$\Lambda_a \hookrightarrow Z_a$
	is a quasi-isometric embedding.
	\end{proof}

	Combining all the previous steps, it follows that the inclusion $\link_X(\Delta)^{+W} \hookrightarrow \big( X - 
\mbox{Sat}(\Delta) \big)^{+W}$ is a quasi-isometric embedding, which concludes the proof of Lemma \ref{lem:QI4}.

\begin{proof}[Proof of Proposition~\ref{prop:QI}]
	The preceding sequence of Lemmas~\ref{lem:QI2}, \ref{lem:QI3}, \ref{lem:QI4},  together with Corollary~\ref{cor:unbounded link}, imply  the proposition.
\end{proof}

\subsection{Obtaining hierarchical hyperbolicity}\label{subsec:assembling} 
We are now ready to prove the main theorem of this article, Theorem \ref{thmintro:main}, which says that Artin groups of large and hyperbolic type are hierarchically hyperbolic. First, we just prove that a structure exists, and later on we give a more detailed description for the interested reader.

\begin{proof}[Proof of Theorem \ref{thmintro:main}]
There are three cases to consider.

\smallskip

\noindent \textbf{Single vertex:} If $\Gamma$ is a single vertex, then $A_\Gamma$ is isomorphic to $\mathbb Z$, which is an HHG.

\par\smallskip
	
\noindent \textbf{Connected case:}  We consider the case where $\Gamma$ is connected and contains at least two vertices, so that we are in the case covered by Convention \ref{conv:connected}. Let us first check that $(X, W)$ satisfy the conditions of 
Definition \ref{defn:combinatorial_HHS}: It satisfies (I) by Proposition \ref{prop:finite_complexity}, condition (II) by Propositions~\ref{prop:hyp_link} and~\ref{prop:QI}, condition (III) by 
Proposition~\ref{prop: C=C0}, and condition (IV) by Proposition~\ref{prop: links lattice} and Corollary~\ref{cor:unbounded link}. Moreover, the action of $A_\Gamma$ is proper and cobounded by 
Proposition~\ref{prop:W_qi_to_A_Gamma} and there are finitely many orbits of links by Proposition \ref{lem:index_set}. It thus follows from  Theorem~\ref{thm:main criterion}  that $A_\Gamma$ is 
hierarchically hyperbolic.

\par\smallskip

\noindent \textbf{General case:}  In general, $A_\Gamma$ is the free product of  Artin groups on connected graphs, which are hierarchically hyperbolic by the above.  Taking free products preserves 
the property of being a hierarchically hyperbolic group, by \cite[Corollary 8.24]{HHS_II} or~\cite[Theorem 9.1]{HHS_II}. 
\end{proof}

\paragraph{\textbf{Description of the HHG structure}}

We will now describe the HHG structure, and we introduce some notation to this end; we will use standard HHS terminology from \cite[Definition 1.1]{HHS_II}.

Given an HHS $(X,\mathfrak S)$, we can write $\mathfrak S=\mathfrak S_{fin} \sqcup \mathfrak S_{\infty} $, where $\mathfrak S_{\infty}$ is the set of indices  for which the corresponding hyperbolic space is unbounded. We typically have various bounded hyperbolic spaces in an HHS structure which need to be there because of the orthogonality axiom, but do not otherwise play a major role. The interesting part of the structure is $\mathfrak S_{\infty}$.

Recall from Definition \ref{defn:fancy_H} that given an Artin group  $A_\Gamma$ we denote
	$$\calH \coloneqq \big\{ \langle a \rangle ~|~ a \mbox{ a standard generator of } A_\Gamma \big\} \cup \big\{ \langle 
z_{ab} \rangle ~|~ a, b \mbox{ span an edge of } \Gamma  \mbox{ with } m_{ab}\geq 3 \big\}. $$ 

We denote two distinct copies of $\bigsqcup_{H\in\mathcal H}G/N(H)$ by $\mathfrak S^{line}$ and $\mathfrak S^{tree}$. We denote the copy of $gN(H)$ in $\mathfrak S^{line}$ by $gN(H)^{line}$, and similarly for $\mathfrak S^{tree}$.

Finally, we denote by $\mathfrak S^{leaf}$ the subset of $\mathfrak S^{tree}$ consisting of all cosets of the form $gN(\langle a\rangle)^{tree}$, where $a$ is a standard generator corresponding to a valence-1 vertex of $\Gamma$ whose incident edge has even label.

\begin{thm}\label{thm:artin_HHG}
Let $A_\Gamma$ be an Artin group of large and hyperbolic type, with $\Gamma$ connected and not a single vertex. Then there is an HHG structure $\mathfrak S$ on $A_\Gamma$ with the following
properties.  We have $\mathfrak S_\infty= \{S\}\sqcup \mathfrak S^{line}\sqcup \left(\mathfrak S^{tree}-\mathfrak S^{leaf}\right)$, and all of the following hold:
\begin{enumerate}

\item $S$ is the $\nest$-maximal element of $\mathfrak S$ and $\mathcal C(S)$ can be taken to be either the coned-off Deligne complex or the complex of irreducible parabolics of finite type.

\item For each $gN(H)\in \mathfrak S^{line}$, $\mathcal C(gN(H))$ is a quasiline on which $gN(H)g^{-1}$ acts, with $gHg^{-1}$ having unbounded orbits.

\item For each $gN(H)\in \mathfrak S^{tree}-\mathfrak S^{leaf}$, the space $\mathcal C(gN(H))$ is an unbounded quasitree. 

\item For each $gN(H)$, we have $gN(H)^{line}\orth gN(H)^{tree}$.

\item $gN(H)^{line}$ is nested in $hN(H')^{tree}$ if and only if $gHg^{-1}$ and $hH'h^{-1}$ are distinct and commute.

\item $gN(H)^{line}$ is orthogonal to $hN(H')^{line}$ if and only if $gHg^{-1}$ and $hH'h^{-1}$ are distinct and commute.

\item The items above completely determine the nesting and orthogonality relation within $\mathfrak S_{\infty}$.

\item No element of $\mathfrak S$ is orthogonal to any of its $A_\Gamma$-translates.
\end{enumerate}
\end{thm}

\begin{proof}
The HHG structure can be obtained from the description of the HHS structure for a combinatorial HHS \cite[Theorem 1.18, Remark 1.19]{CHHS}. The index set of the HHS structure is in bijection with the set of links of simplices, and the hyperbolic spaces in the HHS structure are the augmented links. The identification of the unbounded augmented links is done in Corollary \ref{cor:unbounded link}, with the (equivariant quasi-isometry types of the) augmented links being described in Lemma \ref{lem:maximal_Deligne} (see also Proposition \ref{prop:irred_parabolic}), \ref{lem:quasiline}, and Lemma \ref{lem:quasitree}.

Orthogonality and nesting of a combinatorial HHS are described in \cite[Definition 1.11]{CHHS}, and they correspond to links that form joins and links that are nested into each other. The link of the empty simplex will be the $\nest$--maximal element, and a straightforward case-by-case analysis with the other two types of links yields the relations are in our case (in view of how adjacency in the commutation graph is defined).

Finally, let us prove the last item; for this we argue with the HHS structure of the combinatorial HHS $(X,W)$ that we used throughout the section. Consider by contradiction non-maximal simplices $\Delta,h\Delta$ such that their links form a join. Consider any $x\in\link(\Delta)$; in particular, $x$ and $hx$ are adjacent in $X$. There are two cases. If $x=gN(H)$ for some $H\in\mathcal H$ and $g\in A_{\Gamma}$, then $x$ and $hx$ are also adjacent vertices of the commutation graph $Y$. However, $Y$ is bipartite (with respect to dihedral and tree type vertices, see Lemma \ref{lem:iota}) and no adjacent vertices are in the same orbit. On the other hand, if $x\in \Lambda_{gN(H)}$, then we can repeat the same argument with $gN(H)$.
\end{proof}

\begin{rem}
 The last item of Theorem \ref{thm:artin_HHG} implies that the HHG structure is $A_\Gamma$-\emph{colourable} in the sense of \cite[Definition 2.8]{DMS}.
\end{rem}

\begin{rem}
In the dihedral case, it does not suffice to simply observe that $A_\Gamma$ is virtually free-times-$\integers$ to show that it is a hierarchically hyperbolic group.  Indeed, while the property of being a hierarchically hyperbolic \emph{space} is 
preserved by quasi-isometries, the extra equivariance properties needed to be a hierarchically hyperbolic \emph{group} fail badly to be preserved even by commensurability.  Indeed, while 
$\integers^2$ is a hierarchically hyperbolic group, there are virtually $\integers^2$ groups, like the $(3,3,3)$ triangle group, that are not~\cite[Corollary 4.5]{PetytSpriano}. 
\end{rem}

 We now turn to the consequence of hierarchical hyperbolicity stated in  Corollary~\ref{corintro}.  

\begin{proof}[Proof of Corollary~\ref{corintro}]
Assertion \ref{item:asdim} follows directly from \cite[Thm. A]{HHS_III}.

All of the assertions follow from Theorem~\ref{thmintro:main} and general results about hierarchically hyperbolic groups, as mentioned in the statement, except  
(uniform exponential growth) and~\eqref{item:stable} (on stable subgroups).

We now handle assertion~\eqref{item:unif_exp} (uniform exponential growth).  In the case where $A_\Gamma$ is not dihedral, we have that $A_\Gamma$ is torsion-free by \cite{CharneyDavis} and acylindrically hyperbolic by \cite{MP}, so by 
Corollary~1.3 of~\cite{ANS}, $A_\Gamma$ has uniform exponential growth.  The remaining case is where $A_\Gamma$ is dihedral, i.e. $\Gamma$ is a single edge (and we are assuming that the label is not 2).  In this case, 
Lemma~\ref{lem:dihedral_split} implies that $A_\Gamma$ fits into an exact sequence $1\to\integers\to A_\Gamma\to F\to 1$, where $F$ is virtually a finite-rank non-abelian free group.  Hence 
$A_\Gamma$ has uniform exponential growth by~\cite[Proposition 2.3]{dlH}.  This proves assertion~\eqref{item:unif_exp}.

Next, we deal with assertion~\eqref{item:stable} (stable subgroups). Here, we use two additional facts about the hierarchically hyperbolic structure $\mathfrak S$. The first one is that the hyperbolic space associated to the $\nest$-maximal element is $A_\Gamma$--equivariantly quasi-isometric to the coned-off Deligne complex $\widehat D$, as stated in Theorem \ref{thm:artin_HHG}. Also, we will use the fact that for each non-$\nest$-maximal $V\in\mathfrak S$ with unbounded $\mathcal C V$ there exists $U\in\mathfrak S$ with $U\orth V$ and unbounded $\mathcal C U$. This readily follows from Theorem \ref{thm:artin_HHG}.

According to~\cite[Theorem B]{ABD}, there is an HHS structure on $A_\Gamma$ such that a subgroup $H\leq A_\Gamma$ is stable if and only if $H$--orbit maps to the $\nest$-maximal hyperbolic space are quasi-isometric embeddings. We observe below (Lemma \ref{lem:ABD}) that under our assumptions the $\nest$-maximal hyperbolic space constructed in \cite{ABD} is $A_\Gamma$-equivariantly quasi-isometric to $\widehat D$, which proves the assertion.
\end{proof}

In \cite{ABD}, Theorem B follows combining \cite[Corollary 6.2]{ABD}, which characterises stability for any HHS satisfying an additional condition (having \emph{unbounded products}), and \cite[Theorem 3.7]{ABD}, which says that says that one can change an HHG structure into one satisfying the additional condition.

What we used in the proof of Corollary \ref{corintro} is that this change does not affect the equivariant quasi-isometry type of the $\nest$-maximal hyperbolic space, as we now justify.

\begin{lemma}\label{lem:ABD}
 Let $(G,\mathfrak S)$ be an HHG with the property that for each non-$\nest$-maximal $V\in\mathfrak S$ with unbounded $\mathcal C V$ there exists $U\in\mathfrak S$ with $U\orth V$ and unbounded 
$\mathcal C U$. Then $G$ admits an HHG structure $(G,\mathfrak T)$ with unbounded products where the $\nest$-maximal hyperbolic space is $G$-equivariantly quasi-isometric to that of $(G,\mathfrak 
S)$.
\end{lemma}

\begin{proof}
 We claim that the HHG structure $(G,\mathfrak T)$ constructed in the second paragraph of the proof of \cite[Theorem 3.7]{ABD} has $\nest$-maximal hyperbolic space $\mathcal T_S$ with the required property; we now recall the construction. First, for a constant $M$ let $\mathfrak S^M$ be the set of all $U\in\mathfrak S$ for which there exist $U\nest V$ and $V\orth W$ with $\diam\ \mathcal CV>M$ and $\diam\ \mathcal CW>M$. We can choose $M$ such that $\diam\ \mathcal CV>M$ is equivalent to $\diam\ \mathcal CV=\infty$ (because the $G$-action on $\mathfrak S$ is cofinite by definition of HHG). Then $\mathcal T_S$ is defined as the factored space $\hat G_{\mathfrak S^M}$. As a set, this is $G$, however it is endowed with a different metric. Roughly, we start with $G$ and cone-off certain product regions corresponding to the elements of $\mathfrak S^M$, but we will not need the exact definition here.
 
 By \cite[Proposition 2.4]{HHS_III}, $\hat G_{\mathfrak S^M}$ is an HHS with index set $\mathfrak S-\mathfrak S^M$, and with the same associated hyperbolic spaces and projections as in $(G,\mathfrak S)$. By the distance formula for HHS \cite[Theorem 4.5]{HHS_II}, for any suitably large constant $L$ there are constants $K,C$ so that for all $x,y\in \hat G_{\mathfrak S^M}$ we have.
 
 $$d_{\hat G_{\mathfrak S^M}}(x,y)\approx_{K,C} \sum_{Y\in \mathfrak S-\mathfrak S^M} \left[\dist_{\mathcal C Y}(\pi_Y(x),\pi_Y(y))\right]_L,$$
 
 where $\approx_{K,C}$ denotes quantities that differ by multiplicative constant at most $K$ and additive constant at most $C$, and $[A]_L$ denotes $A$ if $A\geq L$ and $0$ otherwise.
 
 Now it is time to use our hypothesis. In our case $\mathfrak S^M$ contains all $U\in\mathfrak S$ with $\diam \mathcal C U>M$ except the $\nest$-maximal element $S$. If we take the threshold $L$ to be larger than $M$, the distance formula states that $\pi_S$ is a quasi-isometry from $\mathcal T_S=\hat G_{\mathfrak S^M}$ to $\mathcal C S$. Since $\pi_S$ is $G$-equivariant by definition of HHG, we are done. 
\end{proof}

\bibliographystyle{alpha}
\bibliography{artin}
\end{document}

%% file: quasitree.pdf_tex
\begingroup%
  \makeatletter%
  \providecommand\color[2][]{%
    \errmessage{(Inkscape) Color is used for the text in Inkscape, but the package 'color.sty' is not loaded}%
    \renewcommand\color[2][]{}%
  }%
  \providecommand\transparent[1]{%
    \errmessage{(Inkscape) Transparency is used (non-zero) for the text in Inkscape, but the package 'transparent.sty' is not loaded}%
    \renewcommand\transparent[1]{}%
  }%
  \providecommand\rotatebox[2]{#2}%
  \ifx\svgwidth\undefined%
    \setlength{\unitlength}{405.55278719bp}%
    \ifx\svgscale\undefined%
      \relax%
    \else%
      \setlength{\unitlength}{\unitlength * \real{\svgscale}}%
    \fi%
  \else%
    \setlength{\unitlength}{\svgwidth}%
  \fi%
  \global\let\svgwidth\undefined%
  \global\let\svgscale\undefined%
  \makeatother%
  \begin{picture}(1,0.70415926)%
    \put(0,0){\includegraphics[width=\unitlength,page=1]{quasitree.pdf}}%
    \put(0.43113659,0.40709987){\color[rgb]{0,0,0}\makebox(0,0)[lt]{\begin{minipage}{0.05645488\unitlength}\raggedright $e$\end{minipage}}}%
    \put(0.70363974,0.49721047){\color[rgb]{0,0,0}\makebox(0,0)[lt]{\begin{minipage}{0.05428355\unitlength}\raggedright $a$\end{minipage}}}%
    \put(0.20640289,0.52978056){\color[rgb]{0,0,0}\makebox(0,0)[lt]{\begin{minipage}{0.05428349\unitlength}\raggedright $b$\end{minipage}}}%
    \put(0.67324096,0.71000177){\color[rgb]{0,0,0}\makebox(0,0)[lt]{\begin{minipage}{0.13353741\unitlength}\raggedright $a\cdot a$\end{minipage}}}%
    \put(0.85454785,0.60577746){\color[rgb]{0,0,0}\makebox(0,0)[lt]{\begin{minipage}{0.14982239\unitlength}\raggedright $a\cdot ab$\end{minipage}}}%
    \put(0.20748858,0.69480239){\color[rgb]{0,0,0}\makebox(0,0)[lt]{\begin{minipage}{0.13679438\unitlength}\raggedright $b\cdot b$\end{minipage}}}%
    \put(-0.0020457,0.61771982){\color[rgb]{0,0,0}\makebox(0,0)[lt]{\begin{minipage}{0.16067913\unitlength}\raggedright $b\cdot ba$\end{minipage}}}%
    \put(0.84477682,0.15631011){\color[rgb]{0,0,0}\makebox(0,0)[lt]{\begin{minipage}{0.15742215\unitlength}\raggedright $ab\cdot b$\end{minipage}}}%
    \put(0.65804163,0.02820109){\color[rgb]{0,0,0}\makebox(0,0)[lt]{\begin{minipage}{0.17045016\unitlength}\raggedright $ab\cdot ba$\end{minipage}}}%
    \put(0.70472543,0.24533504){\color[rgb]{0,0,0}\makebox(0,0)[lt]{\begin{minipage}{0.07165422\unitlength}\raggedright $ab$\end{minipage}}}%
    \put(0.14560541,0.24642071){\color[rgb]{0,0,0}\makebox(0,0)[lt]{\begin{minipage}{0.07382556\unitlength}\raggedright $ba$\end{minipage}}}%
    \put(0.00012562,0.16282414){\color[rgb]{0,0,0}\makebox(0,0)[lt]{\begin{minipage}{0.14982239\unitlength}\raggedright $ba\cdot a$\end{minipage}}}%
    \put(0.20640292,0.04340045){\color[rgb]{0,0,0}\makebox(0,0)[lt]{\begin{minipage}{0.16936452\unitlength}\raggedright $ba\cdot ab$\end{minipage}}}%
  \end{picture}%
\endgroup%

%% file: Figure_coneoff2.pdf_tex
\begingroup%
  \makeatletter%
  \providecommand\color[2][]{%
    \errmessage{(Inkscape) Color is used for the text in Inkscape, but the package 'color.sty' is not loaded}%
    \renewcommand\color[2][]{}%
  }%
  \providecommand\transparent[1]{%
    \errmessage{(Inkscape) Transparency is used (non-zero) for the text in Inkscape, but the package 'transparent.sty' is not loaded}%
    \renewcommand\transparent[1]{}%
  }%
  \providecommand\rotatebox[2]{#2}%
  \newcommand*\fsize{\dimexpr\f@size pt\relax}%
  \newcommand*\lineheight[1]{\fontsize{\fsize}{#1\fsize}\selectfont}%
  \ifx\svgwidth\undefined%
    \setlength{\unitlength}{306.03754038bp}%
    \ifx\svgscale\undefined%
      \relax%
    \else%
      \setlength{\unitlength}{\unitlength * \real{\svgscale}}%
    \fi%
  \else%
    \setlength{\unitlength}{\svgwidth}%
  \fi%
  \global\let\svgwidth\undefined%
  \global\let\svgscale\undefined%
  \makeatother%
  \begin{picture}(1,0.5249266)%
    \lineheight{1}%
    \setlength\tabcolsep{0pt}%
    \put(0,0){\includegraphics[width=\unitlength,page=1]{Figure_coneoff2.pdf}}%
    \put(0.41879581,0.53052504){\color[rgb]{0,0,0}\makebox(0,0)[lt]{\begin{minipage}{0.12323417\unitlength}\raggedright $v_a$\end{minipage}}}%
    \put(0.39265524,0.1608225){\color[rgb]{0,0,0}\makebox(0,0)[lt]{\begin{minipage}{0.20072232\unitlength}\raggedright $v_{ab}$\end{minipage}}}%
    \put(0.5242918,0.28872463){\color[rgb]{0,0,0}\makebox(0,0)[lt]{\begin{minipage}{0.19418717\unitlength}\raggedright $v_{ac}$\end{minipage}}}%
    \put(0.70727584,0.17856074){\color[rgb]{0,0,0}\makebox(0,0)[lt]{\begin{minipage}{0.19605438\unitlength}\raggedright $v_{bc}$\end{minipage}}}%
    \put(0.22647581,0.16175608){\color[rgb]{0,0,0}\makebox(0,0)[lt]{\begin{minipage}{0.12883571\unitlength}\raggedright $T_a$\end{minipage}}}%
  \end{picture}%
\endgroup%

%% file: blowup2.pdf_tex
\begingroup%
  \makeatletter%
  \providecommand\color[2][]{%
    \errmessage{(Inkscape) Color is used for the text in Inkscape, but the package 'color.sty' is not loaded}%
    \renewcommand\color[2][]{}%
  }%
  \providecommand\transparent[1]{%
    \errmessage{(Inkscape) Transparency is used (non-zero) for the text in Inkscape, but the package 'transparent.sty' is not loaded}%
    \renewcommand\transparent[1]{}%
  }%
  \providecommand\rotatebox[2]{#2}%
  \newcommand*\fsize{\dimexpr\f@size pt\relax}%
  \newcommand*\lineheight[1]{\fontsize{\fsize}{#1\fsize}\selectfont}%
  \ifx\svgwidth\undefined%
    \setlength{\unitlength}{247.02801782bp}%
    \ifx\svgscale\undefined%
      \relax%
    \else%
      \setlength{\unitlength}{\unitlength * \real{\svgscale}}%
    \fi%
  \else%
    \setlength{\unitlength}{\svgwidth}%
  \fi%
  \global\let\svgwidth\undefined%
  \global\let\svgscale\undefined%
  \makeatother%
  \begin{picture}(1,0.65003618)%
    \lineheight{1}%
    \setlength\tabcolsep{0pt}%
    \put(0,0){\includegraphics[width=\unitlength,page=1]{blowup2.pdf}}%
    \put(0.45656362,0.23871443){\color[rgb]{0,0,0}\makebox(0,0)[lt]{\begin{minipage}{0.08178436\unitlength}\raggedright $p$\end{minipage}}}%
    \put(0.12332711,0.13892657){\color[rgb]{0,0,0}\makebox(0,0)[lt]{\begin{minipage}{0.09323493\unitlength}\raggedright $v$\end{minipage}}}%
    \put(0.67738168,0.14548017){\color[rgb]{0,0,0}\makebox(0,0)[lt]{\begin{minipage}{0.11166081\unitlength}\raggedright $v'$\end{minipage}}}%
    \put(0.81150825,0.51023882){\color[rgb]{0,0,0}\makebox(0,0)[lt]{\begin{minipage}{0.21885723\unitlength}\raggedright $X_{v'}$\end{minipage}}}%
    \put(-0.00306495,0.50533173){\color[rgb]{0,0,0}\makebox(0,0)[lt]{\begin{minipage}{0.15375478\unitlength}\raggedright $X_v$\end{minipage}}}%
    \put(0.39931455,0.65581513){\color[rgb]{0,0,0}\makebox(0,0)[lt]{\begin{minipage}{0.22081804\unitlength}\raggedright $X$\end{minipage}}}%
    \put(0.39113609,0.04406747){\color[rgb]{0,0,0}\makebox(0,0)[lt]{\begin{minipage}{0.12499501\unitlength}\raggedright $Y$\end{minipage}}}%
    \put(0,0){\includegraphics[width=\unitlength,page=2]{blowup2.pdf}}%
  \end{picture}%
\endgroup%